\documentclass[12pt]{amsart}
\usepackage{amssymb}
\usepackage{url}
\usepackage{graphicx}
\usepackage{tikz-cd}
\usepackage{mathtools}
\usepackage{mathrsfs}
\usepackage{comment}
\usepackage{stmaryrd}
\usepackage{appendix}

\usepackage[all]{xy}

\allowdisplaybreaks

\newtheorem{theorem}[equation]{Theorem}

\newtheorem{lemma}[equation]{Lemma}

\newtheorem{proposition}[equation]{Proposition}
\newtheorem{corollary}[equation]{Corollary}

\newtheorem*{theorem:derhamisomorphism}{Theorem~\ref{T:derhamisomorphism}}
\newtheorem*{theorem:characterization}{Theorem~\ref{T:characterization}}

\theoremstyle{definition}
\newtheorem{definition}[equation]{Definition}
\newtheorem{example}[equation]{Example}

\newtheorem{remark}[equation]{Remark}

\numberwithin{equation}{subsection}

\newcommand{\ZZ}{\mathbb{Z}}

\newcommand{\CC}{\mathbb{C}}

\newcommand{\QM}{\mathcal{Q}\mathcal{M}}

\DeclareMathOperator{\lcm}{lcm}

\DeclareMathOperator{\GL}{GL}
\DeclareMathOperator{\SL}{SL}

\DeclareMathOperator{\val}{val}

\DeclareMathOperator{\Tr}{Tr}
\DeclareMathOperator{\deriv}{der}
\DeclareMathOperator{\der}{d}

\DeclareMathOperator{\Id}{Id}

\DeclareMathOperator{\trdeg}{tr.deg}

\newcommand{\oK}{\mkern2.5mu\overline{\mkern-2.5mu K}}

\newcommand{\dnorm}[1]{\lVert #1 \rVert}
\newcommand{\inorm}[1]{{\lvert #1 \rvert}}

\begin{document}

\title[Nearly holomorphic Drinfeld modular forms]{Nearly holomorphic Drinfeld modular forms and their special values at CM points}

\author{Yen-Tsung Chen}
\address{Department of Mathematics, Pennsylvania State University, University Park, PA 16802, U.S.A.}

\email{ytchen.math@gmail.com}

	\author{O\u{g}uz Gezm\.{i}\c{s}}
\address{University of Heidelberg, IWR, Im Neuenheimer Feld 205, 69120, Heidelberg, Germany}
\email{oguz.gezmis@iwr.uni-heidelberg.de}



\date{September 4, 2023}

\keywords{Nearly holomorhic modular forms,  Drinfeld modular forms, Drinfeld modules}

\subjclass[2010]{Primary 11F52, 11R37; Secondary 11F03, 11G09}

\begin{abstract} In the present paper, we introduce the notion of nearly holomorphic Drinfeld modular forms and study an analogue of Maass-Shimura operators in this context. Furthermore, for a given nearly holomorphic Drinfeld modular form, we show that its special values at CM points are algebraically independent whenever the associated endomorphism algebras are distinct. As an application of our results on nearly holomorphic Drinfeld modular forms, we study Drinfeld quasi-modular forms for arbitrary congruence subgroups and investigate the structure of the vector spaces and the algebras generated by them.
\end{abstract}

\maketitle
\section{Introduction}
\subsection{Classical setting and motivation}
Let $\mathbb{H}$ be the upper half plane. \textit{A nearly holomorphic modular form of weight $k\in \mathbb{Z}_{\geq 0}$ and depth $r\in \mathbb{Z}_{\geq 0}$ for a congruence subgroup $\Gamma$ of $\SL_2(\mathbb{Z})$} is a smooth function $\mathfrak{f}:\mathbb{H}\to \mathbb{C}$ satisfying the following properties:
\begin{itemize}
	\item[(i)] For any $\gamma=\begin{pmatrix}
	a_{\gamma}&b_{\gamma}\\c_{\gamma}&d_{\gamma}
	\end{pmatrix}\in \Gamma$, we have
	\[
	\mathfrak{f}(\gamma \cdot z):=\mathfrak{f}\left(\frac{a_{\gamma}z+b_{\gamma}}{c_{\gamma}z+d_{\gamma}}\right)=(c_{\gamma}z+d_{\gamma})^k\mathfrak{f}(z).
	\]
	\item[(ii)] There exist holomorphic functions $\mathfrak{f}_0,\dots,\mathfrak{f}_r$ on $\mathbb{H}$ having a certain growth condition and uniquely defined by $\mathfrak{f}$, such that 
	\[
	\mathfrak{f}(z)=\sum_{j=0}^r\frac{\mathfrak{f}_j(z)}{(2\pi i (z-\bar{z}))^j}
	\]
	where $\bar{z}$ is the complex conjugate of $z$. 
	\item[(iii)] For any $\gamma\in \SL_2(\mathbb{Z})$, there exists a positive integer $n_{\gamma}$ such that for any $z\in \mathbb{H}$, we have 
	\[
	(\mathfrak{f}|_{k,m}\gamma)(z)=\sum_{i=0}^r\frac{1}{(2\pi i (z-\bar{z}))^i}\sum_{\ell= 0}^{\infty}a_{\gamma,i,\ell}e^{2\pi i \ell z/n_{\gamma}}, \ \ a_{\gamma,i,\ell}\in \mathbb{C}.
	\]
\end{itemize}
These interesting objects as well as the $\mathbb{C}$-vector spaces generated by them have been studied by Kaneko and Zagier \cite{KZ95} and Shimura \cite[\S 8, 12]{Shi07}, \cite{Shi87}, \cite[Chap. III]{Shi00} independently. As an explicit example, one can consider the non-holomorphic Eisenstein series $G_2$ of weight 2 and depth one, whose construction dates back to the work of Hecke and revisited later by Shimura \cite{Shi75a,Shi75b,Shi77}, given by 
\begin{equation}\label{E:Maass}
G_2(z):=-\frac{1}{4\pi i(z-\bar{z})}-\frac{1}{24}+\sum_{j=1}^{\infty}\big(\sum_{\substack{d>0\\ d|j}}d\big) e^{2\pi ijz}.
\end{equation}
Using the transcendence of the imaginary part function over the field of meromorphic functions on $\mathbb{H}$, it can be shown that any nearly holomorphic modular form may be uniquely expressed as a polynomial in $G_2$ with coefficients being (elliptic) modular forms for $\Gamma$ of certain weights. Moreover, nearly holomorphic modular forms are highly related to quasi-modular forms introduced by Kaneko and Zagier in \cite{KZ95} (see also \cite[Chap I. \S5.3]{BGHZ08}).

By generalizing the work of Maass, Shimura obtained an operator $D_k^r$, nowadays called \textit{a Maass-Shimura operator}, so that for any holomorphic function $g$ on $\mathbb{H}$, it is given by 
\[
D_k^rg(z):=\sum_{n=0}^r\binom{r}{n}\frac{\Gamma(k+r)}{\Gamma(k+n)}\frac{1}{(2\pi i(z-\bar{z}))^{r-n}}g^{(n)}(z)
\]
where $g^{(n)}$ is the $n$-th complex derivative of $g$. If $g$ is a  modular form of weight $k$ for $\Gamma$ then $D_k^rg$ is a nearly holomorphic modular form of weight $k+2r$ and depth $r$ for $\Gamma$. In addition to its arithmetic nature, when the Fourier expansion coefficients of $g$ lie in a cyclotomic field, special values of $D_k^rg$ at CM points have also remarkable properties described as follows:
\begin{enumerate}
		\item If  $h$ is another modular form of weight $k+2r$ which does not vanish at $z_0$ and has Fourier expansion coefficients lying in a cyclotomic field, then the ratio $\frac{D_k^rg(z_0)}{h(z_0)}$ generates an abelian extension of $L$ (see \cite[Main Thm. III]{Shi75a}).
	\item For a given CM point $z_0$ lying in an imaginary quadratic field $L$, there exists a period $\omega_{z_0}$ of a CM elliptic curve such that $D_k^rg(z_0)/\omega_{z_0}^{k+2r}\in\overline{\mathbb{Q}}$ (see \cite[Thm. 1.27]{Hid13}).
\end{enumerate}
For more details on these results as well as their generalization to Hilbert modular forms, the reader can consult the aforementioned references. 

The main motivation of the present paper is to investigate an analogy in the function field setting for the results of Shimura on nearly holomorphic modular forms given in (1) and (2) above as well as for their relation with quasi-modular forms.
\subsection{Nearly holomorphic Drinfeld modular forms} We start with setting up some notation. Let $\mathbb{F}_q$ be the finite field with $q$ elements where $q$ is a positive power of a prime $p$. We set $A$ to be the polynomial ring $\mathbb{F}_q[\theta]$ where $\theta$ is an indeterminate over $\mathbb{F}_q$, and $K$ to be the fraction field of $A$. We let $\inorm{\cdot}$ be the non-archimedean norm corresponding to the place at infinity normalized so that $\inorm{\theta}=q$. We denote by $K_{\infty}$ the completion of $K$ with respect to $\inorm{\cdot}$ which can be identified with the field of formal Laurent series $\mathbb{F}_q((1/\theta))$. We also set $\mathbb{C}_{\infty}$ to be the completion of a fixed algebraic closure of $K_{\infty}$ in $\mathbb{C}_{\infty}$ and $\overline{K}$ to be a fixed algebraic closure of $K$ in $\mathbb{C}_{\infty}$.

In the classical case, since $\mathbb{C}$ is of degree two over $\mathbb{R}$, one can decompose any element of $\mathbb{C}$ into its real and the imaginary part so that the latter induces a smooth but non-meromorphic function on $\mathbb{H}$ satisfying a certain transformation law under the action of $\SL_2(\mathbb{Z})$ on $\mathbb{H}$. However, since $\mathbb{C}_{\infty}$ is of infinite degree over $K_{\infty}$, one does not have an immediate analogy. Motivated by the work of Franc \cite{Fra11} who sets up a theory of nearly holomorphic modular forms over the $p$-adic numbers, we overcome this problem as follows: Consider a fixed algebraic closure $\overline{\mathbb{F}}_q$ of $\mathbb{F}_q$ and let $\widehat{K_{\infty}^{\text{nr}}}$ be the completion of the maximal unramified extension of $K_{\infty}$ which can be identified with $\overline{\mathbb{F}}_q((1/\theta))$. We define the Frobenius map $\sigma:\widehat{K_{\infty}^{\text{nr}}} \to \widehat{K_{\infty}^{\text{nr}}}$ by
\begin{equation}\label{E:sigmamap}
\sigma\Big(\sum_{i\geq i_0}a_i\theta^{-i}\Big):=\sum_{i\geq i_0}a_i^q\theta^{-i}, \ \ a_i\in \overline{\mathbb{F}}_q.
\end{equation}
Note that $\sigma$ is in fact a continuous field automorphism on $\widehat{K_{\infty}^{\text{nr}}}$.

Let $M$ be an extension of $\widehat{K_{\infty}^{\text{nr}}}$ and $\psi$ be an element in the set of continuous automorphisms of $\mathbb{C}_{\infty}$ fixing $K_{\infty}$. We call $\psi$ \textit{an extension of $\sigma$} if $\psi|_{\widehat{K_{\infty}^{\text{nr}}}}=\sigma$. We further let $M^{\psi}:=\{z\in M \ \ | \ \ \psi(z)=z\}$ and set $\Omega^{\psi}(M):=M\setminus M^{\psi}$. Using the methods developed in \cite[\S 3, 4]{Fra11}, we will see, in Theorem \ref{T:RA}, that $\psi$ is indeed a non-meromorphic function on $\Omega^{\psi}(M)$ (see \S2.2 for further details on non-meromorphic functions). Moreover, in Proposition \ref{Cl:QMF}, we will obtain \textit{an identity theorem for $\Omega^{\psi}(M)$} which allows us to determine the holomorphic functions on the rigid analytic space $\Omega:=\mathbb{C}_{\infty}\setminus K_{\infty}\supset \Omega^{\psi}(M)$ uniquely by its values on $\Omega^{\psi}(M)$.

In what follows, we provide the necessary background which will be used to determine the behavior of our functions around the cusps. Let $\exp_{C}$ be the exponential function of the Carlitz module and $\tilde{\pi}\in \mathbb{C}_{\infty}^{\times}$ be the Carlitz period which serves as an analogue of $2\pi i$ in our setting (see \S2.1 for details). For any non-zero element $N\in A$ and $z\in \Omega$, let $u_{N}(z):=\exp_{C}\left(\frac{\tilde{\pi} z}{N}\right)^{-1}\in \mathbb{C}_{\infty}^{\times}$. Set $\Gamma(1):=\GL_2(A)$ and let $\Gamma$ be a congruence subgroup of $\Gamma(1)$. There exists an ideal $\mathcal{I}$ of $A$ which is maximal among the ideals $\mathcal{J}$ satisfying 
\[
\left\{ \begin{pmatrix} 1& a\\
0 & 1
\end{pmatrix}\ \ | \ \ a\in \mathcal{J}           \right\}\subseteq \Gamma.
\]
We let $m_{\Gamma}\in A$ be the monic generator of $\mathcal{I}$. Furthermore, for any $z\in \Omega$, we set 
\[
\inorm{z}_i:=\inf\{ \inorm{z-a} \ \ : a\in K_{\infty}\}.
\]
A holomorphic function $g:\Omega\to \mathbb{C}_{\infty}$ is called \textit{bounded on vertical lines} if there exists real numbers $R, N>0$ such that for any $z\in \Omega$ satisfying $|z|_i>R$, we have $|f(z)|<N$. Assume that $f:\Omega\to \mathbb{C}_{\infty}$ is a holomorphic function and $f(z+a)=f(z)$ for all $a\in \mathcal{I}=(m_{\Gamma})$. Then, by \cite[Prop. 5.16]{BBP21}, whenever $\inorm{u_{m_{\Gamma}}(z)}$ is sufficiently small, there exists a unique infinite series expansion, called \textit{the $u_{m_{\Gamma}}$-expansion of $f$},
\[
f(z)=\sum_{n=0}^{\infty}a_nu_{m_{\Gamma}}(z)^n, \ \ a_n\in \mathbb{C}_{\infty}
\]
if and only if $f$ is bounded on vertical lines.

Throughout this paper, unless otherwise stated, we will fix a field $M$ over $\widehat{K_{\infty}^{\text{nr}}}$ and an extension $\psi$ of $\sigma$. We define \textit{a nearly holomorphic Drinfeld modular form of weight $k$, type $m\in \mathbb{Z}/(q-1)\mathbb{Z}$ and depth $r$ for $\Gamma$} as a continuous function $F:\Omega^{\psi}(M)\to \mathbb{C}_{\infty}$ satisfying the following properties:
\begin{itemize}
	\item[(i)] For any $\gamma=\begin{pmatrix}
	a_{\gamma}&b_{\gamma}\\c_{\gamma}&d_{\gamma}
	\end{pmatrix}\in \Gamma$, we have
	\[
	F|_{k,m}\gamma (z):=(c_{\gamma}z+d_{\gamma})^{-k}\det(\gamma)^{m}F(\gamma \cdot z):=(c_{\gamma}z+d_{\gamma})^{-k}\det(\gamma)^{m}F\left(\frac{a_{\gamma}z+b_{\gamma}}{c_{\gamma}z+d_{\gamma}}\right)=F(z).
	\]
	\item[(ii)] There exist holomorphic functions $f_0,\dots,f_r$ on $\Omega$ which are bounded on vertical lines and uniquely defined by $F$ satisfying  
	\begin{equation}\label{E:expansion}
	F(z)=\sum_{i=0}^r\frac{f_i(z)}{(\tilde{\pi}z-\tilde{\pi}\psi(z))^{i}},\ \ z\in \Omega^{\psi}(M).
	\end{equation}
	\item[(iii)] For any $\gamma\in \Gamma(1)$ and $z\in \Omega^{\psi}(M)$ such that  $|u_{m_{\gamma^{-1}\Gamma\gamma}}(z)|$ is sufficiently small, we have 
	\[
	(F|_{k,m}\gamma)(z)=\sum_{i=0}^r\frac{1}{(\tilde{\pi}z-\tilde{\pi}\psi(z))^{i}}\sum_{\ell= 0}^{\infty}a_{\gamma,i,\ell}u_{m_{\gamma^{-1}\Gamma\gamma}}(z)^{\ell}, \ \ a_{\gamma,i,\ell}\in \mathbb{C}_{\infty}.
	\]
\end{itemize}

We denote the $\mathbb{C}_{\infty}$-vector space of nearly holomorphic Drinfeld modular forms of weight $k$, type $m$ and depth at most $r$ for $\Gamma$ by $\mathcal{N}_{k}^{m,\leq r}(\Gamma)$. We further define $E_2:\Omega^{\psi}(M)\to \mathbb{C}_{\infty}$  by
\begin{equation}\label{E:efunc}
E_2(z):=E(z)-\frac{1}{\tilde{\pi}z-\tilde{\pi}\psi(z)}, 
\end{equation}
where $E$ is \textit{the false Eisenstein series of weight 2} introduced by Gekeler \cite[\S 8]{Gek88} (see \S3.1 for details). In Corollary \ref{C:1}, we will see that  $E_2\in \mathcal{N}_{2}^{1,\leq 1}(\Gamma(1))$.

Before stating our first theorem, we briefly discuss a fundamental class of holomorphic functions on $\Omega$. In 1980s, inspired by the work of Drinfeld \cite{Dri74}, Goss introduced the notion of \textit{Drinfeld modular forms} in his PhD thesis \cite{Gos80} which serves as an analogue of modular forms in our setting. Analytically, they are the holomorphic functions which satisfy a certain automorphy condition as well as a particular growth property at cusps (see \S2.2 for more details). On the other hand, algebraically, they are given by global sections of an ample invertible sheaf defined on the compactification of the Drinfeld moduli space. Later on, the theory has been developed over the years by  Gekeler in his series of works \cite{Gek84, Gek85, Gek86, Gek88, Gek89}, by  B\"{o}ckle in his habilitation thesis \cite{Boc02} and by Pellarin \cite{Pel21} studying Drinfeld modular forms in a more general setting, just to name a few.

Let $\mathcal{N}(\Gamma)$ be the $\mathbb{C}_{\infty}$-algebra generated by all the nearly holomorphic  Drinfeld modular forms for $\Gamma$. Our first result, restated later as Theorem \ref{Thm:Structure_of_N} and Corollary \ref{C:transE}, is as follows. 
\begin{theorem}\label{T:B} Any element $F\in \mathcal{N}_{k}^{m,\leq r}(\Gamma)$ may be uniquely expressed in the form
	\[
	F=\sum_{0\leq j\leq r}g_jE_2^j
	\]
	where each $g_j$ is a Drinfeld modular form of weight $k-2j$ and type $m-j$ for $\Gamma$. Moreover, the function $E_2$ is transcendental over the ring of Drinfeld modular forms for $\Gamma$. Furthermore, $\mathcal{N}(\Gamma)$ is a finitely generated $\mathbb{C}_{\infty}$-algebra.
\end{theorem}

\begin{remark} Combining Corollary \ref{C:1} and Theorem \ref{T:B}, we see that the function $E_2$ serves as an analogue of $G_2$ defined in \eqref{E:Maass} for our setting.
\end{remark}

\subsection{Maass-Shimura operators} Our next goal is to construct a differential operator sharing similar features with $D_k^r$. However, due to the  positive characteristic nature of our setting, a naive analogue of $D_k^r$ does not work. Indeed, the classical $n$-th derivation vanishes trivially when $n$ is divisible by $p$. Hence, we are obliged to use a different operator acting on the space of holomorphic functions, namely \textit{hyperderivatives} (see \S4 for the details). Regarding our aim, for any holomorphic function $f:\Omega\to \mathbb{C}_{\infty}$, we set $\delta_k^rf:=f$ if $r=0$, and for $r\geq 1$, we define
\[
\delta_k^rf:=\sum_{i=0}^r\binom{k+r-1}{i} \frac{\der^{r-i}f}{(\tilde{\pi}\Id-\tilde{\pi}\psi)^i},
\]
where $\der^{\ell}$ is a constant multiple of the $\ell$-th hyperderivative of $f$. One of the basic features of our operator, which we will prove in Proposition \ref{P:3}, is that for any $\gamma\in \Gamma(1)$, it satisfies
\[
\delta_k^r(f)|_{k+2r,m+r}\gamma=\delta_k^r(f|_{k,r}\gamma).
\]
A certain analysis on the behavior of holomorphic functions at cusps of $\Gamma$ and a generalization of the operator $\delta_{k}^r$ yield an action on $\mathcal{N}(\Gamma)$ (Proposition \ref{P:3}). Moreover, one can recover Rankin-Cohen brackets in the positive characteristic case introduced by Uchino and Satoh \cite{US98}, in terms of $\delta_k^r$ (Theorem \ref{Thm:Rankin_Cohen_Braket}). Furthermore, similar methods allow us to define a new sequence of operators acting on the space of Drinfeld modular forms (Theorem \ref{T:DO}). The proof of our results explained in this subsection mainly uses several combinatorial approaches which provide alternative proofs for their classical counterparts.

\subsection{Drinfeld quasi-modular forms} Motivated by the results in \cite{KZ95} and generalizing the work of Bosser and Pellarin  \cite{BP08}, we investigate \textit{Drinfeld quasi-modular forms} for any congruence subgroup $\Gamma \subseteq \Gamma(1)$. In our context, they are the holomorphic functions on $\Omega$ which are also holomorphic at infinity with respect to $\Gamma$ satisfying a certain automorphy condition. In particular, the false Eisenstein series $E$ can be given as an example of such forms (see \S5 for details). Moreover, each Drinfeld modular form can be considered as a Drinfeld quasi-modular form.

Let $\mathcal{Q}\mathcal{M}_{k}^{m,\leq r}(\Gamma)$ be the $\mathbb{C}_{\infty}$-vector space of Drinfeld quasi-modular forms of weight $k$, type $m$ and depth at most $r$ for $\Gamma$. We also set $\mathcal{Q}\mathcal{M}(\Gamma)$ to be the $\mathbb{C}_{\infty}$-algebra generated by all the Drinfeld quasi-modular forms for $\Gamma$. Our next result, restated in \S5 later, can be given as follows.
 
\begin{theorem}\label{T:C} We have $\mathcal{N}_{k}^{m,\leq r}(\Gamma)\cong \mathcal{Q}\mathcal{M}_{k}^{m,\leq r}(\Gamma)$ as $\CC_{\infty}$-vector spaces. In particular, any element $f\in \mathcal{Q}\mathcal{M}_{k}^{m,\leq r}(\Gamma)$ can be uniquely expressed in the form
	\[
	f=\sum_{0\leq j\leq r}g_jE^j
	\]
	where each $g_j$ is a Drinfeld modular form of weight $k-2j$ and type $m-j$ for $\Gamma$.  Moreover, $E$ is transcendental over the  ring of Drinfeld modular forms for $\Gamma$. Furthermore, $\mathcal{Q}\mathcal{M}(\Gamma)$ is a finitely generated $\mathbb{C}_{\infty}$-algebra that is stable under the hyperderivatives.
\end{theorem}
\begin{remark} We note that Theorem \ref{T:C} is a generalization of \cite[Thm. 1, Thm. 2]{BP08} for any congruence subgroup of $\Gamma(1)$. Our method is distinguished from the approach of Bosser and Pellarin due to implementing the theory of nearly holomorphic Drinfeld modular forms. 
\end{remark}

\subsection{Special values at CM points} Following the terminology in \cite{BCPW22}, we call a separable quadratic extension of $K$  \textit{a CM field} if the place at infinity does not split. Moreover, we call $z_0\in \Omega$ \textit{a CM point} if $z_0$ generates a CM field.  

For any $N\in A$, we call a finite extension $K_N$ of $K$ \textit{the $N$-th Carlitz cyclotomic field} if it is generated over $K$ by all $\exp_{C}(\tilde{\pi}a/N)\in \mathbb{C}_{\infty}^{\times}$ satisfying $a\in A$ with  $\deg_{\theta}(a)<\deg_{\theta}(N)$ (see \cite[\S12]{Ros02} for more details). Moreover, we let $L_N$ be a certain abelian extension of $L$ where $\infty$ splits completely (see \S6 for details) and  $L^{\text{ab}}$ be the maximal abelian extension of $L$ in $\mathbb{C}_{\infty}$. Our next result, which will be proven in \S6.2, can be seen as a function field analogue of (1) and is stated as follows.

\begin{theorem}\label{T:D} Let $f$ and $g$ be Drinfeld modular forms for $\Gamma$ of weight $k$ and $k+2r$ respectively such that their  $u_{m_{\Gamma}}$-expansion coefficients lie in a cyclotomic extension of $K$. Let $z_0\in \Omega$ be a CM point lying in $L$. Assume that $g(z_0)\neq 0$. Then there exists $N\in A$ such that 
	\[
	\frac{(\delta^r_{k}f)(z_0)}{g(z_0)}\in K_{N}\cdot L_{N}\subset L^{\text{ab}}.
	\]
\end{theorem}
\begin{remark}
To prove Theorem \ref{T:D}, motivated by the work of Shimura \cite[\S 12]{Shi07}, we analyze the value $(\delta^r_{k}f)(z_0)$ which can be given in terms of the value at $z_0$ of a certain Drinfeld modular form and an element in the function field of a Drinfeld modular curve of a particular level defined explicitly by Gekeler \cite[Chap. VII]{Gek86}. The main difference between the classical setting and our case is that although, classically, for a natural number $n$, the $n$-th cyclotomic field lies in the function field of the modular curve of level $n$ \cite[Thm. 6.6]{Shi71} defined over $\mathbb{Q}$, the function field of the Drinfeld modular curve of level $N$ over $K$ only contains \textit{the real subfield} of $K_{N}$ where $\infty$ splits completely \cite[Chap. VII, Thm. 1.9]{Gek86} (see \S6 for more details). 
\end{remark}

\begin{remark} Note that $\Omega^{\sigma}(\widehat{K_{\infty}^{\text{nr}}})$ only contains  CM points which are inert in a quadratic extension of $K$ (see Remark \ref{R:CM}). By introducing the space $\Omega^{\psi}(M)$ for a quadratic extension $M$ of $K_{\infty}$, we are able to study special values of nearly holomorphic Drinfeld modular forms also at ramified CM points.  In Appendix A, using Artin-Schreier extensions and Kummer theory, we will precisely construct the field $M$ including $z_0$ and the extension $\psi$ of $\sigma$ depending on the parity of the characteristic of $K_{\infty}$ so that the value of a nearly holomorphic Drinfeld modular form at $z_0$ is well-defined.
\end{remark}

Our final result, an analogue of (2) in our setting, concerns the transcendence properties of aforementioned special values (Theorem \ref{Thm:Rationality_at_CM_points} and Corollary \ref{Cor:Rationality_at_CM_points2}). The proof is based on Theorem \ref{T:B} and Theorem \ref{T:C} as well as an interpretation of values at CM points of certain Drinfeld modular forms in terms of a CM period (see \S6.2 and \S6.3 for details).
\begin{theorem}\label{T:E} Let $\alpha_1,\dots, \alpha_\ell$ be CM points in $\Omega$. Then for any Drinfeld modular form $f$ of weight $k$ for $\Gamma$ whose $u_{m_{\Gamma}}$-expansion coefficients lie in $\overline{K}$ and for each $1\leq i \leq \ell$, there exists a CM-period $\omega_{\alpha_i} \in\mathbb{C}_\infty^\times$ such that 
\[
    \delta_{k}^r(f)(\alpha_i)\in\overline{K}\cdot\left(\frac{\omega_{\alpha_i}}{\widetilde{\pi}}\right)^{k+2r}.
\]
Moreover, assume that $\delta_{k}^r(f)(\alpha_i)\neq 0$ for each $i$. Then, $\delta_{k}^r(f)(\alpha_1),\dots,\delta_{k}^r(f)(\alpha_{\ell})$ are algebraically independent over $\overline{K}$ if and only if the CM fields $K(\alpha_1),\dots,K(\alpha_\ell)$ are distinct.
\end{theorem}

\begin{remark} Although our setup for nearly holomorphic Drinfeld modular forms is inspired by the results of Franc \cite{Fra11}, there are certain aspects that distinguish our work from \cite{Fra11}. Firstly, unlike the $p$-adic setting, we have an explicit function $E_2\in \mathcal{N}_2^1(\Gamma(1))$ given in \eqref{E:efunc} which is not a Drinfeld modular form. Moreover, we are also able to study values of nearly holomorphic Drinfeld modular forms at arbitrary CM points although Franc requires the prime $p$ being inert in the imaginary quadratic extension where the CM point lies (see \cite[\S3.6]{Fra11} for further details). Furthermore, we also obtain the algebraic independence of special values as explained in Theorem \ref{T:E} which is more general than the $p$-adic analogue given in \cite[Thm. 5.3.1]{Fra11}.
\end{remark}

\begin{remark}
    We note that in the classical case, non-zero values of nearly holomorphic modular forms at CM points can be interpreted as Shimura's period symbols (see \cite[Eq.~32.3(b)]{Shi98} and \cite[Thm.~12.2]{Shi07}). In \cite{BCPW22}, Brownawell, Chang, Papanikolas and Wei introduced an analogue of Shimura's period symbol in the global function field setting. One of the key ingredients in the proof of Theorem~\ref{T:E} is to show that $\delta_{k}^r(f)(\alpha_i)$ is an algebraic multiple of $(\omega_{\alpha_i}/\widetilde{\pi})^{k+2r}$, where $\omega_{\alpha_i}$ is a period of a CM Drinfeld module. This also implies that $\delta_{k}^r(f)(\alpha_i)$ is either zero or a function field counterpart of Shimura's period symbol for a specific CM type in the sense of \cite{BCPW22} (see \cite[Remark~1.3.5]{BCPW22} for related discussions).
\end{remark}

\subsection{Outline of the paper} An outline of the present paper is given as follows. In \S2, after introducing Drinfeld modules, we discuss Drinfeld modular forms and emphasize several examples of them which will be later useful for our work. In \S3, we first aim to discuss particular properties of an extension of $\sigma$ and later on, provide necessary tools to obtain Theorem \ref{T:B}. In \S4, we analyze our Maass-Shimura operators and, by using their properties, we reobtain Rankin-Cohen brackets in Theorem \ref{Thm:Rankin_Cohen_Braket} and, analogous to the work in \cite[Sec. 16.1]{Shi07}, introduce a sequence of operators in \S4.3 preserving the modularity (Theorem \ref{T:DO}). In \S5,  we introduce Drinfeld quasi-modular forms for any given congruence subgroup of $\Gamma(1)$ and provide a proof for Theorem \ref{T:C}.  In \S6, we  prove Theorem \ref{T:D} and Theorem \ref{T:E} and finally in Appendix A, we study extensions of $\sigma$ for quadratic extensions of $\widehat{K_{\infty}^{\text{nr}}}$.

\subsection*{Acknowledgments} The authors are grateful to Gebhard B\"{o}ckle, Chieh-Yu Chang, Simon H\"{a}berli, Mihran Papikian and Jing Yu for many fruitful discussions and useful suggestions. We would like to thank Fu-Tsun Wei for his comments which lead authors to improve the results which appeared in a previous version. The authors are also indebted  Sriram Chinthalagiri Venkata to point out an inconsistency in an earlier version of the paper.   The second author acknowledges support by Deutsche Forschungsgemeinschaft (DFG) through CRC-TR 326 `Geometry and Arithmetic of Uniformized Structures', project number 444845124. The second author was also partially supported by National Center for Theoretical Sciences.

\section{Preliminaries and Background}
In this section, we provide necessary background, without giving proofs, for Drinfeld modules, the rigid analytic space $\Omega$ and Drinfeld modular forms. Our exposition is mainly based on \cite{Gek84}, \cite[Sec. 5]{Gek88}, \cite[Sec. 4]{Gos96},  \cite{GvdP80,SS91} and \cite[Chap. 3, 5]{Pap23}.
\subsection{Drinfeld modules} 
Let $K\subseteq L \subseteq \mathbb{C}_{\infty}$ be a field. We define the twisted polynomial ring $L[\tau]$ subject to the relation $\tau\alpha=\alpha^q\tau$ for all $\alpha\in L$. The ring $L[\tau]$ operates on $L$ by setting for $\varphi :=  b_0+b_1\tau+ \dots + b_r\tau^{\ell}  \in L[\tau]$ and $x \in L$,
\[
\varphi\cdot x:=\varphi(x) := b_0 x +  b_1x^q+ \dots + b_{\ell}x^{q^{\ell}}\in L.
\]
 
Let $r\in\mathbb{Z}_{>0}$. \textit{A Drinfeld module $\phi$ of rank $r$ defined over $L$} is an $\mathbb{F}_q$-algebra homomorphism $\phi:A\to L[\tau]$, which is uniquely determined by 
\[
\phi_\theta:=\theta+g_1\tau+\cdots+g_r\tau^r\in L[\tau],
\]
such that $g_r\neq 0$. For each Drinfeld module $\phi$, there exists a unique $\mathbb{F}_q$-linear power series $\exp_\phi(X)\in \mathbb{C}_\infty\llbracket X\rrbracket$, which is called \textit{the exponential function of $\phi$} such that $\exp_\phi(X)\equiv X~(\text{mod}~X^q\mathbb{C}_\infty\llbracket X\rrbracket)$ and 
\begin{equation}\label{E:funceqexp}
\phi_a(\exp_\phi(X))=\exp_\phi(aX), \ \ a\in A.
\end{equation} 
Moreover, it induces an entire, surjective and $\mathbb{F}_q$-linear function on $\mathbb{C}_\infty$. Note that $\Lambda_\phi:=\ker(\exp_\phi)\subset\mathbb{C}_\infty$ is a free $A$-module of rank $r$. Furthermore, $\Lambda_\phi$ is \emph{strongly discrete} in the sense that, for any finite radius open disc $B\subset\mathbb{C}_\infty$,  $\Lambda_\phi\cap B$ contains only finitely many elements. On the other hand, given a strongly discrete free $A$-module $\Lambda\subset\mathbb{C}_\infty$ of rank $r$, which we simply call \textit{an $A$-lattice of rank $r$}, we can associate an $\mathbb{F}_q$-linear power series 
\[
\exp_{\Lambda}(X):=X\prod_{0\neq\lambda\in\Lambda}\left(1-\frac{X}{\lambda}\right)\in\mathbb{C}_\infty\llbracket X\rrbracket.
\] 
For any $a\in A$, there exists an $\mathbb{F}_q$-algebra homomorphism $\phi^{\Lambda}:A\to\mathbb{C}_\infty[\tau]$ such that the following diagram commutes:
\begin{equation*}\label{Eq:Drinfeld_Modules_Uniformization}
\begin{tikzcd}
0 \arrow[r] & \Lambda \arrow[d, "a"] \arrow[r] & \mathbb{C}_\infty \arrow[d, "a"] \arrow[r, "\exp_{\Lambda}(\cdot)"] & \mathbb{C}_\infty \arrow[d, "\phi^{\Lambda}_a(\cdot)"] \arrow[r] & 0 \\
0 \arrow[r] & \Lambda \arrow[r] & \mathbb{C}_\infty \arrow[r, "\exp_{\Lambda}(\cdot)"] & \mathbb{C}_\infty \arrow[r] & 0.
\end{tikzcd}
\end{equation*}
In fact, due to Drinfeld \cite{Dri74}, the assignment $\Lambda\to\phi^{\Lambda}$ gives a bijection between the set of $A$-lattices of rank $r$ and Drinfeld modules of rank $r$ defined over $\mathbb{C}_\infty$ with the inverse map $\phi\to\Lambda_\phi$ so that $\exp_{\phi}=\exp_{\Lambda_{\phi}}$.

To illustrate the above correspondence, we define  \emph{the Carlitz module} $C:A\to \mathbb{C}_{\infty}[\tau]$ given by  $C_\theta:=\theta+\tau$, which is a rank one Drinfeld module defined over $K$. The corresponding rank one $A$-lattice $\Lambda_C$ is generated over $A$ by \textit{the Carlitz period} $\tilde{\pi}$ given by  
	\[
	\tilde{\pi}:=\theta (-\theta)^{1/(q-1)}\prod_{i=1}^{\infty} (1-\theta^{1-q^i})^{-1}\in \mathbb{C}_{\infty}^{\times},
	\]
	 where $(-\theta)^{1/(q-1)}$ is a fixed $(q-1)$-st root of  $-\theta$. 
	
\subsection{Drinfeld upper half plane} Recall the connected rigid analytic space $\Omega=\mathbb{C}_\infty\setminus K_\infty$. For any $n\in \mathbb{Z}_{\geq 1}$, we set the following affinoid subdomain
\[
\Omega_n:=\{z\in \Omega \ \ | \ \ \inorm{z}\leq q^{n} \text{ and } \inorm{z}_i\geq q^{-n}\}.
\]
Indeed, we have $\Omega_n\subset \Omega_{n+1}$ so that 
$
\Omega=\cup_{n=1}^{\infty}\Omega_n
$  (see \cite[\S1]{SS91}).

We call a function $f:\Omega\to \mathbb{C}_{\infty}$ \textit{holomorphic} if for each $n$, $f|_{\Omega_n}$ can be written as a uniform limit of rational functions whose poles lie outside of $\Omega_n$. We denote by $\mathcal{O}$ the structure sheaf on $\Omega$ and by abuse of notation, we also denote by $\mathcal{O}:=\mathcal{O}(\Omega)$, the ring of $\mathbb{C}_{\infty}$-valued holomorphic functions on $\Omega$. Furthermore, we say that $f$ is \textit{meromorphic} if it can be written as a ratio of two holomorphic functions. Otherwise, $f$ is called \textit{non-meromorphic}.

\subsection{Drinfeld modular forms} Our goal in this subsection is to give some details on Drinfeld modular forms. Firstly, for each $\gamma\in\GL_2(K_\infty)$ and $z\in\Omega$, we define
\[
\gamma\cdot z:=\frac{az+b}{cz+d}\in\Omega, \ \ \gamma=\begin{pmatrix}
a_{\gamma} & b_{\gamma}\\
c_{\gamma} & d_{\gamma}
\end{pmatrix}.
\]
We further set $j(\gamma;z):=cz+d$. For integers $m$ and $k$, the slash operator on a meromorphic function $f:\Omega\to\mathbb{C}_\infty$ is given by
\[
(f|_{k,m}\gamma)(z):=\det(\gamma)^mj(\gamma;z)^{-k}f(\gamma\cdot z).
\]

Recall that $\Gamma(1)=\GL_2(A)$. For $\mathfrak{n}\in A\setminus \mathbb{F}_q$, we define the \emph{principal congruence subgroup} of level $\mathfrak{n}$ by
\[
\Gamma(\mathfrak{n}):=\left\{\begin{pmatrix}
a & b\\
c & d
\end{pmatrix}\in\Gamma(1)\mid a\equiv d \equiv 1 \pmod{\mathfrak{n}}, \ \ b\equiv c \equiv 0 \pmod{\mathfrak{n}}\right\}.
\]
 We call a $\Gamma\leqslant \Gamma(1)$ \emph{a congruence subgroup} if $\Gamma(\mathfrak{n})\subseteq\Gamma$ for some $\mathfrak{n}\in A$. We say that an element $f\in \mathcal{O}$ is \textit{a weak Drinfeld modular form of weight $k$ and type $m\in \mathbb{Z}/(q-1)\mathbb{Z}$ for $\Gamma$ }if for each $\gamma \in \Gamma$, we have $f|_{k,m}\gamma=f$.

Recall the polynomial $m_{\Gamma}$ associated to $\Gamma$ defined in \S1.2 and $u_{m_{\Gamma}}=\exp_{C}\left(\frac{\tilde{\pi}z}{m_{\Gamma}}\right)^{-1}$. Any weak Drinfeld modular form $f$ for $\Gamma$ admits a unique $u_{m_{\Gamma}}$-expansion 
\[
f(z)=\sum_{n\geq n_0}a_nu_{m_{\Gamma}}(z)^n, \ \ a_n\in\mathbb{C}_\infty, \ \  n_0\in \mathbb{Z}
\]
for sufficiently small values of $\inorm{u_{m_{\Gamma}}(z)}$.  We further call $f$ \textit{meromorphic (holomorphic resp.) at infinity with respect to $\Gamma$} if $n_{0}\in \mathbb{Z}_{<0}$ ($n_0\in \mathbb{Z}_{\geq 0}$ resp.). 
\begin{definition}\label{Def:Modular_Forms}
	Let $k\in\mathbb{Z}$ and $m\in\mathbb{Z}/(q-1)\mathbb{Z}$. Consider a field $L\subseteq\mathbb{C}_\infty$ and a congruence subgroup $\Gamma\subseteq\Gamma(1)$.
	\begin{enumerate}
		\item[(i)] A weak Drinfeld modular form $f$ of weight $k$ and type $m$ for $\Gamma$ is called \emph{a Drinfeld modular form} if for each $\alpha\in\Gamma(1)$, the function $f|_{k,m}\alpha$ is holomorphic at infinity with respect to $\alpha^{-1}\Gamma\alpha$. We further call $f$ \textit{a Drinfeld cusp form} if, in the $u_{m_{\alpha^{-1}\Gamma\alpha}}$-expansion of $f|_{k,m}\alpha$, the smallest power of $u_{m_{\alpha^{-1}\Gamma\alpha}}(z)$ is positive. We let $\mathcal{M}(\Gamma)$ be the $\mathbb{C}_{\infty}$-algebra generated by all the Drinfeld modular forms for $\Gamma$. Moreover, we denote by $\mathcal{M}_k^m(\Gamma;L)$ the $L$-vector space generated by all Drinfeld modular forms of weight $k$ and type $m$ for $\Gamma$ whose $u_{m_{\Gamma}}$-expansion have coefficients lying in $L$. We further set 
		\[
		\mathcal{M}_k(L):=\bigcup_{\substack{\Gamma \text{ a congruence }\\\text{ subroup of $\Gamma(1)$}\\m\in \mathbb{Z}/(q-1)\mathbb{Z}}}\mathcal{M}_k^m(\Gamma;L).
		\]
		\item[(ii)]  A meromorphic function $f:\Omega\to \mathbb{C}_{\infty}$ is called a \emph{meromorphic Drinfeld modular form of weight $k$ and type $m$ for $\Gamma$ defined over $L$} if $f=g_1/g_2$ for some $g_1\in\mathcal{M}_{k+k'}^{m+m'}(\Gamma;L)$ and $g_2\in\mathcal{M}_{k'}^{m'}(\Gamma;L)$ with $k'\in \mathbb{Z}_{\geq 1}$ and $m'\in \mathbb{Z}/(q-1)\mathbb{Z}$. We denote by $\mathcal{A}_k^m(\Gamma;L)$ the $L$-vector space generated by meromorphic Drinfeld modular forms of weight $k$ and type $m$ for $\Gamma$ whose $u_{m_{\Gamma}}$-expansion have coefficients lying in $L$. Observe that $\mathcal{M}_k^m(\Gamma;L)\subset \mathcal{A}_k^m(\Gamma;L)$. We further set 
		\[
		\mathcal{A}_k(L):=\bigcup_{\substack{\Gamma \text{ a congruence }\\\text{ subroup of $\Gamma(1)$}\\m\in \mathbb{Z}/(q-1)\mathbb{Z}}}\mathcal{A}_k^m(\Gamma;L)
		\]
		and call any element $f\in\mathcal{A}_0(L)$  \emph{a Drinfeld modular function defined over $L$}.
	\end{enumerate} 
\end{definition}

In what follows, we give various examples of Drinfeld modular forms.
\begin{example}\label{Ex:Drinfeld_Modular_Forms}
	\begin{itemize}
		\item[(i)] Let $z\in\Omega$ and set $\Lambda_z:=Az+A$ which is an $A$-lattice of rank $2$. It corresponds to a rank $2$ Drinfeld module $\phi^{\Lambda_z}$ given by
		\[
		\phi^{\Lambda_z}_\theta:=\theta+\mathfrak{g}(z)\tau+\Delta(z)\tau^2.
		\]
		Then the function $\mathfrak{g}:\Omega\to\mathbb{C}_\infty$ lies in $\mathcal{M}_{q-1}^0(\Gamma(1);K)$. Moreover, $\Delta:\Omega\to \mathbb{C}_{\infty}$ is nowhere vanishing on $\Omega$ and a Drinfeld cusp form lying in $\mathcal{M}_{q^2-1}^0(\Gamma(1);K)$.
		\item[(ii)] Let $N\in A\setminus\mathbb{F}_q$ and $u=(u_1,u_2)\in (N^{-1}A/A)^2\setminus\{(0,0)\}$. We define \emph{the Eisenstein series of level $N$} by 
		\[
		E_{u}(z):=\sum_{\substack{(a,b)\in K^2\\ (a,b)\equiv u\pmod{A^2}}}\frac{1}{a\Tilde{\pi}z+\Tilde{\pi}b}=\frac{1}{\tilde{\pi}}\exp_{\phi^{\Lambda_z}}\left(\frac{u_1z+u_2}{N}\right)^{-1}, \ \ z\in \Omega.
		\]
		The function $E_u:\Omega\to \mathbb{C}_{\infty}$ is indeed a nowhere vanishing holomorphic function on $\Omega$. Moreover, $E_u\in \mathcal{M}_{1}^{0}(\Gamma(N);K_N)$ (see \cite[\S2]{Gek84} for more details).
		\item[(iii)] Let $u=(u_1,u_2)\in(\theta^{-1}A/A)^2$. We call $u$ \emph{monic} if either $u_1\in 1/\theta+A$ and $u_2\in A$ or $u_2\in 1/\theta+A$. We define the function $\mathfrak{h}:\Omega\to\mathbb{C}_\infty$, whose definition is due to Gekeler \cite[\S5]{Gek88}, by
		\[
		\mathfrak{h}(z):=\Tilde{\pi}^{\frac{1-q^2}{q-1}}(-\theta)^{1/(q-1)}\prod_{\substack{u\in(\theta^{-1}A/A)^2\\u~\text{monic}}}E_u(z),
		\]
		Moreover, $\mathfrak{h}$ is a Drinfeld cusp form lying in $\mathcal{M}_{q+1}^1(\Gamma(1);K)$.
	\end{itemize}
\end{example}

\section{Nearly holomorphic Drinfeld modular forms}
Let $M$ be a fixed extension  of $\widehat{K_{\infty}^{\text{nr}}}$ and let $\psi:M\to M$ be a fixed map in the set of continuous $K_{\infty}$-automorphisms of $\mathbb{C}_{\infty}$ such that $\psi|_{\widehat{K_{\infty}^{\text{nr}}}}=\sigma$ where $\sigma$ is given as in \eqref{E:sigmamap}. Our first goal in this section is to analyze properties of the map $\psi$. Such properties and their proofs are mainly based on the work of Franc in \cite{Fra11}. Later on,  we define nearly holomorphic Drinfeld modular forms, provide an explicit example in Corollary \ref{Cor:1} and investigate the $\mathbb{C}_{\infty}$-vector spaces generated by them.
\subsection{The extension $\psi$ of the Frobenius map $\sigma$} We start with proving our next lemma which reveals the behavior of meromorphic functions on $\Omega$ under a certain condition. Whereas it can be proven by using methods in \cite[Lem. 4.3.4]{Fra11}, we, instead, use a different argument and we are indebted to Gebhard B\"{o}ckle to share it with the authors.

\begin{lemma}[{cf.~Franc \cite[Lem. 4.3.4]{Fra11}}]\label{L:lin} Let $f$ be a meromorphic function on $\Omega$ satisfying 
\begin{equation}\label{E:der}
f(\alpha z)=\alpha^kf(z)
\end{equation}
for any $\alpha \in K_{\infty}^{\times}$ and for some $k\in \ZZ_{\geq 0}$. Then $f(z)=\beta z^k$ for some $\beta\in \CC_{\infty}$.  
\end{lemma}
\begin{proof} We pick an element $z_0\in \Omega$ which is neither a zero nor a pole of $f$. Consider the function $g:\Omega\to \mathbb{C}_{\infty}$ given by 
	\begin{equation}\label{E:gfunc}
	g(z):=\left(\frac{z}{z_0}\right)^kf(z_0).
	\end{equation}
	Clearly, $g\in \mathcal{O}$. On the other hand, by \eqref{E:der}, we have $g\equiv  f$ on the subset $z_0K_{\infty}^{\times}\subset \Omega$. Since $z_0$ is an accumulation point in $ z_0K_{\infty}^{\times}$ and $\Omega$ is connected \cite[Ex. 7.4.9]{FvdP04}, we have $f\equiv g $ on $\Omega$ and it finishes the proof of the lemma.
\end{proof}

Recall that $ M^{\psi}=\{z\in M \ \ | \ \ \psi(z)=z \}$ and  $\Omega^{\psi}(M)=M\setminus M^{\psi}$. We are now ready to prove the first main result of this section.
\begin{theorem}[{cf.~Franc \cite[Lem. 4.3.1]{Fra11}}]\label{T:RA} There exists no meromorphic function $f$ on $\Omega$ such that $f{|_{\Omega^{\psi}(M)}}=\psi$.
\end{theorem}
\begin{proof} Consider the affinoid subdomain
\begin{equation}\label{E:aff}
    \mathcal{A}:=\{z\in \CC_{\infty}|\ \ \inorm{z}\leq 1,\inorm{z-c}\geq 1,\forall c\in \mathbb{F}_q\}\subset \Omega.
\end{equation}
By \cite[Thm. 2.2.9]{FvdP04}, for any meromorphic function $f$ on $\Omega$ and $z\in \mathcal{A}$, we have $f(z)=g(z)u(z)$ where $g(z)\in \mathbb{C}_{\infty}(z)$ is a rational function and $u$ is a unit in the ring of holomorphic functions on $\mathcal{A}$. Hence, our argument reduces to showing that there exists no non-zero polynomial $h(z)\in \mathbb{C}_{\infty}[z]$ such that $h(z)f(z)$ is a holomorphic function on $\mathcal{A}$ with the property that 
\[
    h(z)f(z)|_{\mathcal{A}\cap M}=h(z)\psi(z).
\]

Assume to the contrary that there exists such a polynomial. We proceed by induction on the degree of $h(z)$. We begin with the case $\deg(h)=0$. Assume that $h(z)=\alpha\in \mathbb{C}_{\infty}^{\times}$. We set $\pi_{\infty}:=1/\theta$ and for any positive integer $n$, consider $a_n:=1+\pi_{\infty}^n\in K_\infty$. Then for any $z\in\mathcal{A}\cap M$, we have
\[
    h(a_nz)f(a_nz)=\alpha\psi(a_nz)=\alpha a_n\psi(z)=a_nh(z)f(z)
\]
as $\psi$ is a $K_{\infty}$-linear function. We pick an element $z_0\in\mathcal{A}\cap M$ which is neither a zero nor a pole of $f$. Consider $g\in \mathcal{O}$ given as in \eqref{E:gfunc}. Note that $g-hf$ vanishes at $a_nz_0$ for any positive integer $n$. Then by \cite[Thm. 2.2.9]{FvdP04} again we conclude that $g-hf$ vanishes identically on $\mathcal{A}$ and thus $f(z)=az$ for some $a\in\mathbb{C}_\infty$. In particular, for any $z\in \mathcal{A}\cap \widehat{K_{\infty}^{\text{nr}}}\subset \mathcal{A}\cap M$, we have $f(z)=\sigma(z)=az$ which leads to a contradiction.

Now let $\deg(h)=n$ and assume that the induction hypothesis holds for $\deg(h)<n$. Set 
\begin{equation}\label{E:fun1}
r(z):=h(z)f(z).
\end{equation} 
Then for any $z\in \mathcal{A}\cap M$ and $\alpha\in K_{\infty}^{\times}$, we have 
\begin{equation}\label{E:fun2}
\alpha h(\alpha z)\psi(z)=r(\alpha z).
\end{equation}
Multiplying both sides of \eqref{E:fun1} by $\alpha^{n+1}$ and subtracting from \eqref{E:fun2} imply
\[
(\alpha^{n+1}h(z)-\alpha h(\alpha z))\psi(z)=\alpha^{n+1}r(z)-r(\alpha z).
\]
Since the polynomial in the left hand side has degree strictly smaller than $n$, by the induction hypothesis, we have $\alpha^{n}h(z)=h(\alpha z)$ and $\alpha^{n+1}r(z)= r(\alpha z)$. By the same argument for the case $n=0$ again, we obtain $h(z)=\mathfrak{a}z^n$ and $r(z)=\mathfrak{b}z^{n+1}$ for some $\mathfrak{a},\mathfrak{b}\in \mathbb{C}_{\infty}$. By \eqref{E:fun1}, we thus have 
$
\mathfrak{b}z^{n+1}=\mathfrak{a}z^n\psi(z)
$
 for all $z\in \mathcal{A}\cap M$. But this implies that $\mathfrak{a}=0$ and hence $h\equiv 0$, which is a contradiction. Thus the proof is finished.
\end{proof}

Let $t_1,\dots,t_m$ be variables over $\mathbb{C}_{\infty}$. We define \textit{the Tate algebra} by the subring of power series in $\mathbb{C}_{\infty}[[t_1,\dots,t_m]]$ satisfying a particular condition:
\begin{multline}
\mathbb{T}_m:=\\
\left\{\sum_{(i_1,\dots,i_m)\in \mathbb{Z}^m_{\geq 0}} c_{(i_1,\dots,i_m)}t_1^{i_1}\dots t_m^{i_m}\in \mathbb{C}_{\infty}[[t_1,\dots,t_m]], \ \  \inorm{c_{(i_1,\dots,i_m)}}\to 0 \ \ \text{as } i_1+\dots +i_m\to \infty     \right\}.
\end{multline}

For any $g=\sum_{(i_1,\dots,i_m)\in \mathbb{Z}^m_{\geq 0}} c_{(i_1,\dots,i_m)}t_1^{i_1}\dots t_m^{i_m}\in \mathbb{T}_m$,  we define \textit{the Gauss norm} $\dnorm{g}$ of $g$ by 
\[
\dnorm{g}:=\max \{\inorm{c_{(i_1,\dots,i_m)}} \ \ | \ \ (i_1,\dots,i_m)\in \mathbb{Z}^m_{\geq 0}  \}.
\]
Note that $\mathbb{T}_m$, equipped with the Gauss norm, forms a Banach algebra. We call $B$ \textit{an affinoid algebra} if there exists an ideal $I$ of $\mathbb{T}_m$ so that $\mathbb{T}_m/I\cong B$ as $\mathbb{C}_{\infty}$-algebras. Therefore, $B$ forms a Banach algebra with respect to the quotient norm denoted by $\dnorm{\cdot}_{B}$. 
 
Let $C(\Omega^{\psi}(M),\CC_{\infty})$ be the ring of $\mathbb{C}_{\infty}$-valued continuous functions on $\Omega^{\psi}(M)$. To prove our next theorem, we need the following proposition.

\begin{proposition}[{cf.~Franc \cite[Prop. 3.4.4]{Fra11}}] \label{Cl:QMF}  The map $\mathcal{O}\to C(\Omega^{\psi}(M),\CC_{\infty})$ which is given by the restriction of any function $f\in \mathcal{O}$ to $\Omega^{\psi}(M)$ is injective.
\end{proposition}
\begin{proof} For each $n\in \mathbb{Z}_{\geq 1}$, consider $\mathcal{O}(\Omega_n)$ which is given explicitly in \cite[pg. 53]{SS91}. Let $h\in \mathcal{O}(\Omega_{n})$ and assume without loss of generality that $\dnorm{h}_{\mathcal{O}(\Omega_n)}=1$. Then, letting $\bar{h}$ be reduction of $h$ modulo the ideal $\mathfrak{M}\subset \mathbb{T}_m$ consisting of elements whose norm is less than 1, for any $z\in \overline{\mathbb{F}}_q\subset \Omega^{\psi}(M)\cap \Omega_n$, one can write $\bar{h}(z)$ as a rational function of $z$ (see also \cite[Chap. III (1.8.3)]{GvdP80}). Then there exists $z_0\in \overline{\mathbb{F}}_q$ such that $\inorm{h(z_0)}=1$ and this implies that
\[
\dnorm{h}_{\mathcal{O}(\Omega_n)}=\sup \{ \inorm{h(z)}\ \ | \ \ z\in \Omega^{\psi}(M)\cap \Omega_n \}.
\] 
Thus, we have the natural isometric inclusion $\mathcal{O}(\Omega_n)\xhookrightarrow{} C_{\text{bd}}(\Omega^{\psi}(M)\cap \Omega_n,\mathbb{C}_{\infty})$ where the set $C_{\text{bd}}(\Omega^{\psi}(M)\cap \Omega_n,\mathbb{C}_{\infty})$ of bounded and continuous functions on $\Omega^{\psi}(M)\cap \Omega_n$ is equipped with the sup-norm. Now, assume that there exists $f_1, f_2\in \mathcal{O}$ such that $f_1|_{\Omega^{\psi}(M)}=f_2|_{\Omega^{\psi}(M)}$. Then, for each $n\geq 1$,  $f_1|_{\Omega^{\psi}(M)\cap \Omega_n}=f_2|_{\Omega^{\psi}(M)\cap \Omega_n}$. By the above discussion, this implies that $f_1|_{\Omega_n}=f_2|_{\Omega_n}$. Since $\mathcal{O}$ is the structure sheaf of $\Omega$ and $\{\Omega_{n}\}_{n\geq 1}$ an affinoid covering of $\Omega$, we obtain $f_1=f_2$ as desired.
\end{proof}

Our next goal is to establish  an injective ring homomorphism which will be used to define nearly holomorphic Drinfeld modular forms in the next subsection.
\begin{theorem}[{cf.~Franc \cite[Prop.~4.3.3]{Fra11}}] \label{Thm:Transcendence_of_Sigma}
The assignment $X\to \frac{1}{\Id-\psi}$ yields an injective ring homomorphism $\mathcal{O}[X]\hookrightarrow C(\Omega^{\psi}(M),\CC_{\infty})$.
\end{theorem}

\begin{proof}
If $\sum_{i=0}^n\frac{f_i}{(\Id-\psi)^{i}}\equiv0 $ for some $f_0,\dots,f_n\in \mathcal{O}$, then we have
    \begin{align*}
        0&=\sum_{i=0}^nf_i(z)(z-\psi(z))^{n-i}=\sum_{j=0}^n\left(\sum_{i=0}^{n-j}(-1)^j\binom{n-i}{j}f_i(z)z^{n-i-j}\right)\psi(z)^j, \ \ z\in \widehat{K_{\infty}^{\text{nr}}}.
    \end{align*}
    Furthermore,  $\sum_{i=0}^{n-j}(-1)^j\binom{n-i}{j}f_i(z)z^{n-i-j}=0$ for $j=0,1,\dots,n$ if and only if $f_i(z)=0$ for $i=0,1,\dots,n$. Therefore, in what follows, we show that there exists no non-trivial relation  $\sum_{i=0}^n\mathfrak{g}_i\psi^{i}\equiv0 $ for meromorphic functions $\mathfrak{g}_0,\dots,\mathfrak{g}_{n}$ on $\Omega$, which is slightly stronger than we aim to prove, and, by the above discussion, it will imply the theorem. Assume to the contrary that such a relation exists. Since, the poles of a meromorphic function on $\Omega$ is discrete, by rescaling and  abuse of notation, we further assume that     
    \begin{equation}\label{Eq:Algebraic_Relations_1}
       \psi(z)^n+\mathfrak{g}_{n-1}(z)\psi(z)^{n-1}+\cdots+\mathfrak{g}_1(z)\psi(z)+\mathfrak{g}_0(z)=0,
    \end{equation}
    for any $z\in \Omega^{\psi}(M)$ outside of the set of the zeros of $\mathfrak{g}_n$. Suppose on the contrary that there exists $j\in \{0,\dots,n-1\}$ such that $\mathfrak{g}_j$ is not identically zero. We proceed by induction on $n$. The case $n=0$ follows from Proposition \ref{Cl:QMF} and the case $n=1$ is an application of Theorem~\ref{T:RA}. We may assume that $n\geq 2$. Replacing $z$ by $az$ for $a\in K_\infty^\times$ and using the fact that $\psi$ is $K_\infty$-linear, \eqref{Eq:Algebraic_Relations_1} now becomes
    \begin{equation}\label{Eq:Algebraic_Relations_2}
        a^n\psi(z)^n+a^{n-1}\mathfrak{g}_{n-1}(az)\psi(z)^{n-1}+\cdots+a\mathfrak{g}_1(az)\psi(z)+\mathfrak{g}_0(az)=0.
    \end{equation}
    Multiplying \eqref{Eq:Algebraic_Relations_1} by $a^n$ and subtracting \eqref{Eq:Algebraic_Relations_2} imply that
    \begin{equation}\label{Eq:Algebraic_Relations_3}
        \sum_{i=0}^{n-1}\left(a^n\mathfrak{g}_i(z)-a^i\mathfrak{g}_i(az)\right)\psi(z)^i=0.
    \end{equation}
    The induction hypothesis on $n$ implies that $\mathfrak{g}_i(az)=a^{n-i}\mathfrak{g}_i(z)$. Then Lemma~\ref{L:lin} shows that $\mathfrak{g}_i(z)=b_iz^{n-i}$ for some $b_i\in\mathbb{C}_\infty$. Substituting it in \eqref{Eq:Algebraic_Relations_1} yields
    \begin{equation}\label{Eq:Algebraic_Relations_4}
        \psi(z)^n+b_{n-1}z\psi(z)^{n-1}+\cdots+b_1z^{n-1}\psi(z)+b_0z^n=0.
    \end{equation}
    If $b_i$ is identically zero for all $0\leq i\leq n-1$, we are done. If not, then take any $z\in \overline{\mathbb{F}}_q\setminus \mathbb{F}_q$ in \eqref{Eq:Algebraic_Relations_4} to obtain
    \begin{equation}\label{Eq:Albegraic_Relations_6}
    z^{qn}+b_{n-1}z^{q(n-1)+1}+\cdots+b_1z^{n-1+q}+b_0z^n=0.
    \end{equation}
    Then, the polynomial in the left hand side of \eqref{Eq:Albegraic_Relations_6} would have infinitely many zeros and and it is a contradiction. This completes the proof.
\end{proof}

\subsection{Basic properties of nearly holomorphic Drinfeld modular forms} For any continuous function $f:\Omega^{\psi}(M)\to\mathbb{C}_\infty$, $\gamma\in\GL_2(K)$ and integers $m,~k$, we define the slash operator 
\[
(f|_{k,m}\gamma)(z):=\det(\gamma)^{m}j(\gamma;z)^{-k}f(\gamma\cdot z).
\]

We fix $\Gamma\leqslant  \Gamma(1)$ to be a congruence subgroup. For convenience, we restate the definition of nearly holomorphic Drinfeld modular forms. 

\begin{definition}\label{D:def} \textit{A nearly holomorphic Drinfeld modular form $F$ of weight $k\in \ZZ_{\geq 0}$, type $m\in \ZZ/(q-1)\ZZ$ and depth $r\in \ZZ_{\geq 0}$ for $\Gamma$} is an element in $C(\Omega^{\psi}(M),\mathbb{C}_{\infty})$ such that
		\begin{itemize}
			\item[(i)] $F|_{k,m}\gamma(z)=F(z)$ for any $\gamma\in \Gamma$.
			\item[(ii)] There exist $f_0,\dots,f_r\in \mathcal{O}$ which are bounded on vertical lines and uniquely defined by $F$ satisfying  
		\begin{equation}\label{E:UE}
		F(z)=\sum_{i=0}^r\frac{f_i(z)}{(\tilde{\pi}z-\tilde{\pi}\psi(z))^{i}},\ \ z\in \Omega^{\psi}(M).
		\end{equation}
		\item[(iii)] For any $\gamma\in \Gamma(1)$ and for any $z$ such that  $|u_{m_{\gamma^{-1}\Gamma\gamma}}(z)|$ is sufficiently small, we have 
		\begin{multline}\label{E:cusp}
		(F|_{k,m}\gamma)(z):=(c_{\gamma}z+d_{\gamma})^{-k}\det(\gamma)^{m}F(\gamma\cdot z)\\=\sum_{i=0}^r\frac{1}{(\tilde{\pi}z-\tilde{\pi}\psi(z))^{i}}\sum_{\ell= 0}^{\infty}a_{\gamma,i,\ell}u_{m_{\gamma^{-1}\Gamma\gamma}}(z)^{\ell}, \ \ a_{\gamma,i,\ell}\in \mathbb{C}_{\infty}.
		\end{multline}
		\end{itemize}
\end{definition}
\begin{remark} \label{R:1}
	\begin{itemize}
		\item[(i)] 	Note that Theorem \ref{Thm:Transcendence_of_Sigma} implies the uniqueness of $f_i$ given in \eqref{E:UE} for each $0\leq i \leq r$. 
		\item[(ii)] For any $\gamma\in \Gamma(1)$ and $0\leq i \leq r$, set
		\[
		\mathfrak{C}_{i,F}(z):=\sum_{\mu=i}^r(-1)^{\mu-i}\binom{\mu}{i}f_i|_{k-2i,m-i}\gamma(z)\left(\frac{c_{\gamma}}{c_{\gamma}z+d_{\gamma}}\right)^{\mu-i}.
		\]
		It is clear that $\mathfrak{C}_i\in \mathcal{O}$. Moreover, combining with the assumption \eqref{E:cusp}, we obtain 
		\[
	(F|_{k,m}\gamma)(z)=\sum_{i=0}^r\frac{1}{(\tilde{\pi}z-\tilde{\pi}\psi(z))^{i}}\mathfrak{C}_{i,F}(z)=\sum_{i=0}^r\frac{1}{(\tilde{\pi}z-\tilde{\pi}\psi(z))^{i}}\sum_{\ell= 0}^{\infty}a_{\gamma,i,\ell}u_{m_{\gamma^{-1}\Gamma\gamma}}(z)^{\ell},
		\]
		 whenever $|u_{m_{\gamma^{-1}\Gamma\gamma}}(z)|$ is sufficiently small.
		 Furthermore, note that if $\inorm{z}_i>R$ for some $R\in \mathbb{R}_{>0}$, then 
		 \[
		 \inorm{c_{\gamma}}R<\inorm{c_{\gamma}}\inorm{z}_i\leq \inorm{c_{\gamma}}\inorm{z+c_{\gamma}^{-1}d}=\inorm{c_{\gamma}z+d_{\gamma}}.
		 \]
		 This implies that $\inorm{c_{\gamma}/c_{\gamma}z+d_{\gamma}}<1/R$. Thus, using \cite[Prop. 5.16]{BBP21}, one can recursively obtain that $f_i|_{k-2i,m-i}\gamma$ is bounded on vertical lines.
	\end{itemize}
\end{remark}
 We denote the $\mathbb{C}_{\infty}$-vector space of nearly holomorphic Drinfeld modular forms of weight $k$, type $m$ and depth at most $r$ for $\Gamma$ by $\mathcal{N}_k^{m, \leq r}(\Gamma)$. For later use, we also consider the $\mathbb{C}_{\infty}$-vector space 
 $
 \mathcal{N}_{k}^{m}(\Gamma):=\bigcup_{r\in \mathbb{Z}_{\geq 0}}\mathcal{N}_{k}^{m,\leq r}(\Gamma).
 $

We continue with proving the following proposition which will be useful to provide an explicit element in $\mathcal{N}_2^{1, \leq 1}(\Gamma(1))$.

\begin{proposition}\label{P:1} Let $K_{\infty}\subseteq L\subseteq \mathbb{C}_{\infty}$ be a field  and $\psi:L\to L$ be a non-trival homomorphism of $K_{\infty}$-algebras and $L^\psi$ be the $K_{\infty}$-algebra of elements in $L$ fixed by $\psi$. Then for any $z\in L \setminus L^\psi$ and $\gamma\in \GL_2(K_{\infty})$, we have 
	\begin{equation}\label{E:sigmafunceq}
	\frac{1}{\gamma\cdot z-\psi( \gamma \cdot z)}=j(\gamma;z)^{2}\det(\gamma)^{-1}\Big(\frac{1}{z-\psi(z)}-\frac{c_{\gamma}}{j(\gamma;z)}\Big).
	\end{equation}
\end{proposition}
\begin{proof} If $\omega:=\begin{pmatrix}0 & 1\\
	1&0
	\end{pmatrix}$, then by using $\psi(z^{-1})=\psi(z)^{-1}$, we see that 
	\[
	\frac{1}{\omega\cdot z-\psi(\omega\cdot z)}=\frac{1}{z^{-1}-\psi(z^{-1})}=-\frac{z\psi(z)}{z-\psi(z)}
	\]
	which implies \eqref{E:sigmafunceq}. On the other hand, if $\gamma=\begin{pmatrix}a & b\\
	0&d
	\end{pmatrix}\in \GL_2(K_{\infty})$, by using $K_{\infty}$-linearity of $\psi$, we have
	\[
	\frac{1}{\frac{az+b}{d}-\psi\Big(\frac{az+b}{d}\Big)}=\frac{1}{\frac{az+b}{d}-\frac{b}{d}-\frac{a}{d}\psi(z)}=\frac{d}{a}\frac{1}{z-\psi(z)}=\det(\gamma)^{-1}\frac{d^2}{z-\psi(z)},	
	\]
	which implies \eqref{E:sigmafunceq}. Lastly, assume that \eqref{E:sigmafunceq} holds for $\gamma_1=\begin{pmatrix}a_1 & b_1\\
c_1&d_1
\end{pmatrix}\in \GL_2(K_{\infty})$ and for $\gamma_2=\begin{pmatrix}a_2 & b_2\\
c_2&d_2
\end{pmatrix}\in \GL_2(K_{\infty})$. We claim that \eqref{E:sigmafunceq} also holds for $\gamma_1\gamma_2$. To ease the notation, let us set $\mathcal{Z}:=(c_1a_2+d_1c_2)z+c_1b_2+d_1d_2$. Using the assumptions on $\gamma_1$ and $\gamma_2$, we have
	\begin{align*}
	&\frac{1}{\gamma_1\gamma_2\cdot z -\psi(\gamma_1\gamma_2\cdot z)}\\
	&=\det(\gamma_1)^{-1}\bigg(c_1\Big(\frac{a_2z+b_2}{c_2z+d_2}\Big)+d_1\bigg)\bigg(\frac{c_1\Big(\frac{a_2z+b_2}{c_2z+d_2}\Big)+d_1}{\gamma_2\cdot z-\psi(\gamma_2\cdot z)}-c_1\bigg)\\
	&=\det(\gamma_1)^{-1}\frac{\mathcal{Z}}{c_2z+d_2}\bigg(\frac{\mathcal{Z}}{c_2z+d_2}\Big(\det(\gamma_2)^{-1}(c_2z+d_2)\Big(\frac{c_2z+d_2}{z-\psi(z)}-c_2\Big)\Big)-c_1 \bigg)\\
	&=\det(\gamma_1)^{-1}\frac{\mathcal{Z}}{c_2z+d_2}\Big(\frac{\mathcal{Z}(c_2z+d_2)}{\det(\gamma_2)(z-\psi(z))}-\frac{\mathcal{Z}c_2-c_1\det(\gamma_2)}{\det(\gamma_2)}\Big)\\
	&=\det(\gamma_1)^{-1}\det(\gamma_2)^{-1}\mathcal{Z}\frac{\mathcal{Z}}{z-\psi(z)}-\det(\gamma_1)^{-1}\det(\gamma_2)^{-1}\frac{\mathcal{Z}}{c_2z+d_2}(c_2z+d_2)(c_1a_2+d_1c_2)\\
	&=\det(\gamma_1)^{-1}\det(\gamma_2)^{-1}\mathcal{Z}\Big(\frac{\mathcal{Z}}{z-\psi(z)}-(c_1a_2+d_1c_2)\Big)
	\end{align*} 
	which implies the claim. Since $\GL_2(K_{\infty})$ is generated by upper triangular matrices and $\omega$, we conclude the proof of the proposition by above computations.
\end{proof}

\begin{definition} Let $F$ be an element in $\mathcal{N}_{k}^{m,\leq r}(\Gamma)$ for some $k,r \in \mathbb{\ZZ}_{\geq 0}$, $m\in \mathbb{Z}/(q-1)\mathbb{Z}$ whose unique expansion is given as in \eqref{E:UE}. Let $m_{\Gamma}\in A$ be a monic polynomial given as in \S 1.2 for $\Gamma$ and let $K\subseteq L\subseteq\mathbb{C}_\infty$ be a field. We say that $F$ is \textit{defined over $L$} if the $u_{m_{\Gamma}}$-expansion of each $f_i$ has coefficients in $L$. We denote the $L$-vector space of nearly holomorphic Drinfeld modular forms of weight $k$, type $m$ and depth at most $r$ for $\Gamma$ defined over $L$ by $\mathcal{N}_k^{m, \leq r}(\Gamma;L)$. We also set
		\[
	\mathcal{N}_k(\Gamma;L)=\bigcup_{\substack{m\in \mathbb{Z}/(q-1)\mathbb{Z}\\r\in \mathbb{Z}_{\geq 0}}}\mathcal{N}_k^{m,\leq r}(\Gamma;L), \text{ and } \mathcal{N}_k(L):=\bigcup_{\substack{\Gamma \text{ a congruence }\\\text{ subroup of $\GL_2(K)$}\\ k\in \mathbb{\ZZ}_{\geq 0}}}\mathcal{N}_k(\Gamma;L).
	\]
\end{definition}

Let $E:\Omega\to \CC_{\infty}$ be \textit{the false Eisenstein series} introduced by Gekeler in \cite[Sec. 8]{Gek88} whose $u:=u_{1}$-expansion is given by 
\begin{equation}\label{E:uexp}
E(z)=\frac{1}{\tilde{\pi}}\frac{(\der \Delta)(z)}{\Delta(z)}=\sum_{\substack{a\in A\\ a \text{ monic}}}au(az).
\end{equation}
Moreover, for any $\gamma\in \Gamma(1)$, we have 
\begin{equation}\label{E:FE}
E(\gamma\cdot z)=j(\gamma;z)^2\det(\gamma)^{-1}\left(E(z)-\frac{\tilde{\pi}^{-1}c}{j(\gamma;z)}\right).
\end{equation}
We note that for any $N\in A$ of positive degree, 
\begin{equation}\label{E:uN}
u=\frac{u_N^{\inorm{N}}}{\mathfrak{C}_N(u_N)}\in A[[u_N]]
\end{equation}
where $\mathfrak{C}_a(X):=X^{|a|}C_a(X^{-1})\in A[X]$ for any $a\in A$.

The next corollary, which follows from Proposition \ref{P:1}, \eqref{E:uexp},  \eqref{E:FE} and \eqref{E:uN}, provides an explicit example of nearly holomorphic Drinfeld modular forms.

\begin{corollary}\label{Cor:1}  The function $E_2:\Omega^{\psi}(M)\to \CC_{\infty}$ defined by 
	\[
	E_2(z)=E(z)- \frac{1}{\tilde{\pi}z-\tilde{\pi}\psi(z)}
	\]
is an element in $\mathcal{N}_2^{1,\leq 1}(\Gamma(1);K)$.
\end{corollary}

From now on, to ease the notation, for  $\gamma\in\GL_2(K)$ and $z\in\Omega$, we set  
\[
\mathscr{J}(\gamma;z):=\frac{c_{\gamma}}{c_{\gamma}z+d_{\gamma}}=\frac{c_{\gamma}}{j(\gamma;z)}.
\]

\begin{proposition}\label{P:2}  For any $F=\sum_{i=0}^r\frac{f_i}{(\tilde{\pi}\Id-\tilde{\pi}\psi)^{i}}\in \mathcal{N}_k^{m,\leq r}(\Gamma;L)$, $\gamma\in \Gamma$, $z\in \Omega$ and $0\leq i \leq r$, we have 
	\begin{equation}\label{E:QM1}
	\begin{split}
	f_{r-i}(\gamma \cdot  z)&=j(\gamma;z)^{k-2(r-i)}\det(\gamma)^{-m+r-i}\bigg(f_{r-i}(z)+\mathscr{J}(\gamma;z)\binom{r-i+1}{1}f_{r-i+1}(z)\\
	&\ \  \ \ \ +\mathscr{J}(\gamma;z)^{2}\binom{r-i+2}{2}f_{r-i+2}(z)+\dots +\mathscr{J}(\gamma;z)^{i}\binom{r}{i}f_{r}(z)\bigg).
	\end{split}
	\end{equation}
	In particular, $f_r\in \mathcal{M}_{k-2r}^{m-r}(\Gamma;L)$.
\end{proposition}

\begin{proof} First, let us assume that $z\in \Omega^{\psi}(M)$. Note that for any $1\leq i \leq r$
	\begin{equation}\label{E:0}
	\sum_{j=1}^i(-1)^j\binom{r-(i-j)}{j}\binom{r}{i-j}=-\binom{r}{i}.
	\end{equation}
	Indeed, multiplying both sides of \eqref{E:0} by $\frac{i!}{r!}$ implies that \eqref{E:0} is equivalent to showing 
	\begin{equation}
	\sum_{j=1}^i (-1)^j\binom{i}{j}=-1,
	\end{equation}
	which follows from letting $x=1$ in $(x-1)^i=\sum_{j=0}^i(-1)^j\binom{i}{j}x^{i-j}$.
	
	Now, to prove the proposition, we do induction on $i$. When $i=0$, using Proposition \ref{P:1} and comparing the coefficient of $\Big(\frac{1}{\tilde{\pi}z-\tilde{\pi}\psi(z)}\Big)^r$ in both sides of 
	\begin{equation}\label{E:QM2}
	F(\gamma \cdot z)=j(\gamma;z)^k\det(\gamma)^{-m}F(z)
	\end{equation}
	imply that 
	\[
	f_r(\gamma\cdot z)=j(\gamma;z)^{k-2r}\det(\gamma)^{-m+r}f_r(z).
	\]	
	Now assuming that the proposition holds for non-negative integers smaller than $i$, we will show that \eqref{E:QM1} holds for $i$. To do this, let $\mathcal{L}$ be the coefficient of $\Big(\frac{1}{\tilde{\pi}z-\tilde{\pi}\psi(z)}\Big)^{r-i}$ on the left hand side of \eqref{E:QM2}. Using Proposition \ref{P:1} and the induction hypothesis, we obtain
	\begin{align*}
	\mathcal{L}&=f_{r-i}(\gamma\cdot z)j(\gamma;z)^{2r-2i}\det(\gamma)^{-r+i}\\
	&\ \ +j(\gamma;z)^{2r-2i+2}\det(\gamma)^{-r+i-1}\binom{r-i+1}{1}(-1)\mathscr{J}(\gamma;z)j(\gamma;z)^{k-2r+2i-2}\det(\gamma)^{-m+r-i+1}\\ &\ \ \times\left(f_{r-i+1}(z)+\mathscr{J}(\gamma;z)\binom{r-i+2}{1}f_{r-i+2}(z)+\dots+\mathscr{J}(\gamma;z)^{i-1}\binom{r}{i-1}f_r(z)\right)+\cdots+\\
	&\ \ +j(\gamma;z)^{2r-2i+4}\det(\gamma)^{-r+i-2}\binom{r-i+2}{2}(-1)^2\mathscr{J}(\gamma;z)^{2}j(\gamma;z)^{k-2r+2i-4}\det(\gamma)^{-m+r-i+2}\\ &\ \ \times\left(f_{r-i+2}(z)+\mathscr{J}(\gamma;z)\binom{r-i+3}{1}f_{r-i+3}(z)+\dots+\mathscr{J}(\gamma;z)^{i-2}\binom{r}{i-2}f_r(z)\right)+\cdots+\\
	&\ \ +j(\gamma;z)^{2r}\det(\gamma)^{-r}\binom{r}{i}(-1)^i\mathscr{J}(\gamma;z)^{i}j(\gamma;z)^{k-2r}\det(\gamma)^{-m+r}f_r(z).
	\end{align*}
	Now equating $\mathcal{L}$ to the coefficient of $\Big(\frac{1}{\tilde{\pi}z-\tilde{\pi}\psi(z)}\Big)^{r-i}$ on the right hand side of \eqref{E:QM2} and multiplying both sides of \eqref{E:QM2} by $j(\gamma;z)^{2i-2r}\det(\gamma)^{r-i}$, we get
	\begin{align*}
	&j(\gamma;z)^{k-2r+2i}\det(\gamma)^{-m+r-i}f_{r-i}(z)\\
	&=f_{r-i}(\gamma \cdot z)+j(\gamma;z)^{k-2r+2i}\det(\gamma)^{-m+r-i}\mathscr{J}(\gamma;z)(-1)\binom{r-i+1}{1}f_{r-i+1}(z)\\
	&\ \ \ \ +\mathscr{J}(\gamma;z)^{2}\Big(-\binom{r-i+1}{1}\binom{r-i+2}{1}+\binom{r-i+2}{2} \Big)f_{r-i+2}(z)+\cdots \\
	&\ \ \ \ + \mathscr{J}(\gamma;z)^{i}\Big(-\binom{r-i+1}{1}\binom{r}{i-1}+\dots + (-1)^i\binom{r}{i}\binom{r}{0}\Big)f_r(z)\bigg).
	\end{align*}
	Finally, using \eqref{E:0}, the above calculation finishes the proof of \eqref{E:QM1} for any element in $\Omega^{\psi}(M)$. We now show \eqref{E:QM1} for any $z\in \Omega$. To do this, for any $1\leq i \leq r$ and $\gamma \in \Gamma$, set 
	\[
	\mathcal{G}_{i,\gamma}(z):=j(\gamma;z)^{-k+2(r-i)}\det(\gamma)^{m-r+i}f_{r-i}(\gamma\cdot z).
	\]
	By \eqref{E:QM1}, we know that 
	\begin{equation}\label{E:extension}
	\mathcal{G}_{i,\gamma}(z)=f_{r-i}(z)+\mathscr{J}(\gamma;z)\binom{r-i+1}{1}f_{r-i+1}(z) +\dots +\mathscr{J}(\gamma;z)^{i}\binom{r}{i}f_{r}(z), \ \ z\in \Omega^{\psi}(M).
	\end{equation}
	Since both sides of \eqref{E:extension} are elements in $\mathcal{O}$, by Proposition \ref{Cl:QMF}, \eqref{E:extension} holds for all $z\in \Omega$. Moreover, for each $\alpha\in \Gamma(1)$, the holomorphy at infinity of $f_r$ with respect to $\alpha^{-1}\Gamma\alpha$ follows from the condition (iii) in Definition \ref{D:def} as well as the uniqueness of $u_{m_{\alpha^{-1}\Gamma\alpha}}$-expansion and it finishes the proof of the proposition.
\end{proof}

Now we are ready to state a fundamental property of nearly holomorphic Drinfeld modular forms.

\begin{theorem}\label{Thm:Structure_of_N}
    Any element $F$ of $\mathcal{N}_k^{m,\leq r}(\Gamma;L)$ may be uniquely expressed of the form
      \begin{equation}\label{E:Repr1}
  F=\sum_{0\leq j\leq r}g_jE_2^j
  \end{equation}
  with $g_j\in\mathcal{M}_{k-2j}^{m-j} (\Gamma;L)$. 
\end{theorem}

\begin{proof}
    By definition, we may uniquely express $F=\sum_{i=0}^r\frac{f_i}{(\tilde{\pi}\Id-\tilde{\pi}\psi)^{i}}$ for some $f_0,\dots,f_r\in \mathcal{O}$. By Proposition~\ref{P:2}, we know that $f_r\in\mathcal{M}_{k-2r}^{m-r}(\Gamma)$. Then $F-(-1)^rE_2^rf_r$ is an element in $\mathcal{N}_{k}^{m,\leq r-1}(\Gamma)$. By definition and Proposition~\ref{P:2} again, $F-(-1)^rE_2^rf_r$ can be expressed uniquely as $\sum_{i=0}^{r-1}\frac{\tilde{f}_i}{(\tilde{\pi}\Id-\tilde{\pi}\psi)^{i}}$ for some $\tilde{f}_0,\dots,\tilde{f}_{r-1}\in \mathcal{O}$ such that $\tilde{f}_{r-1}\in\mathcal{M}_{k-2(r-1)}^{m-(r-1)}(\Gamma)$. By subtracting $(-1)^{r-1}E_2^{r-1}\tilde{f}_{r-1}$, we see that $F-(-1)^rE_2^rf_r-(-1)^{r-1}E_2^{r-1}\tilde{f}_{r-1}$ is an element in $\mathcal{N}_{k}^{m,\leq r-2}(\Gamma)$. By repeating this procedure, we eventually obtain the expression for $F$ given as in \eqref{E:Repr1}.
\end{proof}

\begin{corollary}\label{C:depth} For any $F\in \mathcal{N}_k^{m,\leq r}(\Gamma)$, we have $r\leq k/2$.
\end{corollary}
\begin{proof} The corollary follows from the unique expression of $F$ given as in \eqref{E:Repr1} and the fact that $\mathcal{M}_{i}^{\mu}(\Gamma)=0$ for any $i<0$ and $\mu \in \mathbb{Z}/(q-1)\mathbb{Z}$ \cite[Thm. 11.1(b)]{BBP21}.
\end{proof}

We finish this section with the following proposition which will be used in \S6.

\begin{proposition}\label{P:mod} For any $\alpha\in \GL_2(K)$ and $F\in \mathcal{N}_k^{m,\leq r}(\Gamma)$, we have $F|_{k,m}\alpha\in \mathcal{N}_{k}^{m,\leq r}(\alpha^{-1}\Gamma\alpha)$.
\end{proposition}
\begin{proof} By Theorem \ref{Thm:Structure_of_N}, it suffices to show that $E_2|_{2,1}\alpha\in \mathcal{N}_2^{1,\leq 1}(\alpha^{-1}\Gamma\alpha)$. It is clear that for any $\delta\in  \alpha^{-1}\Gamma\alpha$, we have $(E_2|_{2,1}\alpha)|_{2,1}\delta=E_2|_{2,1}\alpha$. Now, for any $z\in \Omega$ and $B\in \GL_2(K)$, set $\tilde{g}_B(z):=E|_{2,1}B(z)+\frac{c_{B}}{c_{B}z+d_{B}}$. Clearly, $\tilde{g}_B\in \mathcal{O}$. Moreover, by Proposition \ref{P:1}, we have 
\[
E_2|_{2,1}\alpha(z)=\tilde{g}_{\alpha}(z)-\frac{1}{\tilde{\pi}(z-\psi(z)}.
\]
Similarly, note that for any $\gamma\in \GL_2(A)$, we have 
\[
(E_2|_{2,1}\alpha)|_{2,1}\gamma(z)=E_2|_{2,1}\alpha\gamma (z) = \tilde{g}_{\alpha\gamma}(z)-\frac{1}{\tilde{\pi}(z-\psi(z))}.
\]
Let $P(K)$ be the set of matrices in $\GL_2(K)$ of the form $\begin{pmatrix}
*&*\\0&*
\end{pmatrix}$. By \cite[Prop. 6.3(c)]{BBP21}, the double coset  $\Gamma \setminus \GL_2(K)/ P(K)$ can be represented by elements in $\Gamma(1)$. Now choose $\tilde{\rho}\in \Gamma(1)$ so that $\tilde{\rho}=\tilde{\gamma}(\alpha \gamma)\gamma'$ for some $\tilde{\gamma}\in \Gamma$ and $\gamma'\in P(K)$. Then 
\[
E_2|_{2,1}\tilde{\rho}(z)=E(z)-\frac{1}{\tilde{\pi}(z-\psi(z))} =(E_2|_{2,1}\alpha \gamma)|_{2,1}\gamma'(z)=\tilde{g}_{\alpha\gamma}|_{2,1}\gamma'(z)-\frac{1}{\tilde{\pi}(z-\psi(z))}
\]
where the last equality follows from Proposition \ref{P:1}. Using \eqref{E:uexp} and \eqref{E:FE}, we see that the $u_{m_{\tilde{\rho}^{-1}\Gamma\tilde{\rho}}}$-expansion of $E$ has no polar terms. Thus, by \cite[Prop. 5.11]{BBP21}, we also conclude that the $u_{m_{(\alpha \gamma)^{-1}\Gamma (\alpha \gamma)}}$-expansion of $\tilde{g}_{\alpha\gamma}$ has no polar terms, concluding the proof of the proposition.
\end{proof}

\section{Maass-Shimura operators}
In this section, our goal is to study Maass-Shimura operators and investigate some of their properties as well as their relation with Rankin-Cohen brackets. As a consequence of our analysis, we also obtain a sequence of operators in \S4.3 acting on the space of Drinfeld modular forms.

Let $f:\Omega\to \mathbb{C}_{\infty}$ be a holomorphic function and $z\in \Omega$. For each $n\geq 0$, we define \textit{the $n$-th hyperderivative $\deriv^nf:\Omega\to \mathbb{C}_{\infty}$ of $f$} by the formula
\[
f(z+\epsilon)=\sum_{n\geq 0}(\deriv^nf)(z)\epsilon^n, \ \ \epsilon\in \mathbb{C}_{\infty}, \ \ |\epsilon| \text{ is sufficiently small}.
\]
For any $r\geq 0$, let us set $\der^r:=\frac{\deriv^r}{\tilde{\pi}^r}$. Note that for any $f,g\in \mathcal{O}$, we have 
\begin{equation}\label{E:product}
\der^n(fg)=\sum_{i=0}^n(\der^if)(\der^{n-i}g).
\end{equation}
For more details on hyperderivatives, we refer the reader to \cite[\S3.1]{BP08} and \cite[\S2]{US98}.

\begin{definition} 
	\begin{itemize}
		\item[(i)] Let $k,\mu$ and $r$ be non-negative integers so that $k\geq 2\mu$. For any holomorphic function $f$, we define \textit{the Maass-Shimura operator} $\delta_k^r$  by $\delta_k^r:=\Id$ for $r=0$ and 
	\begin{multline*}
	\delta_k^r\left(\frac{f}{\left(\tilde{\pi}\Id-\tilde{\pi}\psi\right)^{\mu}}\right):=\frac{\der^rf}{\left(\tilde{\pi}\Id-\tilde{\pi}\psi\right)^\mu}+\binom{k-\mu+r-1}{1}\frac{\der^{r-1}f}{\left(\tilde{\pi}\Id-\tilde{\pi}\psi\right)^{\mu+1}}+\cdots\\ 
	\ \ \ \ \ \ +\binom{k-\mu+r-1}{r-1}\frac{\der f}{(\tilde{\pi}\Id-\tilde{\pi}\psi)^{\mu+r-1}}+\binom{k-\mu+r-1}{r}\frac{f}{(\tilde{\pi}\Id-\tilde{\pi}\psi)^{\mu+r}}
	\end{multline*}
	for $r\geq 1$ where $\der^i$ is the $i$-th hyperderivative of $f$. For convenience, we further set $\der:=\der^1$ and  $\delta_{k}:=\delta_{k}^{1}$.
	\item[(ii)] For any $F=\sum_{i=0}^r\frac{f_i}{(\tilde{\pi}\Id-\tilde{\pi}\psi)^{i}}\in \mathcal{N}_{k}^{m,\leq r}(\Gamma)$, we set 
	\[
	\delta_{k}^r(F):=\sum_{i=0}^r\delta_{k}^r\left(\frac{f_i}{\left(\tilde{\pi}\Id-\tilde{\pi}\psi\right)^{i}}\right).
	\]
\end{itemize}
\end{definition}

\subsection{Basic properties of Maass-Shimura operators} In this subsection, we investigate several properties of our Maass-Shimura operator. 

We start with proving our next lemma.

\begin{lemma}\label{L:3} For any $k,r\in \ZZ_{\geq 1}$ and $j,\ell \in \ZZ_{\geq 0}$ so that $0\leq j+\ell\leq r-1$, we have 
	\[
	\sum_{i=0}^{r}(-1)^{i-\ell}\binom{k+r-1}{i}\binom{k+r-1-i}{r-j-i}\binom{i}{\ell}=0.
	\]
\end{lemma}
\begin{proof} Note that 
	\[
	\binom{k+r-1}{i}\binom{k+r-1-i}{r-j-i}\binom{i}{\ell}=\frac{(k+r-1)!}{(r-j-i)!(k+j-1)!(i-\ell)!\ell!}.
	\]
	Thus, after dividing and multiplying by $(r-j-\ell)!$, one can see that proving the lemma is equivalent to showing that 
	\[
	\sum_{i=0}^r(-1)^{i-\ell}\binom{r-j-\ell}{i-\ell}=0.
	\]	
	Note that 
	\begin{align*}
	\sum_{i=0}^r(-1)^{i-\ell}\binom{r-j-\ell}{i-\ell}&=\sum_{i=\ell}^{r-j}(-1)^{i-\ell}\binom{r-j-\ell}{i-\ell}
	=\sum_{\mu=0}^{r-j-\ell}(-1)^{\mu}\binom{r-j-\ell}{\mu}=0
	\end{align*}
	where the last equality follows from setting $x=1$ in $(1-x)^{r-j-\ell}=\sum_{i=0}^{r-j-\ell}(-1)^i\binom{r-j-\ell}{i}x^i$ and it finishes the proof.
\end{proof}

The proof of the next lemma easily follows from \eqref{E:product} and the definition of $\delta_k^r$.

\begin{lemma}\label{L:derivation} Let $f,g\in \mathcal{O}$. For any $m,n,\mu_1,\mu_2\in \ZZ_{\geq 0}$, we have 
	\begin{multline*}
	\delta_{m+n}\left(\frac{f}{(\tilde{\pi}\Id-\tilde{\pi}\psi)^{\mu_1}}\frac{g}{(\tilde{\pi}\Id-\tilde{\pi}\psi)^{\mu_2}}\right)\\=\frac{f(z)}{(\tilde{\pi}\Id-\tilde{\pi}\psi)^{\mu_1}}\delta_m\left(\frac{g}{(\tilde{\pi}\Id-\tilde{\pi}\psi)^{\mu_2}}\right)+\frac{g}{(\tilde{\pi}\Id-\tilde{\pi}\psi)^{\mu_2}}\delta_n\left(\frac{f}{(\tilde{\pi}\Id-\tilde{\pi}\psi)^{\mu_1}}\right).
	\end{multline*}
\end{lemma}

\begin{proposition}\label{P:3}  For any $\gamma\in\GL_2(K)$, a holomorphic function $f:\Omega\to \mathbb{C}_{\infty}$ and $k,m,\mu,r\in \mathbb{Z}_{\geq 0}$ such that $k\geq 2\mu$, we have 
	\begin{equation}\label{E:mod}
	\delta_k^r\left(\frac{f}{\left(\tilde{\pi}\Id-\tilde{\pi}\psi\right)^{\mu}}\right)\bigg|_{k+2r,m+r}\gamma=\delta_{k}^r\left(\frac{f}{\left(\tilde{\pi}\Id-\tilde{\pi}\psi\right)^{\mu}}\bigg|_{k,m}\gamma\right).
	\end{equation}
	In particular, for any field $K\subseteq L \subseteq \mathbb{C}_{\infty}$, if $F\in \mathcal{N}_{k}^{m,\leq s}(\Gamma;L)$, then $\delta_{k}^r(F)\in \mathcal{N}_{k+2r}^{m+r,\leq s+r}(\Gamma;L)$.
\end{proposition}
\begin{proof} Let $z\in \Omega^{\psi}(M)$. We first prove \eqref{E:mod} when $\mu=0$. Using Proposition \ref{P:1} and \cite[Lem. 3.4]{BP08}, whose proof indeed works for any holomorphic function on $\Omega$, for any $\gamma\in \GL_2(K)$, one can obtain
	\begin{align*}
	&\delta_k^r(f)(\gamma \cdot z)\\
	&=\frac{j(\gamma;z)^{k+2r}}{\det(\gamma)^{m+r}}\binom{k+r-1}{0}\sum_{j=0}^r\binom{k+r-1}{r-j}\mathscr{J}(\gamma;z)^{r-j}\left(\der^jf|_{k,m}\gamma\right)(z)\\
	&\ \ +\frac{j(\gamma;z)^{k+2r}}{\det(\gamma)^{m+r}}\binom{k+r-1}{1}\left(\sum_{j=0}^{r-1}\binom{k+r-2}{r-1-j}\mathscr{J}(\gamma;z)^{r-1-j}\left(\der^jf|_{k,m}\gamma\right)(z)\right)\\
	&\ \ \ \ \ \  \times\left(\sum_{\ell=0}^1\binom{1}{\ell}\frac{(-1)^{1-\ell}}{(z-\psi(z))^\ell}\mathscr{J}(\gamma;z)^{1-\ell}\right)+\cdots \\
	&\ \ +\frac{j(\gamma;z)^{k+2r}}{\det(\gamma)^{m+r}}\binom{k+r-1}{r-1}\left(\sum_{j=0}^1\binom{k}{1-j}\mathscr{J}(\gamma;z)^{1-j}\left(\der^jf|_{k,m}\gamma\right)(z)\right)\\
	&\ \ \ \ \ \  \times\left(\sum_{\ell=0}^{r-1}\binom{r-1}{\ell}\frac{(-1)^{r-1-\ell}}{(z-\psi(z))^\ell}\mathscr{J}(\gamma;z)^{r-1-\ell}\right)\\
	&\ \ +\frac{j(\gamma;z)^{k+2r}}{\det(\gamma)^{m+r}}\binom{k+r-1}{r}\left(f|_{k,m}\gamma\right)(z)\sum_{\ell=0}^{r}\binom{r}{\ell}\frac{(-1)^{r-\ell}}{(z-\psi(z))^\ell}\mathscr{J}(\gamma;z)^{r-\ell}\\
	&=\frac{j(\gamma;z)^{k+2r}}{\det(\gamma)^{m+r}}\times \\
	&\ \ \sum_{\substack{0\leq j, \ell\leq r\\0\leq j+\ell\leq r}}\left( \sum_{i=0}^r(-1)^{i-\ell}\binom{k+r-1}{i}\binom{k+r-1-i}{r-j-i}\binom{i}{\ell} \right) \mathscr{J}(\gamma;z)^{r-j-\ell}\frac{\left(\der^jf|_{k,m}\gamma\right)(z)}{(z-\psi(z))^\ell} \\
	&=\frac{j(\gamma;z)^{k+2r}}{\det(\gamma)^{m+r}}\sum_{\ell=0}^{r}\binom{k+r-1}{\ell}\frac{\left(\der^{r-\ell}f|_{k,m}\gamma\right)(z)}{(z-\psi(z))^{\ell}}\\
	&=\frac{j(\gamma;z)^{k+2r}}{\det(\gamma)^{m+r}}\delta_{k}^r(f|_{k,m}\gamma)(z)
	\end{align*}
	where the third equality follows from Lemma \ref{L:3}. Thus, the proof of the first assertion is completed when $\mu=0$. We now assume $\mu\geq 1$ and claim that 
	\begin{equation}\label{C:1}
	\delta_{k}^r\left(\frac{f}{\left(\tilde{\pi}\Id-\tilde{\pi}\psi\right)^{\mu}}\bigg|_{k,m}\gamma\right)(z)=\frac{1}{\tilde{\pi}^{\mu}}\left(\frac{1}{z-\psi(z)}-\frac{c}{cz+d}\right)^{\mu}\frac{j(\gamma;z)^{\mu}}{\det(\gamma)^{\mu}}\delta_{k-\mu}^r\left(f|_{k-\mu,m}\gamma\right)(z).
	\end{equation} 
	By using Proposition \ref{P:1}, above discussion and \eqref{C:1}, we obtain
	\begin{align*}
	\delta_{k}^r\left(\frac{f}{(\tilde{\pi}\Id-\tilde{\pi}\psi)^{\mu}}\right)(\gamma\cdot z)&=\frac{1}{\tilde{\pi}^{r}}\frac{1}{\left(\gamma\cdot z-\psi(\gamma\cdot z)\right)^{\mu}}\delta_{k-\mu}^r(f)(\gamma\cdot z)\\
	&=\frac{1}{\left(\gamma\cdot z-\psi(\gamma\cdot z)\right)^{\mu}}\frac{j(\gamma;z)^{k-\mu+2r}}{\det(\gamma)^{m+r}}\delta_{k-\mu}^r(f|_{k-\mu,m}\gamma)(z)\\
	&=\left(\frac{1}{z-\psi(z)}-\frac{c_{\gamma}}{c_{\gamma}z+d_{\gamma}}  \right)^{\mu}\frac{j(\gamma;z)^{k+\mu+2r}}{\det(\gamma)^{m+r+\mu}}\delta_{k-\mu}^r(f|_{k-\mu,m}\gamma)(z)\\
	&=\frac{j(\gamma;z)^{k+2r}}{\det(\gamma)^{m+r}}\delta_{k}^r\left(\frac{f}{\left(\tilde{\pi}\Id-\tilde{\pi}\psi\right)^{\mu}}\bigg|_{k,m}\gamma\right)(z)
	\end{align*}
	which implies \eqref{E:mod}. Hence, it suffices to show the claim to finish the proof of the first assertion. Note that, for any $i, \ell \geq 0$, by \eqref{E:product} (see also \cite[pg. 20]{BP08}),  we have 
	\begin{equation}\label{E:jder}
	\der^{i}((c_{\gamma}z+d_{\gamma})^{\ell})=\binom{\ell}{i}\mathscr{J}(\gamma;z)^i(c_{\gamma}z+d_{\gamma})^{\ell}.
	\end{equation}
	Using Proposition \ref{P:1}, \eqref{E:product} and \eqref{E:jder}, we see that   
	\eqref{C:1} is equivalent to the identity
	\begin{multline}\label{C:22}
	\sum_{w=0}^{\mu}\binom{\mu}{w}(-c_{\gamma})^{\mu-w}\sum_{\ell=0}^r\binom{k-w+r-1}{\ell}\frac{1}{(z-\psi(z))^{w+\ell}}\times \\
	\sum_{\ell_1=0}^{r-\ell}\binom{w}{\ell_1}\frac{(-1)^{\ell_1}(-c_{\gamma})^{\ell_1}}{(c_{\gamma}z+d_{\gamma})^{\ell_1}}(c_{\gamma}z+d_{\gamma})^{w}\der^{r-\ell-\ell_1}((c_{\gamma}z+d_{\gamma})^{\mu}f|_{k,m}\gamma)(z)\\
	=\left(\sum_{w=0}^{\mu}\binom{\mu}{w}(-c_{\gamma})^{\mu-w}(c_{\gamma}z+d_{\gamma})^w\frac{1}{(z-\psi(z))^w}\right)\times\\\left(\sum_{\ell=0}^{r}\binom{k-\mu+r-1}{\ell}\frac{1}{(z-\psi(z))^{\ell}}\der^{r-\ell}((c_{\gamma}z+d_{\gamma})^{\mu}f|_{k,m}\gamma)(z)\right).
	\end{multline}
	Note that for each $w,\ell\geq 0$, the term $\frac{1}{(z-\psi(z))^{w+\ell}}(-c_{\gamma})^{\mu-w}(c_{\gamma}z+d_{\gamma})^{w}\der^{r-\ell}((c_{\gamma}z+d_{\gamma})^{\mu}f|_{k,m}\gamma)(z)$ on the right hand side of \eqref{C:22} has  coefficient $\binom{\mu}{w}\binom{k-\mu+r-1}{\ell}$. On the other hand, the coefficient of the same term on the left hand side of \eqref{C:22} is 
	\begin{align*}
	\sum_{i=w}^{w+\ell}(-1)^{i+w}\binom{\mu}{i}\binom{k+r-1-i}{w+\ell-i}\binom{i}{i-w}&=\sum_{i=0}^{\ell}(-1)^i\binom{\mu}{i+w}\binom{k+r-1-w-i}{\ell-i}\binom{i+w}{i}\\
	&=\binom{\mu}{w}\sum_{i=0}^{\ell}(-1)^i\binom{\mu-w}{i}\binom{k+r-1-w-i}{\ell-i}\\
	&=\binom{\mu}{w}\sum_{\nu=0}^{\ell}(-1)^{\ell+\nu}\binom{\mu-w}{\ell-\nu}\binom{k+r-w-\ell+\nu-1}{\nu}\\
	&=\binom{\mu}{w}\binom{k-\mu+r-1}{\ell},
	\end{align*}
	where the last equality follows from \cite[Lem. 3.2]{BP08}. Hence, the identity \eqref{C:22} holds and thus the proof of our claim is completed.
	To prove the second assertion, it remains to show that, for any $F=\sum_{i=0}^r\frac{f_i}{(\tilde{\pi}\Id-\tilde{\pi}\psi)^{i}}\in \mathcal{N}_{k}^{m,\leq r}(\Gamma;L)$, Definition \ref{D:def}(iii) holds for $\delta_{k}^r(F)$. Note that by \cite[Lem. 3.6]{US98}, one can see that if $f(z)=\sum_{i\geq 0}a_iu_{N}(z)^i$ for some $a_i\in \mathbb{C}_{\infty}$, $N\in A$ and sufficiently large $z\in \Omega$, then, for any $n\in \mathbb{Z}_{\geq 1}$, we have $\der^n(f)(z)=\sum_{i\geq 0}b_{i,n}u_{N}(z)^i$ for some $b_{i,n}\in \mathbb{C}_{\infty}$. Thus, for any $\gamma\in \Gamma(1)$, by \eqref{E:mod}, we obtain
	\begin{align*}
	\delta_{k}^r(F)|_{k+2r,m+r}\gamma(z)&=\delta_{k}^r(F|_{k,m}\gamma)(z)\\
	&=\delta_{k}^r\left( \sum_{i=0}^s\frac{1}{(\tilde{\pi}z-\tilde{\pi}\psi(z))^{i}}\sum_{\ell= 0}^{\infty}a_{\gamma,i,\ell}u_{m_{\gamma^{-1}\Gamma\gamma}}(z)^{\ell}\right)\\
	&=\sum_{\mu=0}^{s+r}\frac{1}{(\tilde{\pi}z-\tilde{\pi}\psi(z))^{\mu}}\sum_{\ell= 0}^{\infty}b_{\gamma,\mu,\ell}u_{m_{\gamma^{-1}\Gamma\gamma}}(z)^{\ell}
	\end{align*}
	which finishes the proof of the second assertion.
\end{proof}

\subsection{Rankin-Cohen brackets}
In this subsection, we investigate the relation of our Maass-Shimura  operators with Rankin-Cohen brackets defined in \cite[Thm. 3.7]{US98}. 

Let $k,w$ and $r$ be non-negative integers. Letting $0\leq \nu\leq r$, we set 
\[
\tilde{\beta}_{r,\nu}:=(k+r-1)(k+r-2)\cdots (k+\nu)(w+r-1)(w+r-2)\cdots(w+r-\nu)
\]
and consider
$
\beta_{r,\nu}:=\frac{\tilde{\beta}_{r,\nu}}{\gcd_{0\leq \nu \leq r}\tilde{\beta}_{r,\nu}}.
$
Let $f,g\in \mathcal{O}$. We define their \textit{$r$-th Rankin-Cohen bracket} $[f,g]_{k,w,r}$ by 
\[
[f,g]_{k,w,r}:=\sum_{\nu=0}^r(-1)^{r-\nu}\beta_{r,\nu}(\der^{\nu}f)(\der^{r-\nu}g).
\]

Our goal in the next theorem is to obtain the $r$-th Rankin-Cohen bracket in terms of the Maass-Shimura operator and hence providing an alternative proof for the generalization of \cite[Thm.~3.7]{US98}.

\begin{theorem}\label{Thm:Rankin_Cohen_Braket} On $\Omega^{\psi}(M)$, we have
	\[
	[f,g]_{k,w,r}=\sum_{\nu=0}^r(-1)^{r-\nu}\beta_{r,\nu}(\delta_k^{\nu}f)(\delta_{w}^{r-\nu}g).
	\]
	Moreover, for any field $K\subseteq L \subseteq \mathbb{C}_{\infty}$,  if $f\in\mathcal{M}_{k}^{m_1}(\Gamma;L)$ and $g\in\mathcal{M}_{w}^{m_2}(\Gamma;L)$, then $[f,g]_{k,w,r}\in\mathcal{M}_{k+w+2r}^{m_1+m_2+r}(\Gamma;L)$.
\end{theorem}
\begin{proof}
	Firstly, observe, for any $0\leq \nu \leq r$, that 
	\begin{equation}\label{E:c}
	\tilde{\beta}_{r,\nu}=(r-\nu)!\nu!\binom{k+r-1}{k+\nu-1}\binom{w+r-1}{w+r-\nu-1}=(r-\nu)!\nu!\binom{k+r-1}{r-\nu}\binom{w+r-1}{\nu}.
	\end{equation}
	By a direct calculation, we have 
	\begin{align*}
	&\sum_{\nu=0}^r(-1)^{r-\nu}\tilde{\beta}_{r,\nu}(\delta_k^{\nu}f)(\delta_{w}^{r-\nu}g)\\
	&=\sum_{\nu=0}^{r}(-1)^{r-\nu}\tilde{\beta}_{r,\nu}\sum_{i=0}^{\nu}\binom{k+\nu-1}{\nu-i}(\der^if)\sum_{j=0}^{r-\nu}\binom{w+r-\nu-1}{r-\nu-j}\frac{1}{(\tilde{\pi}\Id-\tilde{\pi}\psi)^{r-i-j}}(\der^jg).
	\end{align*}
	
	Using \eqref{E:c}, note that for each $m,n\in \mathbb{Z}_{\geq 0}$ so that $0\leq m+n\leq r$, the coefficient of the term $\frac{(\der^mf)(\der^ng)}{(\tilde{\pi}\Id-\tilde{\pi}\psi)^{r-m-n}}$ on the right hand side of above equality is
	\begin{align*}
	&\sum_{i=m}^r(-1)^{r-i}\tilde{\beta}_{r,i}\binom{k+i-1}{i-m}\binom{w+r-i-1}{r-i-n}\\
	&=\sum_{i=m}^r(-1)^{r-i}(r-i)!i!\binom{k+r-1}{k+i-1}\binom{w+r-1}{w+r-i-1}\binom{k+i-1}{i-m}\binom{w+r-i-1}{r-i-n}\\
	&=\sum_{i=m}^r(-1)^{r-i}\frac{(k+r-1)!(w+r-1)!(r-m-n)!}{(i-m)!(k+m-1)!(w+n-1)!(r-i-n)!(r-m-n)!}\\
	&=\frac{(k+r-1)!(w+r-1)!}{(r-m-n)!(k+m-1)!(w+n-1)!}(-1)^r\sum_{i=m}^r(-1)^i\binom{r-m-n}{i-m}\\
	&=\tilde{\beta}_{r,m}(-1)^{r+m}\sum_{\ell=0}^{r-m}(-1)^{\ell}\binom{r-m-n}{\ell}\\
	&=\tilde{\beta}_{r,m}(-1)^{r-m}\sum_{\ell=0}^{r-m-n}(-1)^{\ell}\binom{r-m-n}{\ell}.
	\end{align*}
	Thus, we have 
	\[
	\sum_{i=m}^r(-1)^{r-i}\tilde{\beta}_{r,i}\binom{k+i-1}{i-m}\binom{w+r-i-1}{r-i-n}
	=\begin{cases} 0 & \text{ if }r>m+n\\
	\tilde{\beta}_{r,m}(-1)^{r-m} & \text{ if } r=m+n  \end{cases},
	\]
	which implies the desired equality. On the other hand, by \cite[Lem. 2.4]{US98}, we know that $[f,g]_{k,w,r}$ is an element in $\mathcal{O}$. Then, for any $z\in \Omega$ and $\gamma\in \Gamma$, if we set
	\[ \mathcal{G}_{\gamma}(z):=j(\gamma;z)^{-k-w-2r}\det(\gamma)^{m_1+m_2+r}[f,g]_{k,w,r}(\gamma\cdot z)
	\]
	and use Proposition \ref{P:3}, we see that 
	$\mathcal{G}_{\gamma}$ coincides with $[f,g]_{k,w,r}$ on $\Omega^{\psi}(M)$. Hence, by Proposition \ref{Cl:QMF}, 
	$\mathcal{G}_{\gamma}$ agrees with $[f,g]_{k,w,r}$ on $\Omega$. Thus $[f,g]_{k,w,r}$ is a weak Drinfeld modular form for $\Gamma$. For each $\alpha\in \Gamma(1)$, to prove that $[f,g]_{k,w,r}|_{k+w+2r,m_1+_2+r}\alpha$ is holomorphic at infinity with respect to $\alpha^{-1}\Gamma \alpha$, we use the argument in the proof of the second assertion of Proposition \ref{P:3} and conclude that, for any sufficiently large $z\in \Omega$, 
	\[
	[f,g]_{k,w,r}|_{k+w+2r,m_1+m_2+r}\alpha(z)=\sum_{i=0}^{\infty}a_{i}u_{m_{\alpha^{-1}\Gamma \alpha}}(z)^i, \ \  a_{i}\in L.
	\]
This implies that $[f,g]_{k,w,r}\in \mathcal{M}_{k+w+2r}^{m_1+m_2+r}(\Gamma;L)$.
\end{proof}

\subsection{A sequence of operators}
In this subsection, inspired by Shimura \cite[\S 16.1]{Shi07}, we obtain a sequence of operators preserving the modularity.

In what follows, for each $r\in \mathbb{Z}_{\geq 2}$, we define an operator acting on $\mathcal{O}$. For any non-negative integer $k$, let
\[
\widetilde{c_1^r(k)}:=(r-1)\binom{k+r-1}{r-1}, \text{ and }  \widetilde{c_v^r(k)}:=rk^{v-1}\binom{k+r-1}{r-v}, \ \ v\geq 2.
\]
Consider $c_v^r(k):=\frac{\widetilde{c_v^r(k)}}{\gcd_{0\leq \nu \leq r}\widetilde{c_v^r(k)}}$ and define 
\[
\mathcal{U}_k^r(f):=\sum_{v=1}^r(-1)^{r-v}c_v^r(k)f^{v-1}(\der f)^{r-v}(\der^vf).
\]
By \cite[Lem. 2.4]{US98}, $\mathcal{U}_k^r(f)$ is holomorphic on $\Omega$. We now collect some properties of the operator $\mathcal{U}_k^r$.
\begin{theorem}\label{T:DO} The following statements hold.
	\begin{itemize}
		\item[(i)] On $\Omega^{\psi}(M)$, we have
		\begin{equation}\label{E:operator}
		\mathcal{U}_k^r(f)=\sum_{v=1}^r(-1)^{r-v}c_v^r(k)f^{v-1}(\delta_kf)^{r-v}(\delta_k^vf).
		\end{equation}
		\item[(ii)] Let $m\in \mathbb{Z}/(q-1)\mathbb{Z}$ and $\gamma\in \GL_2(K)$. We have 
		\[
		\mathcal{U}_k^r(f|_{k,m}\gamma)=\mathcal{U}_k^r(f)|_{kr+2r,mr+r}\gamma.
		\]
		Furthermore, for any field $K\subseteq L \subseteq \mathbb{C}_{\infty}$, if $f\in \mathcal{M}_{k}^{m}(\Gamma;L)$, then $\mathcal{U}_k^r(f)\in \mathcal{M}_{kr+2r}^{mr+r}(\Gamma;L)$.
	\end{itemize}
\end{theorem}
\begin{proof} We prove the first assertion. Let us denote the right hand side of \eqref{E:operator} by $\mathcal{V}_k^r(f)$. We claim that the coefficients of the terms in
	\begin{multline}\label{E:terms}
	\mathcal{V}_k^r(f)=(-1)^{r-1}(r-1)\binom{k+r-1}{r-1}\left(\der f+ \frac{kf}{\tilde{\pi}\Id-\tilde{\pi}\psi}\right)^{r}+\\r\sum_{v=2}^r(-1)^{r-v}k^{v-1}\binom{k+r-1}{r-v}f^{v-1}\left(\der f+\frac{kf}{\tilde{\pi}\Id-\tilde{\pi}\psi}\right)^{r-v}\left(\sum_{i=0}^v\binom{k+v-1}{i}\frac{\der^{v-i}f}{(\tilde{\pi}\Id-\tilde{\pi}\psi)^i}\right)
	\end{multline}
	involving positive powers of $\frac{1}{\tilde{\pi}\Id-\tilde{\pi}\psi}$ are zero. We divide the proof into the following cases:
	
	\textbf{Case 1:} For $n\geq 1$ and $i\geq 2$, the coefficient of the term $(\der f)^{r-n-i}(\der^if)f^{i+n-1}\frac{1}{(\tilde{\pi}\Id-\tilde{\pi}\psi)^n}$ is 
	\begin{multline*}
	rk^{n+i-1}\sum_{j=0}^{n}(-1)^{r-i-j}\binom{k+r-1}{r-i-j}\binom{r-i-j}{n-j}\binom{k+i+j-1}{j}\\
	=rk^{n+i-1}\frac{(k+r-1)!}{(r-i-n)!(k+i-1)!}\sum_{j=0}^{n}(-1)^{r-i-j}\frac{1}{(n-j)!j!}=0.
	\end{multline*}
	
	\textbf{Case 2:} For  $r>n\geq 2$, the coefficient of the term $(\der f)^{r-n}f^n	\frac{1}{(\tilde{\pi}\Id-\tilde{\pi}\psi)^n}$ is 
	\begin{align*}
	&rk^{n-1}\sum_{i=2}^n(-1)^{r-i}\binom{k+r-1}{r-i}\binom{r-i}{n-i}\binom{k+i-1}{i}\\
	&+rk^n\sum_{i=3}^{n+1}(-1)^{r-i}\binom{k+r-1}{r-i}\binom{r-i}{n-(i-1)}\binom{k+i-1}{i-1}\\
	&+(-1)^{r-1}k^{n-1}\binom{k+r-1}{r-1}k(r-1)\binom{r}{n}+(-1)^{r-2}rk^n\binom{k+r-1}{r-2}\binom{r-2}{n-1}\binom{k+1}{1}\\
	&=\frac{rk^n(k+r-1)!}{(r-n)!k!}\sum_{i=2}^n(-1)^{r-i}\frac{1}{(n-i)!i!}+\frac{rk^n(r-n)(k+r-1)!}{(r-n)!k!}\sum_{i=3}^{n+1}\frac{(-1)^{r-i}}{(n-(i-1))!(i-1)!}\\
	&\ \ \ \  \ \ \ \ \ \ \ +(-1)^{r-1}\frac{rk^n(k+r-1)!}{k!(r-n)!}\left(\frac{r-1}{n!}-\frac{r-n}{(n-1)!}\ \right)\\
	&=\frac{(-1)^rrk^n(k+r-1)!(-r+n+1)}{(r-n)!k!}\left(\sum_{i=2}^n\frac{(-1)^{i}}{(n-i)!i!}-\frac{(n-1)}{n!}\right)\\
	&=0.
	\end{align*}
	\textbf{Case 3:} For  $n=r\geq 2$, the coefficient of the term $f^n 	\frac{1}{(\tilde{\pi}\Id-\tilde{\pi}\psi)^n}$ is 
	\begin{align*}
	&nk^{n-1}\sum_{v=2}^n(-1)^{n-v}\binom{k+n-1}{n-v}\binom{k+v-1}{v}+(-1)^{n-1}(n-1)k^n\binom{k+n-1}{n-1}\\
	&=\frac{(-1)^nn(k+n-1)!k^{n-1}}{(k-1)!}\sum_{v=2}^{n}(-1)^v\frac{1}{(n-v)!v!}+\frac{(-1)^{n-1}(k+n-1)!(n-1)k^n}{k!(n-1)!}\\
	&=\frac{(-1)^nn(k+n-1)!k^{n}(n-1)}{k!n!}+\frac{(-1)^{n-1}n(k+n-1)!k^n(n-1)}{k!n!}\\
	&=0.
	\end{align*}
	
	\textbf{Case 4:} The coefficient of the term $(\der f)^{r-1} f\frac{1}{\tilde{\pi}\Id-\tilde{\pi}\psi}$ is 
	\begin{align*}
	(-1)^{r-2}rk\left(\binom{k+r-1}{r-2}(k+1)-(r-1)\binom{k+r-1}{r-1}\right)=0.
	\end{align*}
	By \eqref{E:terms}, we see that above four cases include all the possible terms having a positive power of $\frac{1}{\tilde{\pi}\Id-\tilde{\pi}\psi}$ and hence the claim follows. It is now easy to see that $\mathcal{U}_k^r(f)=\mathcal{V}_k^r(f)$ as desired. Part(ii) follows from the same idea used in the proof of the second assertion of Theorem \ref{Thm:Rankin_Cohen_Braket} and hence we leave the details to the reader.
\end{proof}
\begin{remark} Let $\mathfrak{g}$ and $\mathfrak{h}$ be Drinfeld modular forms defined in Example \ref{Ex:Drinfeld_Modular_Forms}. When $0\leq r\leq q$, using \cite[Thm. 4.1]{BP08}, we have $\mathcal{U}_{q-1}^r(\mathfrak{g})=\mathcal{U}_{q+1}^r(\mathfrak{g})=0$. Moreover, we obtain $\mathcal{U}_{q-1}^{q+1}(\mathfrak{g})=0$ and
	\[
	\mathcal{U}_{q+1}^{q+1}(\mathfrak{h})=\frac{\mathfrak{h}^{q+3}}{\theta^q-\theta}=\frac{\mathfrak{h}^4\Delta }{\theta-\theta^q},
	\]
	where the last equality follows from \cite[\S 6]{Gek88}. 
\end{remark}

\section{Drinfeld quasi-modular forms}

Let $\Gamma\leqslant  \Gamma(1)$ be a congruence subgroup. In this section, we study Drinfeld quasi-modular forms for $\Gamma$ which generalizes the definition of Bosser and Pellarin \cite{BP08} for the full modular group. 
\begin{definition}  
	\begin{itemize}
	\item[(i)] \textit{A Drinfeld quasi-modular form $f:\Omega\to \mathbb{C}_{\infty}$ of weight $k\in \ZZ_{\geq 0}$, type $m\in \ZZ/(q-1)\ZZ$ and depth $r\in \ZZ_{\geq 0}$  for $\Gamma$} is a holomorphic function such that for any $\gamma\in \Gamma$, we have
		\begin{equation}\label{E:quasimod}
		(f|_{k,m}\gamma)(z)=f_{0}(z)+\mathscr{J}(\gamma;z)f_{1}(z)+\dots +\mathscr{J}(\gamma;z)^{r}f_{r}(z),
		\end{equation}
	for some $f_0,\dots,f_r\in \mathcal{O}$ so that for each $0\leq i \leq r$ and $\alpha \in \Gamma(1)$, $f_i|_{k-2i,m-i}\alpha$ is bounded on vertical lines. We denote the $\CC_{\infty}$-vector space generated by all Drinfeld quasi-modular forms of weight $k$, type $m$ and depth at most $r$ for $\Gamma$ by $\QM_{k}^{m,\leq r}(\Gamma)$. 
	\item[(ii)] Observe that setting $\gamma=\Id_2$ in \eqref{E:quasimod} implies that $f_0=f$. In particular, since $f$ is bounded on vertical lines, choosing $\gamma=\begin{pmatrix}1&a\\0&1\end{pmatrix}$ for any $a\in m_{\Gamma}$, we see that $f$ is holomorphic at infinity with respect to $\Gamma$. For any field $K\subseteq L\subseteq \mathbb{C}_{\infty}$, we denote by $\QM_{k}^{m,\leq r}(\Gamma;L)$ the $L$-vector space generated by all $f\in \QM_{k}^{m,\leq r}(\Gamma)$  whose $u_{m_{\Gamma}}$-expansions have coefficients lying in $L$.
	\end{itemize}
\end{definition}

Our main goal is to prove that $\mathcal{N}_{k}^{m,\leq r}(\Gamma)$ and $\mathcal{Q}\mathcal{M}_{k}^{m,\leq r}(\Gamma)$ are canonically isomorphic. First, we need a slight generalization of \cite[Sec.~2,~Rem.~(i)]{BP08} which should be compared with Theorem \ref{Thm:Transcendence_of_Sigma}.

\begin{lemma}\label{Lem:Uniqueness_of_QM}
    Let $f_0,\dots,f_r\in\mathcal{O}$. If
    \begin{equation}\label{Eq:QM_Uniqueness}
        f_0(z)+f_1(z)\mathscr{J}(\gamma;z)+\cdots+f_r(z)\mathscr{J}(\gamma;z)^r=0
    \end{equation}
    for all $z\in\Omega$ and $\gamma\in\Gamma$, then $f_i\equiv 0$ for all $0\leq i\leq r$. In particular, for any $f\in \QM_{k}^{m,\leq r}(\Gamma)$, the expression in \eqref{E:quasimod} is unique.
\end{lemma}

\begin{proof}
   We proceed by using the argument in the proof of \cite[Lem.~7.2]{CL19}. Let $\Gamma(N)\subseteq\Gamma$ for some $N\in A$. Then there are infinitely many elements $\gamma\in\Gamma(N)$ so that $c\neq 0$ and hence we have infinitely many values $\mu_\gamma:=d_{\gamma}/c_{\gamma}\in K$. Assume to the contrary that there exist $z_0\in \Omega$ and $0\leq m \leq r$ such that $f_m(z_0)\neq 0$. Note that if $c_{\gamma}\neq 0$, then \eqref{Eq:QM_Uniqueness} is equivalent to
    \[
        (z_0+\mu_\gamma)^rf_0(z_0)+(z_0+\mu_\gamma)^{r-1}f_1(z_0)+\cdots+f_r(z_0)=0,
    \]
    which implies that the polynomial
    \[
        \sum_{i=0}^r\Big(\sum_{j=i}^r\binom{j}{j-i}z_0^{j-i}f_{r-j}(z_0)\Big)X^i\in\mathcal{O}[X]
    \]
    vanishes at infinitely many $X=\mu_\gamma$. Thus, $\sum_{j=i}^r\binom{j}{j-i}z_0^{j-i}f_{r-j}(z_0)=0$ for any $0\leq i\leq r$ which implies that $f_i(z_0)=0$ for any $0\leq i\leq r$. This contradicts to our assumption and hence the proof of the first part of the lemma is completed. The second part is a simple consequence of the first part.
\end{proof}

The following proposition follows the same spirit of \cite[Lem.~2.5]{BP08}.

\begin{proposition}\label{Prop:QM_QM}
    Let $\mathscr{F}$ be an element in $\mathcal{Q}\mathcal{M}_{k}^{m,\leq r}(\Gamma)$ such that there are $f_0,\dots,f_r\in\mathcal{O}$ satisfying
    \[
        (\mathscr{F}\mid_{k,m}\gamma)(z)=\sum_{i=0}^rf_i(z)\mathscr{J}(\gamma;z)^i
    \]
    for any $z\in\Omega$ and $\gamma\in\Gamma$. Then we have 
    \begin{equation}\label{Eq:QM_Coefficients}
        (f_i\mid_{k-2i,m-i}\gamma)(z)=\sum_{j=0}^{r-i}\binom{i+j}{j}f_{i+j}(z)\mathscr{J}(\gamma;z)^j.
    \end{equation}
    In particular, $f_i\in\QM_{k-2i}^{m-i,\leq r-i}(\Gamma)$ for $0\leq i\leq r-1$ and $f_r\in \mathcal{M}_{k-2r}^{m-r}(\Gamma)$.
\end{proposition}

\begin{proof}
    Let $\gamma_1,\gamma_2\in\Gamma$. Then on the one hand, we have
    \begin{equation}\label{Eq:QM_L}
        (\mathscr{F}\mid_{k,m}\gamma_1\gamma_2)(z)=\sum_{i=0}^rf_i(z)\mathscr{J}(\gamma_1\gamma_2;z)^i.
    \end{equation}
    On the other hand, by \cite[(1.6)]{BBP21}, we have
    \begin{equation}\label{Eq:QM_R}
            (\mathscr{F}\mid_{k,m}\gamma_1\gamma_2)(z)=\left((\mathscr{F}\mid_{k,m}\gamma_1)\mid_{k,m}\gamma_2\right)(z)
        =\sum_{i=0}^r(f_i\mid_{k,m}\gamma_2)(z)\mathscr{J}(\gamma_1;\gamma_2\cdot z)^i.
    \end{equation}
    We claim that 
     \begin{equation}\label{E:Claim}
     \mathscr{J}(\gamma_1;\gamma_2\cdot z)=\det(\gamma_2)^{-1}j(\gamma_2;z)^2\left(\mathscr{J}(\gamma_1\gamma_2;z)-\mathscr{J}(\gamma_2;z)\right).
     \end{equation}
To prove it, let $\gamma_1=\begin{pmatrix}
            a_1 & b_1\\
            c_1 & d_1
        \end{pmatrix}$ and $\gamma_2=\begin{pmatrix}
            a_2 & b_2\\
            c_2 & d_2
        \end{pmatrix}$.
        By definition, we have
        \begin{align*}
            \mathscr{J}(\gamma_1;\gamma_2\cdot z)=\frac{c_1}{c_1\left(\frac{a_2z+b_2}{c_2z+d_2}\right)+d_1}
            =\frac{c_1(c_2z+d_2)}{c_1(a_2z+b_2)+d_1(c_2z+d_2)}
            =\frac{c_1j(\gamma_2;z)}{j(\gamma_1\gamma_2;z)}.
        \end{align*}
Moreover, we have
        \begin{align*}
            \mathscr{J}(\gamma_1\gamma_2;z)-\mathscr{J}(\gamma_2;z)&=\frac{c_1a_2+d_1c_2}{(c_1a_2+d_1c_2)z+(c_1b_2+d_1d_2)}-\frac{c_2}{c_2z+d_2}\\
            &=\frac{(c_1a_2+d_1c_2)(c_2z+d_2)-c_2\left((c_1a_2+d_1c_2)z+(c_1b_2+d_1d_2)\right)}{j(\gamma_1\gamma_2;z)j(\gamma_2;z)}\\
            &=\frac{c_1\det\gamma_2}{j(\gamma_1\gamma_2;z)j(\gamma_2;z)}.
        \end{align*}
        Now the desired claim follows immediately. After applying \eqref{E:Claim}, \eqref{Eq:QM_R} becomes
    \begin{equation}\label{Eq:QM_R2}
        \begin{split}
            (\mathscr{F}\mid_{k,m}\gamma_1\gamma_2)(z)&=\sum_{i=0}^r(f_i\mid_{k-2i,m-i}\gamma_2)(z)(\mathscr{J}(\gamma_1\gamma_2;z)-\mathscr{J}(\gamma_2;z))^i\\
            &=\sum_{i=0}^r(f_i\mid_{k-2i,m-i}\gamma_2)(z)\sum_{j=0}^i(-1)^{i-j}\binom{i}{j}\mathscr{J}(\gamma_1\gamma_2;z)^j\mathscr{J}(\gamma_2;z)^{i-j}\\
            &=\sum_{i=0}^r\left(\sum_{j=i}^r(f_j\mid_{k-2j,m-j}\gamma_2)(z)(-1)^{j-i}\binom{j}{i}\mathscr{J}(\gamma_2;z)^{j-i}\right)\mathscr{J}(\gamma_1\gamma_2;z)^i.
        \end{split}
    \end{equation}
    Then by equating \eqref{Eq:QM_L} and \eqref{Eq:QM_R2}, Lemma~\ref{Lem:Uniqueness_of_QM} implies that
    \begin{equation}\label{Eq:QM_R3}
        f_i(z)=\sum_{j=i}^r(f_j\mid_{k-2j,m-j}\gamma_2)(z)(-1)^{j-i}\binom{j}{i}\mathscr{J}(\gamma_2;z)^{j-i}.
    \end{equation}
    By \cite[Lem.~3.1.3]{Bas14}, we have $j(\gamma_2;\gamma_2^{-1}\cdot z)=j(\gamma_2^{-1};z)^{-1}$ and, by \eqref{E:Claim},  
    \[
    \mathscr{J}(\gamma_2;\gamma_2^{-1}\cdot z)=-\det(\gamma_2)j(\gamma_2^{-1};z)^2\mathscr{J}(\gamma_2^{-1};z).
        \] 
    Thus, using \eqref{Eq:QM_R3}, we obtain
    \[
        f_i(\gamma_2^{-1}\cdot z)=\det(\gamma_2^{-1})^{i-m}j(\gamma_2^{-1};z)^{k-2i}\sum_{j=i}^r\binom{j}{i}f_j(z)\mathscr{J}(\gamma_2^{-1};z)^{j-i},
    \]
which concludes the desired result immediately.
\end{proof}

Let us consider the $\mathbb{C}_{\infty}$-vector space 
\[
\QM_{k}^{m}(\Gamma):=\bigcup_{r\in \mathbb{Z}_{\geq 0}}\QM_{k}^{m,\leq r}(\Gamma).
\]
Furthermore, we define $\QM(\Gamma)$ to be the $\mathbb{C}_{\infty}$-algebra generated by all the Drinfeld quasi-modular forms for $\Gamma$. Using Lemma \ref{Lem:Uniqueness_of_QM} and Proposition \ref{Prop:QM_QM} as well as the idea in the proof of Theorem \ref{Thm:Structure_of_N}, one can easily prove the following decomposition of $\mathcal{Q}\mathcal{M}_k^{m,\leq r}(\Gamma)$ and we leave the details to the reader.

\begin{proposition}\label{P:decomp}  We have 
	\[
	\mathcal{Q}\mathcal{M}_{k}^{m,\leq r}(\Gamma)=\bigoplus_{0\leq j \leq r}\mathcal{M}_{k-2j}^{m-j}(\Gamma)E^j.
	\]
\end{proposition}

In what follows, we generalize \cite[Thm.1, Prop. 2.2]{BP08} for some particular congruence subgroups of $\Gamma(1)$. Firstly, for some monic polynomial $\mathfrak{p}\in A$ of positive degree, define 
\[ \Gamma_1(\mathfrak{p}):=\left\{ \begin{pmatrix}
a & b\\c & d
\end{pmatrix} \in \Gamma(1)~| ~ c\in \mathfrak{p}A, \ \ d-1\in \mathfrak{p}A  \right\}.
\]
Let $\mathcal{S}^{\text{grd}}$ be the set of all congruence subgroups $\Gamma$ of $\Gamma(1)$ such that $\Gamma_1(\mathfrak{p})\subseteq \Gamma$ for some $\mathfrak{p}$.

\begin{proposition}\label{P:Grad}
	\begin{itemize}
		\item[(i)] Let $f_1,\dots, f_n\in \mathcal{Q}\mathcal{M}(\Gamma)$ be such that their weights are pairwise distinct. Then the set $\{f_1,\dots, f_n\}$ is $\mathbb{C}_{\infty}$-linearly independent.
		\item[(ii)] The function $E$ is transcendental over $\mathcal{M}(\Gamma)$. Moreover, $\QM(\Gamma)$ is a finitely generated $\mathbb{C}_{\infty}$-algebra.
		\item[(iii)] If $\Gamma\in \mathcal{S}^{\text{grd}}$, then 
		\[
		\QM(\Gamma)=\bigoplus_{\substack{k\in \mathbb{Z}_{\geq 0}\\m\in \mathbb{Z}/(q-1)\mathbb{Z}}}\QM_{k}^{m}(\Gamma).
		\]
		In particular, $\QM(\Gamma)$ is a $\mathbb{C}_{\infty}$-algebra graded by the weights and types, and filtered by the depths.
	\end{itemize}
\end{proposition}
\begin{proof} 
	We start with proving the first assertion. Let us assume that there exist $c_1,\dots,c_n\in \mathbb{C}_{\infty}$ not all zero and $f_{\nu}\in \QM_{k_\nu}^{m_\nu,\leq r_\nu}(\Gamma)$ for each $\nu\in \{1,\dots,n\}$ satisfying
	\begin{equation}\label{E:Gr0}
	(c_1f_1+\dots+c_nf_n)(z)=0, \ \ z\in\Omega.
	\end{equation}
	Let $\mathfrak{m}$ be a monic polynomial of positive degree such that $\Gamma(\mathfrak{m})\subseteq \Gamma$. For each $\alpha\in A$, we set 
	\[
	\gamma_{\alpha}:=\begin{pmatrix} 1+\alpha \mathfrak{m}& -\alpha^2\mathfrak{m}\\
	\mathfrak{m} & 1-\alpha \mathfrak{m}
	\end{pmatrix}\in \Gamma(\mathfrak{m}).
	\]
	Let $z_0$ be an element in $\Omega$ satisfying $f_1(z_0)\neq 0$. Replacing $z$ with $\gamma_{\alpha}\cdot z_0$ in \eqref{E:Gr0}, we obtain
	\begin{multline}\label{E:Gr1}
	c_1f_1(z_0) (\mathfrak{m}(z_0-\alpha)+1)^{k_1}+\sum_{i=1}^{r_1}c_1f_{1,i}(z_0)\mathfrak{m}^{i}(\mathfrak{m}(z_0-\alpha)+1)^{k_1-i}\\+\dots+ \sum_{i=0}^{r_n}c_nf_{n,i}(z_0)\mathfrak{m}^{i}(\mathfrak{m}(z_0-\alpha)+1)^{k_n-i}=0
	\end{multline}
	for some $f_{\nu,j}\in \mathcal{O}$ where $j\in \{1,\dots,r_n\}$ and $f_{\nu,0}:=f_{\nu}$ when $\nu\geq 2$. By reordering if necessary, we further assume that $c_1\neq 0$ and $k_1>\dots>k_n$. Note that, by Corollary \ref{C:depth} and Proposition \ref{P:decomp}, $k_{\nu}\geq 2 r_{\nu}$ for each $\nu$. Moreover, \eqref{E:Gr1} holds for any $\alpha\in A$. Thus, the polynomial
	\[
	\mathcal{P}(X):=c_1f_1(z_0)X^{k_1} +\sum_{i=1}^{r_1}c_1f_{1,i}(z_0)\mathfrak{m}^{i}X^{k_1-i}+\dots+ \sum_{i=0}^{r_n}c_nf_{n,i}(z_0)\mathfrak{m}^{i}X^{k_n-i}
	\]
	has infinitely many zeros of the form $\mathfrak{m}(z_0-\alpha)+1$. This implies that $c_1f_1(z_0)=0$, which is a contradiction to our assumption and hence $c_1=0$. Repeating this argument for other coefficients conclude the proof of the first part.  
	
	We now prove the first assertion of (ii). Let $P(X)\in \mathcal{M}(\Gamma)[X]$ be such that $P(E)=0$. For $w\in \mathbb{Z}_{\geq 0}$, let $P_{w}(X):=\sum_{i=0}^{\tilde{w}}\tilde{g}_iX^i\in \mathcal{M}(\Gamma)[X]$ be so that $P_w(E)\in \QM_{w}^{m}(\Gamma)$ for some $m\in \mathbb{Z}/(q-1)\mathbb{Z}$ and $P(E)=\sum_{w\in \mathbb{Z}_{\geq 0}}P_w(E)=0$. Thus, by the first part, $P_w(E)=0$ for each $w$. Hence, we have 
	\begin{equation}\label{E:transcendence}
	\tilde{g}_{\tilde{w}}E^{\tilde{w}}=\sum_{i=0}^{\tilde{w}-1}\tilde{g}_iE^i,
	\end{equation}
	where both sides are elements of $\QM(\Gamma)$. Note that the left hand side of \eqref{E:transcendence} has depth $\tilde{w}$ whereas the right hand side has depth less than $\tilde{w}$. Applying $\gamma \cdot z$ to both sides of \eqref{E:transcendence} for any $\gamma \in \Gamma$ and all $z\in \Omega$, we obtain $\tilde{g}_i\equiv 0$ for each $i$ by Lemma \ref{Lem:Uniqueness_of_QM}. Since $w$ is arbitrary, we indeed have $P(X)=0$, completing the proof. The second assertion is a consequence of Proposition \ref{P:decomp} and \cite[Cor. 1.58]{Gos80} (see also \cite[Thm. 11.1(c)]{BBP21}). 
	
	To prove (iii), firstly, note that for any $k_1,k_2\in \mathbb{Z}_{\geq 0}$, $m_1,m_2\in \mathbb{Z}/(q-1)\mathbb{Z}$ and $r_1,r_2\in \mathbb{Z}_{\geq 0}$, we have 
	\[
	\QM_{k_1}^{m_1,\leq r_1}(\Gamma)\QM_{k_2}^{m_2,\leq r_2}(\Gamma)\subseteq \QM_{k_1+k_2}^{m_1+m_2,\leq r_1+r_2}(\Gamma).
	\]
	To conclude the proof of the second part, by the first assertion, it suffices to show that Drinfeld quasi-modular forms of same weight for $\Gamma$ whose types are pairwise distinct are $\mathbb{C}_{\infty}$-linearly independent. Assume to the contrary that there exist $\tilde{c}_1,\dots,\tilde{c}_{\mu}\in \mathbb{C}_{\infty}$ not all zero and, for each $1\leq s \leq \mu$, $g_{s}\in \QM_{k}^{m_s,\leq r_s}(\Gamma)$ with $m_i\neq m_j$ if $i\neq j$ where $i,j\in \{1,\dots,\mu\}$, satisfying
	\begin{equation}\label{E:Gr2}
	(\tilde{c}_1g_1+\dots+\tilde{c}_{\mu}g_\mu)(z)=0, \ \ z\in \Omega.
	\end{equation}  
	By reordering if necessary, we further assume that $\tilde{c}_1\neq 0$ and choose $z'\in \Omega$ such that $g_1(z')\neq 0$. Replacing $z$ in \eqref{E:Gr2} with $\begin{pmatrix} a & 0\\
	0 & 1
	\end{pmatrix}\cdot z'$ for any $a\in \mathbb{F}_q^{\times}$, we obtain 
	\begin{equation}\label{E:Gr3}
	\tilde{c}_1g_1(z')a^{m_1}+\dots+\tilde{c}_{\mu}g_\mu(z')a^{m_{\mu}}=0.
	\end{equation}
	However, \eqref{E:Gr3} implies that the polynomial 
	$
	\tilde{\mathcal{P}}(X):=\tilde{c}_1g_1(z')X^{m_1}+\dots+\tilde{c}_{\mu}g_\mu(z')X^{m_{\mu}}
	$
	of degree at most $q-2$ has $q-1$-many distinct roots. This implies that $\tilde{\mathcal{P}}(X)=0$ and hence $\tilde{c}_1g_1(z')=0$, which is a contradiction. Thus the proof of the second assertion is completed.
\end{proof}

The following theorem gives an analogue of \cite[Prop.~1~(a)]{KZ95} in our setting.

\begin{theorem}\label{T:Isom}
    The map
    \begin{align*}
        \iota:\mathcal{N}_{k}^{m,\leq s}(\Gamma)&\to \mathcal{Q}\mathcal{M}_{k}^{m,\leq s}(\Gamma)\\
        F=\sum_{i=0}^s\frac{f_i}{(\tilde{\pi}\Id-\tilde{\pi}\psi)^{i}}&\mapsto\iota(F):=f_0
    \end{align*}
    is a well-defined isomorphism of $\mathbb{C}_{\infty}$-vector spaces so that 
    \[
	\iota(\delta_{k}^r(F))=\der^r(\iota(F)).
    \] 
    In particular, we have $\der^r(\mathcal{Q}\mathcal{M}_{k}^{m,\leq s}(\Gamma))\subset\mathcal{Q}\mathcal{M}_{k+2r}^{m+r,\leq s+r}(\Gamma)$ and the following diagram of $\mathbb{C}_\infty$-vector spaces commutes:
    \[
        \begin{tikzcd}
            \mathcal{N}_k^{m,\leq s}(\Gamma) \arrow[d, "\delta_k^r"] \arrow[r, "\iota"] & \mathcal{Q}\mathcal{M}_k^{m,\leq s}(\Gamma)\arrow[d, "\der^r"]\\
            \mathcal{N}_{k+2r}^{m+r,\leq s+r}(\Gamma) \arrow[r, "\iota"] & \mathcal{Q}\mathcal{M}_{k+2r}^{m+r,\leq s+r}(\Gamma)
        \end{tikzcd}.
    \]
\end{theorem}
\begin{proof} Let $F=\sum_{i=0}^s\frac{f_i}{(\tilde{\pi}\Id-\tilde{\pi}\psi)^{i}}\in \mathcal{N}_{k}^{m,\leq s}(\Gamma)$ with uniquely defined $f_i\in \mathcal{O}$. By Remark \ref{R:1}(ii) and Proposition \ref{P:2}, that $f_{0} \in\QM_k^{m,\leq s}(\Gamma)$ and thus $\iota$ is a well-defined map. Clearly, the map $\iota$ is $\CC_{\infty}$-linear. To prove the injectivity,
note that for any $\gamma \in \Gamma$, we have	
\[
f_{0}(\gamma \cdot  z)=j(\gamma;z)^{k}\det(\gamma)^{-m}\left(f_{0}(z)+\mathscr{J}(\gamma;z)f_{1}(z)+\dots +\mathscr{J}(\gamma;z)^{s}f_{s}(z)\right).
\]
Thus, if $\iota(F)=0$ then, by  Proposition \ref{P:2} and Lemma \ref{Lem:Uniqueness_of_QM}, $F$ vanishes identically and hence $\iota$ is injective. We now prove the surjectivity. 
Let $\mathscr{F}\in\QM_k^{m,\leq s}(\Gamma)$. Then there exist $g_0,\dots,g_s\in\mathcal{O}$ satisfying
\[
(\mathscr{F}\mid_{k,m}\gamma)(z)=\sum_{i=0}^sg_i(z)\mathscr{J}(\gamma;z)^i
\]
for any $z\in\Omega$ and $\gamma\in\Gamma$. Let $F:\Omega^{\psi}(M)\to \mathbb{C}_{\infty}$ be the function given by $F(z):=\sum_{i=0}^sg_i(z)\frac{1}{(\tilde{\pi}z-\tilde{\pi}\psi(z))^{i}}$. We first claim that
\begin{equation}\label{E:claim}
(F\mid_{k,m}\gamma)(z)=F(z).
\end{equation}
 Applying Proposition~\ref{P:1} and Proposition~\ref{Prop:QM_QM}, we obtain
\begin{align*}
(F\mid_{k,m}\gamma)(z)&=\sum_{i=0}^s(g_i\mid_{k,m}\gamma)(z)\left(\frac{1}{\tilde{\pi}z-\tilde{\pi}\psi(z)}-\mathscr{J}(\gamma;z)\right)^i\\
&=\sum_{i=0}^s\sum_{j=0}^{s-i}\binom{i+j}{j}g_{i+j}(z)\mathscr{J}(\gamma;z)^j\sum_{\ell=0}^i(-1)^{i-\ell}\binom{i}{\ell}\mathscr{J}(\gamma;z)^{i-\ell}\frac{1}{(\tilde{\pi}z-\tilde{\pi}\psi(z))^{\ell}}\\
&=\sum_{\ell=0}^s\left(\sum_{i=\ell}^s\sum_{j=0}^{s-i}(-1)^{i-\ell}\binom{i+j}{j}\binom{i}{\ell}g_{i+j}(z)\mathscr{J}(\gamma;z)^{i+j-\ell}\right)\frac{1}{(\tilde{\pi}z-\tilde{\pi}\psi(z))^{\ell}}\\
&=\sum_{\ell=0}^s\left(\sum_{i=\ell}^s\left(\sum_{j=\ell}^{i}(-1)^{j-\ell}\binom{i}{i-j}\binom{j}{\ell}\right)g_{i}(z)\mathscr{J}(\gamma;z)^{i-\ell}\right)\frac{1}{(\tilde{\pi}z-\tilde{\pi}\psi(z))^{\ell}}\\
&=\sum_{\ell=0}^sg_{\ell}(z)\frac{1}{(\tilde{\pi}z-\tilde{\pi}\psi(z))^{\ell}}\\
&=F(z),
\end{align*}
by noting that
\[
\sum_{j=\ell}^{i}(-1)^{j-\ell}\binom{i}{i-j}\binom{j}{\ell}=\sum_{m=0}^{i-\ell}(-1)^{m}\binom{i}{i-m-\ell}\binom{m+\ell}{\ell}=\binom{i}{\ell}\sum_{m=0}^{i-\ell}(-1)^{m}\binom{i-\ell}{m}=0
\]    
unless $i=\ell$. The desired claim now follows. On the other hand, for any $\alpha\in \Gamma(1)$, by the assumption on $g_i$, $g_i|_{k-2i,m-i}\alpha$ is bounded on vertical lines and thus, by Remark \ref{R:1}(ii), so is $\mathfrak{C}_{i,F}$ for each $0\leq i \leq s$. Hence, by \cite[Prop. 5.16]{BBP21} and Proposition \ref{Prop:QM_QM}, we have
\[
	(F|_{k,m}\alpha)(z)=\sum_{i=0}^s\frac{1}{(\tilde{\pi}z-\tilde{\pi}\psi(z))^{i}}\sum_{\ell= 0}^{\infty}c_{\alpha,i,\ell}u_{m_{\alpha^{-1}\Gamma\alpha}}(z)^{\ell}
\]
for some $c_{\alpha,i,\ell}\in \mathbb{C}_{\infty}$. Therefore, $F\in \mathcal{N}_k^{m,\leq s}(\Gamma)$. Moreover, since $\iota(F)=g_0=\mathscr{F}$, the proof of surjectivity is completed.

Finally, it is straightforward to check from the definition that
\[
    \iota(\delta_{k}^r(F))=\iota\bigg(\sum_{i=0}^r\delta_{k}^r\left(\frac{f_i}{\left(\tilde{\pi}\Id-\tilde{\pi}\psi\right)^{i}}\right)\bigg)=\der^r(f_0)=\der^r(\iota(F)).
\]
Then the desired inclusion $\der^r(\mathcal{Q}\mathcal{M}_{k}^{m,\leq s}(\Gamma))\subset\mathcal{Q}\mathcal{M}_{k+2r}^{m+r,\leq s+r}(\Gamma)$ and the commutative diagram follows immediately.
\end{proof}

\begin{remark} When $\Gamma=\Gamma(1)$, we remark that, for any $r\geq 0$, Bosser and Pellarin \cite[Thm. 2]{BP08} also showed that  $\der^r(\mathcal{Q}\mathcal{M}_{k}^{m,\leq s}(\Gamma))\subseteq\mathcal{Q}\mathcal{M}_{k+2r}^{m+r,\leq s+r}(\Gamma)$ .
\end{remark}

Recall the $\mathbb{C}_{\infty}$-vector space $\mathcal{N}_{k}^{m}(\Gamma)$ and the $\mathbb{C}_{\infty}$-algebra $\mathcal{N}(\Gamma)$ from \S1.2.  We finish this section with the following result whose proof easily follows from Proposition \ref{P:Grad} and Theorem \ref{T:Isom}. 
\begin{corollary}\label{C:transE}
	\begin{itemize}
		\item[(i)] Let $F_1,\dots, F_n\in \mathcal{N}(\Gamma)$ so that their weights are pairwise distinct. Then the set $\{F_1,\dots, F_n\}$ is $\mathbb{C}_{\infty}$-linearly independent.
		\item[(ii)]  The function $E_2$ is transcendental over $\mathcal{M}(\Gamma)$. Moreover, $\mathcal{N}(\Gamma)$ is a finitely generated $\mathbb{C}_{\infty}$-algebra.
		\item[(iii)]  If $\Gamma\in \mathcal{S}^{\text{grd}}$, then 
		\[
		\mathcal{N}(\Gamma)=\bigoplus_{\substack{k\in \mathbb{Z}_{\geq 0}\\m\in \mathbb{Z}/(q-1)\mathbb{Z}}}\mathcal{N}_{k}^{m}(\Gamma).
		\]
	\end{itemize}
\end{corollary}

\section{Special values of nearly holomorphic Drinfeld modular forms at CM points}
In the present section, our goal is to prove Theorem \ref{Thm:ClassField_at_CM_points} and Theorem \ref{Thm:Rationality_at_CM_points} as well as Corollary \ref{Cor:Rationality_at_CM_points2}. We first introduce necessary background to prove our first result in the following section which will be also in use \S 6.2.

Throughout this section, we let $\Gamma\leqslant \Gamma(1)$ be a congruence subgroup.
\subsection{Drinfeld modular functions} We start with defining \textit{the $j$-function $\mathfrak{j}:\Omega\to \mathbb{C}_{\infty}$} by 
\[
\mathfrak{j}(z):=\frac{\mathfrak{g}(z)^{q+1}}{\Delta(z)}.
\]
Note that $\mathfrak{j}\in \mathcal{A}_0(K)$. Let $N\in A$ be a monic polynomial of positive degree. Recall from \S1.5 that $K_N$ is  the $N$-th Carlitz cyclotomic field.  Let $u=(u_1,u_2)\in (N^{-1}A/A)^2\setminus\{(0,0)\}$ and recall the holomorphic function function $E_u$ from \S2.2. We further define the following Drinfeld modular functions for $\Gamma(N)$ by
\[
\mathbf{f}_{1,u}(z):=\frac{\Tilde{\pi}^{1-q}\mathfrak{g}(z)}{E_u(z)^{q-1}}\ \ \text{and} \ \ ~\mathbf{f}_{2,u}(z):=\frac{\Tilde{\pi}^{1-q^2}\Delta(z)}{E_u(z)^{q^2-1}}.
\]
Note that $\mathbf{f}_{2,u}(z)=\mathbf{f}_{1,u}(z)^{q+1}/\mathfrak{j}(z)$. 
Moreover, by \cite[Chap. VII]{Gek86}, the $u_N$-expansion coefficients of $\mathbf{f}_{1,u}$ and $\mathbf{f}_{2,u}$ lie in $K_N$.

Let $X(N)$ be the smooth projective model of the affine algebraic curve $\Gamma(N)\setminus\Omega$. Note that  $\mathbb{C}_\infty(X(1))=\mathbb{C}_\infty(\mathfrak{j})$ and  we have 
\[
\mathbb{C}_\infty(X(N))=\mathbb{C}_\infty(\mathfrak{j})(\mathbf{f}_{1,u}\mid u=(u_1,u_2)\in(N^{-1}A/A)^2\setminus\{(0,0)\}).
\]
Let $L$ be a CM field and $z_0\in L$. Consider the field
\begin{equation}\label{Eq:Maximal_Abelian_Extension}
L_{N}:=L(\mathfrak{j}(z_0),\mathbf{f}_{1,u}(z_0),~\mathbf{f}_{2,u}(z_0)\mid u=(u_1,u_2)\in(N^{-1}A/A)^2\setminus\{(0,0)\}).
\end{equation}
 Then by \cite[Cor.~4.5]{Gek83}, $L_N$ is an abelian extension of $L$ for which $\infty$ splits completely.

We continue with the following lemma on the values of Drinfeld modular functions at CM points.

\begin{lemma}\label{L:Algebraic_Scalar} Let $K\subseteq H \subseteq \mathbb{C}_{\infty}$ be a field and $H_N$ be the compositum of $H$ and $K_N$. Then for any  $f\in \mathcal{A}_0(\Gamma(N),H)$ and $z_0\in \Omega$ such that $f(z_0)$ is defined, we have 
    \begin{equation}\label{E:modfunct}
        f(z_0)\in H_N(\mathfrak{j}(z_0),\mathbf{f}_{1,u}(z_0)\mid u=(u_1,u_2)\in(N^{-1}A/A)^2\setminus\{(0,0)\}).
     \end{equation}
\end{lemma}

\begin{proof}
  One can find polynomials $G_1,G_2\in\mathbb{C}_\infty[X,X_u]_{u=(u_1,u_2)\in(N^{-1}A/A)^2\setminus\{(0,0)\}}$ so that
	\begin{equation}\label{Eq:Linearly_Disjoint}
	f=\frac{G_1}{G_2}\Big|_{X=\mathfrak{j},X_u=\mathbf{f}_{1,u}}
	\end{equation}
	for any $f\in \mathcal{A}_0(\Gamma(N),H)$.
	By comparing the $u_{m_{\Gamma(N)}}$-expansion on the both sides of \eqref{Eq:Linearly_Disjoint}, we obtain $G_1,G_2\in  H_N[X_j,X_u]_{u=(u_1,u_2)\in(N^{-1}A/A)^2\setminus\{(0,0)\}}$ and thus $f(z_0)$  is contained in the right hand side of \eqref{E:modfunct}, finishing the proof.
\end{proof}

Before we state our next proposition, we observe the following decomposition for $\GL_2(K)$.
\begin{remark} \label{R:decompose}  Define
\[
\mathfrak{M}:=\left\{\begin{pmatrix} a'&b'\\
c'&d'
\end{pmatrix}\in \GL_2(K) \ \ | \ \ c'=0    \right\}.
\]
Then, following the idea in \cite[Sec. 1.3]{Shi07}, one can see that 
$
\GL_2(K)=\Gamma(1)\mathfrak{M}.
$
Moreover, let 
\[
U:=\left\{\begin{pmatrix}1&\beta\\0&1
\end{pmatrix}\ \ | \ \beta\in K \right\}  \text{ and } V:=\left\{\begin{pmatrix}y&0\\0&\xi
\end{pmatrix}\ \ | \ y,\xi\in K^{\times} \right\}.
\]
It can be easily seen that any element $\alpha\in \mathfrak{M}$ can be written as $\alpha=\mathfrak{u}'\mathfrak{v}'$ where $\mathfrak{u}'\in U$ and $\mathfrak{v}'\in V$. 
\end{remark}

\subsection{Algebraic special values} Our goal in section is to prove Theorem \ref{T:D}. First we state our next theorem whose proof will be given in Appendix A.

\begin{theorem}\label{T:CM} For any CM point $z_0$, there exists a field $M_{z_0}$ over $\widehat{K_{\infty}^{\text{nr}}}$ containing $z_0$ and an extension $\psi_{z_0}$ of $\sigma$.
\end{theorem}

Throughout this section, we fix a CM point $z_0\in \Omega$ lying in a CM field $L$ that is, $L=K(z_0)$. We further set $M:= M_{z_0}$ and $\psi:= \psi_{z_0}$.

We have the following injective map of groups
\begin{equation}\label{Eq:Embedding}
\begin{split}
\rho:L^\times&\hookrightarrow\GL_2(K)\\
\alpha&\mapsto\rho_\alpha:=\begin{pmatrix}
a & b\\
c & d
\end{pmatrix},
\end{split}
\end{equation}
where $\alpha z_0=az_0+b$ and $\alpha =cz_0+d$ for some $a,b,c,d\in K$. Note, by \cite[Cor. 3.10]{Ham03}, that
$
\rho(L^\times)=\{\gamma\in\GL_2(K)\mid \gamma\cdot z_0=z_0\}.
$

\begin{remark}\label{R:CM} Using Hensel's Lemma and examples given in \cite[Ex.~3.4.8,~Ex.~3.4.9]{Pap23}, one can indeed show that $\Omega^{\sigma}(\widehat{K_{\infty}^{\text{nr}}})\subset \Omega^{\psi}(M)$ has infinitely many CM points. This statement is equivalent to say that there are infinitely many imaginary quadratic extensions over $K$ such that the infinite place of $K$ is inert.
\end{remark}

For an element $\alpha\in K$, we set $\mathrm{den}(\alpha)\in A$ to be the monic polynomial of the smallest degree so that $\mathrm{den}(\alpha)\alpha\in A$. Let $N\in A$ be a monic polynomial. We set $C[N]$ to be the collection of all the $N$-th torsion points of the Carlitz module. Furthermore, we denote by  $K_N^+$ the \textit{the real subfield of $K_N$} generated by $\lambda^{q-1}$ over $K$ for $\lambda=\exp_{C}\left(\frac{\tilde{\pi}a}{N}\right)$ for some $a,b\in A$ with $\deg_{\theta}(a)<\deg_{\theta}(N)$. Let $\Tr(z_0)$ and $\mathrm{Nr}(z_0)$ be the trace and norm of $z_0$ respectively. We set $\mathfrak{w}_{z_0}:=\mathrm{den}(\Tr(z_0))\mathrm{den}(\mathrm{Nr}(z_0))^2\mathrm{Nr}(z_0)\in A$ and $\mathfrak{r}_{z_0}:=\mathrm{den}(\Tr(z_0))\in A$. 

The following proposition plays a crucial role in the later study of special values of nearly holomorphic Drinfeld modular forms at CM points.

\begin{proposition}\label{P:Arith} 
    Let $\Delta\in \mathcal{M}_{q^2-1}^0(\Gamma(1);K)$ be the Drinfeld discriminant function. Then
    \[
        \Delta\mid_{q^2-1,0}\rho_{z_0}\in\mathcal{M}_{q^2-1}^0(\Gamma(\mathfrak{w}^3_{z_0});K_{\mathfrak{r}_{z_0}}^+).
    \]
\end{proposition}

\begin{proof} 
 Using the minimal polynomial of $z_0$ over $K$, observe that
    \[
        \rho_{z_0}=\begin{pmatrix}
            \Tr(z_0) & -\mathrm{Nr}(z_0)\\
            1 & 0
        \end{pmatrix}\in\GL_2(K).
    \]
    Since $\mathrm{den}(\Tr(z_0))$ is relatively prime to $\mathrm{den}(\Tr(z_0))\Tr(z_0)$, there exist $c,d\in A$ so that 
    \[
    \mathrm{den}(\Tr(z_0))c+\mathrm{den}(\Tr(z_0))\Tr(z_0)d=1.
    \]
    Then
    \begin{align*}
        \rho_{z_0}&=\begin{pmatrix}
            \mathrm{den}(\Tr(z_0))\Tr(z_0) & -c\\
            \mathrm{den}(\Tr(z_0)) & d
        \end{pmatrix}\begin{pmatrix}
            1 & -d/\mathrm{den}(\Tr(z_0))\\
            0 & 1
        \end{pmatrix}\begin{pmatrix}
            1/\mathrm{den}(\Tr(z_0)) & 0\\
            0 & \mathrm{den}(\Tr(z_0))\mathrm{Nr}(z_0)
        \end{pmatrix}.
    \end{align*}
To ease the notation, we let $\mathfrak{p}_{z_0}:=\mathrm{den}(\Tr(z_0))\mathrm{Nr}(z_0)$. Recall the sets $U$ and $V$ given in the proof of Proposition \ref{P:1}. Now we define
\[
    \gamma:=\begin{pmatrix}
        \mathrm{den}(\Tr(z_0))\Tr(z_0) & -c\\
        \mathrm{den}(\Tr(z_0)) & d
    \end{pmatrix},~\mathfrak{u}:=\begin{pmatrix}
            1 & -d/\mathfrak{r}_{z_0}\\
            0 & 1
    \end{pmatrix}\in U,~\mbox{and}~
    \mathfrak{v}:=\begin{pmatrix}
            1/\mathfrak{r}_{z_0} & 0\\
            0 & \mathfrak{p}_{z_0}
    \end{pmatrix}\in V.
\]

Since $\Delta\in\mathcal{M}_{q^2-1}^{0}(\Gamma(1);K)$ and $\gamma\in\Gamma(1)$, by \cite[(1.6)]{BBP21}, we have 
\begin{equation}\label{E:arith0}
\Delta|_{q^2-1,0}\gamma \mathfrak{u}\mathfrak{v}=\left(\Delta|_{q^2-1,0}\gamma\right)|_{q^2-1,0}\mathfrak{u}\mathfrak{v}=\Delta|_{q^2-1,0}\mathfrak{u}\mathfrak{v}=\left(\Delta|_{q^2-1,0}\mathfrak{u}\right)|_{q^2-1,0}\mathfrak{v}.
\end{equation}
By the product formula of $\Delta$ established in \cite{Gek85}, we can write
\[
    \Delta(z)=\sum_{n=1}^{\infty}\mathfrak{c}_{n(q-1)}u(z)^{n(q-1)}
\]
for some $\mathfrak{c}_{n(q-1)}\in K$ where $z$ is in a suitable neighborhood of infinity. By \cite[Lem. 5.11(c)]{BBP21}, we obtain
\begin{equation}\label{E:arith1}
\left(\Delta|_{q^2-1,0}\mathfrak{u}\right)(z)=\sum_{n=1}^{\infty}\left(\sum_{\mu\geq 0}\binom{(\mu-n)(q-1)}{n(q-1)}\mathfrak{c}_{(n-\mu)(q-1)}\exp_{C}(-d\tilde{\pi}/\mathfrak{r}_{z_0})^{\mu(q-1)}\right)u(z)^{n(q-1)}.
\end{equation}
Let 
\[
    \tilde{\mathfrak{c}}_{n(q-1)}:=\sum_{\mu\geq 0}\binom{(\mu-n)(q-1)}{n(q-1)}\mathfrak{c}_{(n-\mu)(q-1)}\exp_{C}(-d\tilde{\pi}/\mathfrak{r}_{z_0})^{\mu(q-1)}\in \mathbb{C}_{\infty}.
\]
Since $\exp_{C}(-d\tilde{\pi}/\mathfrak{r}_{z_0})^{q-1}\in K_{\mathfrak{r}_{z_0}}^+$, it follows that $\tilde{\mathfrak{c}}_{n(q-1)}\in K_{\mathfrak{r}_{z_0}}^+$.  On the other hand, by \cite[Lem. 5.7]{BBP21} and \eqref{E:arith1}, we have 
\begin{equation}\label{E:arith2}
\left(\Delta|_{q^2-1,0}\mathfrak{u}\right)|_{q^2-1,0}\mathfrak{v}=\sum_{n=1}^{\infty}\mathfrak{r}_{z_0}^{n}(\mathfrak{w}_{z_0}/\mathfrak{r}_{z_0})^{n-(q^2-1)}\tilde{\mathfrak{c}}_n\exp_{\tilde{\pi}\frac{\mathfrak{p}_{z_0}}{\mathfrak{r}_{z_0}}A}(\tilde{\pi}z)^{-n}
\end{equation}
where we set 
\[
\exp_{\tilde{\pi}\frac{\mathfrak{p}_{z_0}}{\mathfrak{r}_{z_0}}A}(\tilde{\pi}z):=z'\prod_{\lambda\in \tilde{\pi}\frac{\mathfrak{p}_{z_0}}{\mathfrak{r}_{z_0}}A\setminus\{0\}}\left(1-\frac{z'}{\lambda}\right), \ \ z'\in \mathbb{C}_{\infty}.
\]
By \cite[Prop. 2.3(b)]{BBP21} and the functional equation of $\exp_{C}$, there exists a polynomial $ \mathfrak{a}\in A$, such that 
    \begin{align*}
        \exp_{\tilde{\pi}\frac{\mathfrak{p}_{z_0}}{\mathfrak{r}_{z_0}}A}(\tilde{\pi}z)^{-1}=\frac{\mathfrak{r}_{z_0}}{\mathfrak{p}_{z_0}}\exp_C\left(\frac{\tilde{\pi}\mathfrak{a}z}{\mathfrak{w}_{z_0}^3}\right)^{-1}=\frac{\mathfrak{r}_{z_0}}{\mathfrak{p}_{z_0}}\frac{u_{\mathfrak{w}^{3}}(z)^{q^{\deg(\mathfrak{a})}}}{\mathfrak{C}_{\mathfrak{a}}(u_{\mathfrak{w}_{z_0}^3})}
    \end{align*}
where $\mathfrak{C}_{\mathfrak{a}}(X)\in A[X]$ given as in Example \ref{Ex:Drinfeld_Modular_Forms}. On the other hand, by the argument as in the proof of \cite[Lem. 4.1.1]{Bas14}, we have $\Gamma(\mathfrak{w}_{z_0}^3)\subset \rho_{z_0}^{-1}\Gamma(1)\rho_{z_0}$. Thus, by \cite[Prop. 6.6]{BBP21}, we obtain the desired fact.
\end{proof}

\begin{lemma}[{cf.~Shimura \cite[Lem. 12.3]{Shi07}}] \label{L:Lem:Shimura's_Lemma} Set $	
	\mathfrak{r}:= \frac{\Delta}{\Delta|_{q^2-1,0}\rho_{z_0}}\in \mathcal{A}_0(\Gamma(\mathfrak{w}^3_{z_0});K_{\mathfrak{r}_{z_0}}^+)
	$ and define the function $g_{\rho_{z_0}}:\Omega\to \mathbb{C}_{\infty}$ by 
	\[
	g_{\rho_{z_0}}:=-\frac{1}{\tilde{\pi}}\frac{\der \mathfrak{r}}{\mathfrak{r}}.
	\]
	Then $g_{\rho_{z_0}}\in \mathcal{M}_2^1(\Gamma(\mathfrak{w}^3_{z_0});K_{\mathfrak{r}_{z_0}}^+)$. Moreover, for any $z\in \Omega^{\psi}(M)$, we have
	\begin{equation}\label{E:newg0}
	g_{\rho_{z_0}}(z)=
	E_2|_{2,1}\rho_{z_0}(z)-E_2(z).
	\end{equation}
	 Furthermore, 
    \begin{equation}\label{Eq:newg0_eva}
        g_{\rho_{z_0}}(z_0)=\mathfrak{w}_{z_0}E_2(z_0)
    \end{equation}
    where $\mathfrak{w}_{z_0}:=\frac{\psi(z_0)-z_0}{z_0}  \in L^{\times}$.
\end{lemma}
\begin{proof} The first assertion follows from Corollary \ref{Cor:1} and Proposition \ref{P:mod}. To prove the second assertion, noting $\Delta=\Delta|_{q^2-1,0}\rho_{z_0}	\mathfrak{r}$, we observe by Lemma \ref{L:derivation} that 
	\begin{equation}\label{E:newg1}
	\delta_{q^2-1}(\Delta)=\delta_{q^2-1}(\Delta|_{q^2-1,0}\rho_{z_0}	\mathfrak{r})=\delta_{q^2-1}(\Delta|_{q^2-1,0}\rho_{z_0})	\mathfrak{r}+\Delta|_{q^2-1,0}\rho_{z_0}\tilde{\pi}^{-1}\der 	\mathfrak{r}.
	\end{equation}
	Dividing \eqref{E:newg1} by $\Delta=\Delta|_{q^2-1,0}\rho_{z_0}	\mathfrak{r}$ and noting $\delta_{q^2-1}\Delta=E_2\Delta$, we have
	\begin{equation}\label{E:newg2}
	E_2=\frac{\delta_{q^2-1}(\Delta|_{q^2-1,0}\rho_{z_0})}{\Delta|_{q^2-1,0}\rho_{z_0}}+\frac{1}{\tilde{\pi}}\frac{\der 	\mathfrak{r}}{	\mathfrak{r}}.
	\end{equation}
	On the other hand, by Proposition \ref{P:3}, we get
	\begin{equation}\label{E:newg3}
	E_2|_{2,1}\rho_{z_0}\Delta|_{q^2-1,0}\rho_{z_0}=(E_2\Delta)|_{q^2+1,1}\rho_{z_0}= \delta_{q^2-1}(\Delta)|_{q^2+1,1}\rho_{z_0}=\delta_{q^2-1}(\Delta|_{q^2-1,0}\rho_{z_0}).
	\end{equation}
	Now by combining \eqref{E:newg2} and \eqref{E:newg3}, we obtain \eqref{E:newg0}. Finally, since $\rho_{z_0}\cdot z_0=z_0$, $\det(\rho_{z_0})=\psi(z_0)z_0$ and $j(\rho_{z_0};z_0)=z_0$, we easily obtain \eqref{Eq:newg0_eva}.
\end{proof}

\begin{lemma}\label{L:Abelian} For any $N\in A\setminus \mathbb{F}_q$, the field $K_N\cdot L_N$ is an abelian extension of $L$.
\end{lemma}
\begin{proof} By \cite[Thm. 12.8]{Ros02}, $K_N$ is an abelian extension of $K$. Since $L$ is abelian over $K$, this implies that $K_N\cdot L$ is abelian over $L$.  Moreover, by \cite[(4.5)]{Gek83}, $L_N$ is also an abelian extension of $L$ and hence $K_N\cdot L_N$ is abelian over $L$. 
\end{proof}

Recall, from \S1, that $L^{\text{ab}}$ is the maximal abelian extension of $L$ in $\mathbb{C}_{\infty}$.

\begin{proposition}\label{Thm:Rationality} 
Set $\tilde{\mathfrak{l}}:=\lcm(\theta, \mathfrak{r}_{z_0})\in A$. Let $u=(u_1,u_2)\in(\tilde{\mathfrak{l}}^{-1}A/A)^2$. Then we have
	\[
	\frac{E_2(z_0)}{E_u^2(z_0)} \in K_{\tilde{\mathfrak{l}}}\cdot L_{\tilde{\mathfrak{l}}}\subset L^{\text{ab}}
	.
	\]
\end{proposition}

\begin{proof}
	 By \eqref{Eq:newg0_eva}, we have 
	\begin{equation}\label{E:newg4}
	\frac{E_2(z_0)}{E_u^2(z_0)}=\mathfrak{w}_{z_0}^{-1}\frac{g_{\rho_{z_0}}(z_0)}{E_u^2(z_0)}.
	\end{equation}
	By Lemma \ref{L:Lem:Shimura's_Lemma}, we see that $g_{\rho_{z_0}}/E_u^2$ is an element in $\mathcal{A}_{0}(K_{\tilde{l}})$. In particular, by Lemma \ref{L:Algebraic_Scalar},
	\[
	\frac{g_{\rho_{z_0}}(z_0)}{E_u^2(z_0)}\in K_{\tilde{\mathfrak{l}}}(\mathfrak{j}(z_0),\mathbf{f}_{1,u}(z_0)\mid u=(u_1,u_2)\in(\tilde{\mathfrak{l}}^{-1}A/A)^2\setminus\{(0,0)\})
	\]
 and hence the result follows from \eqref{Eq:Maximal_Abelian_Extension} and \eqref{E:newg4}. The last inclusion follows from Lemma \ref{L:Abelian}.
\end{proof}

We are now ready to prove the main tool to prove Theorem \ref{T:D}.

\begin{theorem}[{cf.~Shimura \cite[Thm. 12.2]{Shi07}}] \label{Thm:ClassField_at_CM_points}
    Let $\mathfrak{f}\in\mathcal{N}_k^{m,\leq r}(\Gamma;K_N)$ and $g\in\mathcal{M}_k^{m'}(\Gamma;K_{N'})$ for some $N, N'\in A$ so that $g(z_0)\neq 0$. Set $\mathfrak{l}:=\lcm(N,N',m_{\Gamma}, \tilde{\mathfrak{l}})\in A$. Then
	\[
	\frac{\mathfrak{f}(z_0)}{g(z_0)}\in K_{\mathfrak{l}}\cdot L_{\mathfrak{l}}\subset L^{\text{ab}}.
	\]
	In particular, if $f\in \mathcal{M}_{k}^{m}(\Gamma; K_N)$ and $g\in \mathcal{M}_{k+2r}^{m'}(\Gamma; K_{N'})$ so that $g(z_0)\neq 0$, then 
	\[
	\frac{(\delta^r_{k}f)(z_0)}{g(z_0)}\in K_{\mathfrak{l}}\cdot L_{\mathfrak{l}}.
	\]
\end{theorem}
\begin{proof}
	By Theorem \ref{Thm:Structure_of_N}, we have
	$
	\mathfrak{f}=\sum_{0\leq n\leq r}g_nE_2^n
	$
	where each $g_n\in \mathcal{M}_{k-2n}^{m-n}(\Gamma; K_N)$. Fix a non-zero $u=(u_1,u_2)\in(\mathfrak{l}^{-1}A/A)^2$. Then we have 
	\[
	\frac{\mathfrak{f}(z_0)}{g(z_0)}=\sum_{0\leq n\leq r}\left(\frac{g_n(z_0)E_u^{2n}(z_0)}{g(z_0)}\right)\left(\frac{E_2(z_0)}{E_u^2(z_0)}\right)^n.
	\]
	Observe that, by Lemma \ref{L:Algebraic_Scalar},  
	\[
	\frac{g_n(z_0)E_u^{2n}(z_0)}{g(z_0)}\in K_\mathfrak{l}(j(z_0), \mathbf{f}_{1,u}(z_0)\mid u=(u_1,u_2)\in(\mathfrak{l}^{-1}A/A)^2\setminus\{(0,0)\}).
	\]
	Thus, the desired fact follows from Proposition \ref{Thm:Rationality}. Note that the second assertion follows immediately from the first assertion and Proposition \ref{P:3}.
\end{proof}

\begin{proof}[{Proof of Theorem \ref{T:D}}] For any given CM point $z_0\in \Omega$, by Theorem \ref{T:CM}, we can define a field over $M_{z_0}$ over $\widehat{K_{\infty}^{\text{nr}}}$ and an extension $\psi_{z_0}$ of $\sigma$. Then the result follows from Theorem \ref{Thm:ClassField_at_CM_points}.
\end{proof}

\subsection{Transcendence of special values at CM points}
Let $\phi$ be a rank $2$ Drinfeld module corresponding to the $A$-lattice $\Lambda_\phi$. Then $\phi$ is called a \emph{CM Drinfeld module} if
\[
    R_\phi:=\{\alpha\in\mathbb{C}_\infty|\alpha\Lambda_\phi\subset\Lambda_\phi\}
\]
defines a quadratic $A$-order over $K$. Recall that $z_0\in\Omega^{\psi}(M)$ is a CM point. We set
\[
    \omega_{z_0}:=\begin{cases}
        \sqrt[q-1]{g(z_0)},~\mbox{if}~j(z_0)\neq 0\\
        \sqrt[q^2-1]{\Delta(z_0)},~\mbox{if}~j(z_0)=0
    \end{cases}.
\]
Then the Drinfeld module $\phi^{\omega_{z_0}\Lambda_{z_0}}$ corresponding to the lattice $\omega_{z_0}\Lambda_{z_0}$ is given by
\[
    \phi^{\omega_{z_0}\Lambda_{z_0}}_t=\omega_{z_0}\phi^{\Lambda_{z_0}}_t\omega_{z_0}^{-1}=\begin{cases}
        \theta+\tau+j(z_0)^{-1}\tau^2,~\mbox{if}~j(z_0)\neq 0\\
        \theta+\tau^2,~\mbox{if}~j(z_0)=0
    \end{cases}.
\]
Note that $\phi^{\omega_{z_0}\Lambda_{z_0}}$ is a CM Drinfeld module defined over $K(j(z_0))\subset\overline{K}$. In particular, the \emph{CM period} $\omega_{z_0}$ of $\phi^{\omega_{z_0}\Lambda_{z_0}}$ is transcendental over $K$ by \cite[Thm.~5.1]{Yu86}.

Let $N\in A\setminus\mathbb{F}_q$. For $u=(u_1,u_2)\in(N^{-1}A/A)^2\setminus\{(0,0)\}$, we have
\begin{equation}\label{Eq:E_u at CM pt}
    E_u(z_0)=\Tilde{\pi}^{-1}\exp_{\Lambda_{z_0}}(u_1z_0+u_2)^{-1}=\exp_{\omega_{z_0}\Lambda_{z_0}}(u_1\omega_{z_0}z_0+u_2\omega_{z_0})^{-1}\left(\frac{\omega_{z_0}}{\Tilde{\pi}}\right)\in L^{\text{ab}}\cdot\left(\frac{\omega_{z_0}}{\Tilde{\pi}}\right)
\end{equation}
by \cite[Thm.~9.2]{Hay79} (see also \cite[Thm.~5]{DG20}). We further set $F_{\omega_{z_0}\Lambda_{z_0}}:\mathbb{C}_{\infty}\to \mathbb{C}_{\infty}$ to be the unique entire function satisfying 
\[
    F_{\omega_{z_0}\Lambda_{z_0}}(\theta z)-\theta F_{\omega_{z_0}\Lambda_{z_0}}(z)=\exp_{\omega_{z_0}\Lambda_{z_0}}(z)^q, \ \ z\in \mathbb{C}_{\infty}.
\]
Before we state the main result of this section, for any $\beta_1, \beta_2\in\mathbb{C}_{\infty}$ with $\beta_2\neq 0$ and $K\subseteq H\subseteq \overline{K}$, we denote by $\beta_1 \sim_H \beta_2$ if the ratio $\beta_1/\beta_2\in H$.

\begin{theorem}\label{Thm:Rationality_at_CM_points}
Let $H$ be a field such that $K_{m_{\Gamma}}\subseteq H\subseteq\overline{K}$.
\begin{itemize}
	\item[(i)] For any $f\in \QM_{k}^{m}(\Gamma;H)$, there exists a homogeneous polynomial $M_f\in(H\cdot L^{\text{ab}})[X,Y]$ of degree $k+2r$ such that 
	\[
	(\delta^r_{k}f)(z_0)=M_f\left(\frac{w_{z_0}}{\tilde{\pi}},\frac{F_{\omega_{z_0}\Lambda_{z_0}}(z_0)}{\tilde{\pi}}\right).
	\]
	In particular, if $(\delta^r_{k}f)(z_0)\neq 0$, then it is transcendental over $K$.
\item[(ii)] 
Suppose that $\mathfrak{f}\in\mathcal{N}_{k}^{m,\leq r}(\Gamma;H)$. Then
$
    \mathfrak{f}(z_0) \sim_{H\cdot L^{\text{ab}}} \left(\frac{\omega_{z_0}}{\tilde{\pi}}\right)^{k}.
$
In particular, if $\mathfrak{f}(z_0)\neq 0$, then $\mathfrak{f}(z_0)$ is transcendental over $K$. Moreover, let $f\in \mathcal{M}_{k}^{m}(\Gamma;H)$. Then we have 
\[
    (\delta^r_{k}f)(z_0) \sim_{H\cdot L^{\text{ab}}} \left(\frac{\omega_{z_0}}{\tilde{\pi}}\right)^{k+2r}.
\]	

\end{itemize}

\end{theorem}

\begin{proof} Note, by Theorem \ref{T:Isom}, that for each $0\leq \ell \leq r$, $\der^{\ell}f$ is an element in $\QM_{k+2\ell}^{m+\ell,\leq s+\ell}(\Gamma;H)$. 
Observe that 
\[
\delta_{k}^r(f)=\sum_{\ell=0}^r\left(\der^{r-\ell}f\right)\frac{1}{(\tilde{\pi}\Id-\tilde{\pi}\psi)^{\ell}}=\sum_{\ell=0}^r\left(\der^{r-\ell}f\right)(E-E_2)^{\ell}=\sum_{i=0}^r\left(\sum_{\ell=0}^r(-1)^i\binom{\ell}{i}\left(\der^{r-\ell}f\right) E^{\ell-i}\right)E_2^i.
\]
Let 
\[
g_{i}:=(-1)^i\sum_{\ell=0}^r\binom{\ell}{i}\left(\der^{r-\ell}f\right) E^{\ell-i}\in \QM_{k+2r-2i}^{m+r-i,\leq s+r-i}(\Gamma;H).
\]
Thus, for some $u\in(m_{\Gamma}^{-1}A/A)^2\setminus\{(0,0)\}$, $g_iE_u^{2i}\in \QM_{k+2r}^{m+r-i,\leq s+r-i}(\Gamma;H)$ and by Proposition \ref{P:decomp}, $g_iE_u^{2i}=\sum_{j=0}^{s+r-i}g_{i,j}E^j$ for some uniquely defined $g_{i,j}\in \mathcal{M}_{k+2r-2j}^{m+r-i-j}(\Gamma;H)$. We claim that $g_{i,j}(z_0)\sim_{H\cdot \widetilde{L^{\text{ab}}}}\left(\frac{w_{z_0}}{\tilde{\pi}}\right)^{k+2r-2j}$. To see this, by Lemma~\ref{L:Algebraic_Scalar} and $K_{m_{\Gamma}}\subset H$, we have
\[
g_{i,j}(z_0)/E_u^{k+2r-2j}(z_0)\in H(\mathfrak{j}(z_0),\mathbf{f}_{1,u'}(z_0)\mid u'=(u_1',u_2')\in(m_{\Gamma}^{-1}A/A)^2\setminus\{(0,0)\}).
\]
Thus \eqref{Eq:E_u at CM pt} gives the desired claim.  On the other hand, by the same argument in the proof of \cite[Thm. 3.3(c)]{Cha12b} as well as the result in \cite[Thm.~7.10]{Gek89}, we have $E(z_0)^j\sim_{H\cdot \widetilde{L^{\text{ab}}}}\left(\frac{w_{z_0}}{\tilde{\pi}}\right)^{j}\left(\frac{F_{\varphi}(z_0)}{\tilde{\pi}}\right)^{j}$. This means that there exists a polynomial $M_i\in (H\cdot \widetilde{L^{\text{ab}}})[X,Y]$ homogeneous of degree $k+2r$ such that 
\begin{equation}\label{E:QMtrans}
g_i(z_0)E_u^{2i}(z_0)=M_i\left(\frac{w_{z_0}}{\tilde{\pi}},\frac{F_{\varphi}(z_0)}{\tilde{\pi}}\right).
\end{equation}
Hence the first assertion of the first part follows from  Theorem \ref{Thm:Rationality} and \eqref{E:QMtrans}. The last assertion of (i) is due to \cite{Thi92} (see also \cite[Thm. 3.1]{Cha12b}). By using Theorem \ref{Thm:Structure_of_N} and similar arguments as in the first part, one can immediately obtain the second assertion.
\end{proof}

As a consequence of Theorem \ref{Thm:Rationality_at_CM_points}, we have the following algebraic independence result for certain special values.

\begin{corollary} \label{Cor:Rationality_at_CM_points2}
    Let $\alpha_1,\dots,\alpha_\ell$ be CM points in $ \Omega^{\psi}(M)$. Then for any $k,r\in \ZZ_{\geq 0}$, $m\in \ZZ/(q-1)\ZZ$, and  $\mathfrak{f}\in\mathcal{N}_k^{m,\leq r}(\Gamma;\overline{K})$ such that $\mathfrak{f}(\alpha_i)\neq 0$ for each $1\leq i \leq \ell$, we have $\trdeg_{\oK}\oK(\mathfrak{f}(\alpha_1),\dots,\mathfrak{f}(\alpha_\ell))=\ell$ if and only if the CM fields $K(\alpha_1),\dots,K(\alpha_\ell)$ are distinct.
\end{corollary}

\begin{proof}
    By Theorem \ref{Thm:Rationality_at_CM_points}, we have
    $
        \oK(\mathfrak{f}(\alpha_1),\dots,\mathfrak{f}(\alpha_\ell))=\oK\Big(\left(\frac{\omega_{\alpha_1}}{\tilde{\pi}}\right)^{k},\dots,\left(\frac{\omega_{\alpha_\ell}}{\tilde{\pi}}\right)^{k}\Big).
    $
    Suppose that $K(\alpha_1),\dots,K(\alpha_\ell)$ are pairwise distinct. Then by \cite[Thm.~2.2.2]{Cha12a}, the elements  $\frac{\omega_{\alpha_1}}{\tilde{\pi}},\dots,\frac{\omega_{\alpha_\ell}}{\tilde{\pi}}$ are algebraically independent over $\overline{K}$, and thus $\trdeg_{\oK}\oK(F(\alpha_1),\dots,F(\alpha_\ell))=\ell$. On the other hand, suppose that there exist $1\leq i,j\leq\ell$ so that $K(\alpha_i)=K(\alpha_j)$. Then the Drinfeld modules $\phi^{\omega_{\alpha_i}\Lambda_{\alpha_i}}$ and $\phi^{\omega_{\alpha_j}\Lambda_{\alpha_j}}$ are isogenous. In particular, $\omega_{\alpha_i}/\omega_{\alpha_j}\in\overline{K}$, and thus $\trdeg_{\oK}\oK(F(\alpha_1),\dots,F(\alpha_\ell))<\ell$. The desired result now follows.
\end{proof}

\begin{proof}[{Proof of Theorem \ref{T:E}}] Since for any given CM point $z_0\in \Omega$, by Theorem \ref{T:CM}, we can define a field over $M_{z_0}$ over $\widehat{K_{\infty}^{\text{nr}}}$ and an extension $\psi_{z_0}$ of $\sigma$, the theorem follows from Theorem \ref{Thm:Rationality_at_CM_points} and Corollary \ref{Cor:Rationality_at_CM_points2}.
\end{proof}

\begin{appendices}
	\section{Proof of Theorem \ref{T:CM}}
	In this section, our goal is to prove Theorem \ref{T:CM}. In particular, for any given CM point $z_0\in \Omega$, we precisely determine $\psi_{z_0}$, an extension of $\sigma$, on a certain quadratic field $M_{z_0}$ over $\widehat{K_{\infty}^{\text{nr}}}$. Our proof will be separated according to the parity of $p$.
	
	\subsection{Even characteristic case} We denote the $\infty$-adic valuation by $\val_{\infty}$ normalized so that $\val_{\infty}(\theta)=-1$. In what follows, we list several facts with some brief explanations.
	\begin{itemize}
		\item[(i)] Let $z_0\in K_{\infty}(\widetilde{\mathfrak{c}})\setminus K_{\infty}\subset \Omega$ where $\widetilde{\mathfrak{c}}^2+\widetilde{\mathfrak{c}}+\tilde{B}=0$ for some $\tilde{B}\in K_{\infty}^{\times}$. By Newton polygon method and the proof of \cite[Chap. III, (2.4)]{FV02} (see also \cite[\S2]{Tho05}), one can indeed choose an element $B\in K_{\infty}^{\times}$ so that $K_{\infty}(\widetilde{\mathfrak{c}})=K_{\infty}(\mathfrak{c})$ where $\mathfrak{c}^2+\mathfrak{c}+B=0$ and $\val_{\infty}(\mathfrak{c})\in \mathbb{Q}\setminus \mathbb{Z}$. Set $M_{z_0}:=\widehat{K_{\infty}^{\text{nr}}}(\mathfrak{c})$. Then the map $\psi_{z_0}:M_{z_0}\to M_{z_0}$ given by 
		\[
		\psi_{z_0}(a+b\mathfrak{c}):=\sigma(a)+\sigma(b)(\mathfrak{c}+1), \ \ a, b \in \widehat{K_{\infty}^{\text{nr}}}
		\]
		is an isometry and hence is continuous. Furthermore, since $K_\infty(z_0)=K_\infty(\mathfrak{c})$ is a ramified extension over $K_\infty$ and $\widehat{K_{\infty}^{\text{nr}}}$ is an unramified extension over $K_\infty$, they are linearly disjoint over $K_\infty$. Hence $\psi_{z_0}$ is a field automorphism on $M_{z_0}$.
		\item[(ii)] There exists $\epsilon\in \mathbb{F}_{2^{n}}\setminus\mathbb{F}_{2^{n-1}}$ such that under the trace map, we have
		\[
		\Tr_{\mathbb{F}_{2^{n}}/\mathbb{F}_{2}}(\epsilon)=\epsilon+\epsilon^2+\dots+\epsilon^{2^{n-1}}=1.
		\]
		Indeed, since $\Tr_{\mathbb{F}_{2^{n}}/\mathbb{F}_{2}}(x)$ is a polynomial of $x$ in degree $2^{n-1}$ having at most $2^{n-1}-1$ distinct non-zero solutions, the desired existence follows from $|\mathbb{F}_{2^{n}}\setminus\mathbb{F}_{2^{n-1}}|=2^{n-1}>2^{n-1}-1$.
		\item[(iii)] Set $q:=2^n$ and let $\alpha$ be a root of $x^2+x+\epsilon\in \mathbb{F}_q[X]$. By (ii), we obtain 
		\[
		\alpha^{q}=(\alpha^2)^{2^{n-1}}=(\alpha+\epsilon)^{2^{n-1}}=(\alpha+\epsilon+\epsilon^2)^{2^{n-2}}=\cdots=\alpha+\epsilon+\epsilon^2+\dots+\epsilon^{2^{n-1}}=\alpha +1.
		\]
		Thus $\alpha$ is a root of $x^q+x+1$. Here, we remark that each root of $x^q+x+1$ lies in $\mathbb{F}_{q^2}\setminus\mathbb{F}_q$. Indeed, let $\beta$ be a root of $x^q+x+1$. Note that $\beta^q=\beta+1\neq\beta$ implies $\beta\not\in\mathbb{F}_q$. On the other hand, $\beta^{q^2}=(\beta+1)^q=\beta$ implies $\beta\in\mathbb{F}_{q^2}$.  This indeed yields that $\alpha$ is quadratic over $\mathbb{F}_q$ with its minimal polynomial $x^2+x+\epsilon$. Moreover,
		\[
		(\alpha+\mathfrak{c})^2+(\alpha+\mathfrak{c})+B+1=0.
		\]
		Furthermore, using Galois theory, one can see that $\alpha+\mathfrak{c}$ is quadratic over $K_\infty$.
		\item[(iv)] The fixed field of $\psi_{z_0}$ in $M_{z_0}$ is $K_{\infty}(\alpha+\mathfrak{c})$. Indeed, if $\psi_{z_0}(a+b\mathfrak{c})=a+b\mathfrak{c}$ for some $a,b\in  \widehat{K_{\infty}^{\text{nr}}}$, then $b=\sum_{i\geq i_0}b_i\theta^{-i}\in K_{\infty}$ and $a=\sum_{i\geq i_0}b_i(\alpha +\mathfrak{b})\theta^{-i}+\mathfrak{h}$ for some $\mathfrak{b}\in \mathbb{F}_q$ and $\mathfrak{h}\in K_{\infty}$. This shows that $a+b\mathfrak{c}\in K_{\infty}(\alpha+\mathfrak{c})$. On the other hand, for any $c+d(\alpha+\mathfrak{c})\in K_{\infty}(\alpha+\mathfrak{c})$, by (iii), we have
		\[
		\psi_{z_0}(c+d(\alpha+\mathfrak{c}))=c+d\alpha^q+d+d\mathfrak{c}=c+d(\alpha+\mathfrak{c}).
		\]
	\end{itemize}
	The above discussion now finishes the proof in the even characteristic case. Observe that it also implies that $\Omega^{\psi_{z_0}}(M_{z_0})=M_{z_0}\setminus K_{\infty}(\alpha+\mathfrak{c})$.
	
	\subsection{Odd characteristic case} By Kummer theory \cite[Chap. VI. \S8]{Lan02}, we see that the field $\widehat{K_{\infty}^{\text{nr}}}(1/\sqrt{\theta})$ is the unique quadratic extension over $\widehat{K_{\infty}^{\text{nr}}}$, and hence contains all the CM points in $\Omega$. Let $\xi$ be a non-zero root of $x^q+x$ in $\overline{\mathbb{F}}_q$. Then $\xi^{q^2}=(-\xi)^q=\xi$ implies that $\xi\in\mathbb{F}_{q^2}$. By Kummer theory again, there are three quadratic extensions over $K_\infty$: $K_\infty(\xi),~K_\infty(1/\sqrt{\theta})$, and  $K_\infty(\xi/\sqrt{\theta})$. Let $z_0\in\Omega$ be a CM point so that $K(z_0)$ generates a ramified extension over $K$. Then only one of the following holds:
	\begin{enumerate}
		\item[(I)] There is an embedding from $K(z_0)$ into $K_\infty(1/\sqrt{\theta})\subset\widehat{K_{\infty}^{\text{nr}}}(1/\sqrt{\theta})$.
		\item[(II)] There is an embedding from $K(z_0)$ into $K_\infty(\xi/\sqrt{\theta})\subset\widehat{K_{\infty}^{\text{nr}}}(1/\sqrt{\theta})$.
	\end{enumerate}
	\subsubsection{\textbf{Case (I)}} Consider the quadratic field $M_{z_0}:=\widehat{K_{\infty}^{\text{nr}}}(1/\sqrt{\theta})$ over $\widehat{K_{\infty}^{\text{nr}}}$. In our first case, $z_0\in K_\infty(1/\sqrt{\theta})$.	We define
	\[
	\psi_{z_0}\left(a+\frac{b}{\sqrt{\theta}}\right):=\sigma(a)-\frac{\sigma(b)}{\sqrt{\theta}}, \ \ a,b \in \widehat{K_{\infty}^{\text{nr}}}.
	\]
	Note that $\psi_{z_0}$ is a continuous field automorphism on $M_{z_0}$ with its fixed field $K_{\infty}(\xi/\sqrt{\theta})$ and hence  $\Omega^{\psi_{z_0}}(M_{z_0})=M_{z_0}\setminus K_{\infty}(\xi/\sqrt{\theta})$.
	
	\subsubsection{\textbf{Case (II)}} In this case, $z_0\in K_\infty(\xi/\sqrt{\theta})$ and we define $M_{z_0}$ as in the Case I. We set
	\[
\psi_{z_0}\left(a+\frac{b}{\sqrt{\theta}}\right):=\sigma(a)+\frac{\sigma(b)}{\sqrt{\theta}}, \ \ a,b \in \widehat{K_{\infty}^{\text{nr}}}.
\]
		Note that $\psi_{z_0}$ is a continuous field automorphism on $M_{z_0}$ with its fixed field $K_{\infty}(1/\sqrt{\theta})$ and hence  $\Omega^{\psi_{z_0}}(M_{z_0})=M_{z_0}\setminus K_{\infty}(1/\sqrt{\theta})$. Thus we complete the proof of the theorem.
\end{appendices}


\begin{thebibliography}{99}
	
\bibitem[Bas14]{Bas14} D. J. Basson, \textit{On the coefficients of Drinfeld modular forms of higher rank}, Ph.D. thesis, Stellenbosch University, \url{http://scholar.sun.ac.za/handle/10019.1/86387}, 2014.	
		
	
\bibitem[BBP21]{BBP21} D. J. Basson, F. Breuer, and R. Pink, \textit{Drinfeld modular forms of arbitrary rank}, available at \url{https://carmamaths.org/breuer/DrinfeldModularformsInArbitraryRank-v4.pdf}, 2021.



\bibitem[BoPe08]{BP08} V. Bosser and F. Pellarin, \textit{Hyperdifferential properties of Drinfeld quasi-modular forms}, Int. Math. Res. Not. IMRN \textbf{2008} (2008), Art. ID rnn032, 56 pp.


\bibitem[Boc02]{Boc02} G. B\"{o}ckle, \textit{An Eichler-Shimura isomorphism over function fields between Drinfeld modular forms and cohomology classes of crystals}, available at \url{https://typo.iwr.uni-heidelberg.de/fileadmin/groups/arithgeo/templates/data/Gebhard_Boeckle/EiShNew.pdf}, 2002.

\bibitem[BCPW22]{BCPW22} D. Brownawell, C.-Y. Chang, M. Papanikolas, and F.-T. Wei, \textit{Function field analogue of Shimura’s conjecture on period symbols}, arXiv:2203.09131, 2022.


\bibitem[BGHZ08]{BGHZ08}J. H. Bruinier, G. van der Geer, G\"{u}nter Harder, and D. Zagier, \textit{The 1-2-3 of modular forms}, Universitext, Springer-Verlag, Berlin, Lectures from the Summer School on Modular Forms and their
Applications held in Nordfjordeid, June 2004, Edited by Kristian Ranestad, 2008.

\bibitem[Cha12a]{Cha12a} C.-Y. Chang, \textit{Special values of Drinfeld modular forms and algebraic independence}, Math. Ann. \textbf{352} (2012), 189-204.

\bibitem[Cha12b]{Cha12b}  C.-Y. Chang, \textit{Transcendence of special values of quasi-modular forms}, Forum Math. \textbf{24} (2012), 539--551.

\bibitem[CL19]{CL19} Y.-J. Choie and M.-H. Lee, \textit{Jacobi-Like Forms, Pseudodifferential Operators, and Quasimodular Forms}, Monographs in Mathematics, Springer, Cham (2019).

\bibitem[DG20]{DG20} L. Demangos and T.M.Gendron, \textit{Modular invariant of rank 1 Drinfeld modules and class field generation}, J. Number Theory \textbf{237}, 40--66 (2020).

\bibitem[Dri74]{Dri74}  V. G. Drinfeld, \textit{Elliptic modules}, Math. Sb. (N.S.) 94 (1974), 594–627, 656, Engl. transl.: Math. USSR-Sb. 23 (1976), 561–592.

\bibitem[FV02]{FV02} I.B. Fesenko and S.V. Vostokov, \textit{Local fields and their extensions}, Transl. Math. Monogr., vol. 121, AMS (2002).	

\bibitem[Fra11]{Fra11} C. Franc, \textit{Nearly rigid analytic modular forms and their values at CM points}, Ph.D. thesis, McGill University, 2011.

\bibitem[FvdP04]{FvdP04} J. Fresnel and M. van der Put, \textit{Rigid Analytic Geometry and its Applications}, Birkh\"{a}user, Boston (2004).

\bibitem[Gek83]{Gek83} %
E.-U. Gekeler, \textit{Zur Arithmetik von Drinfeld Moduln}, Math. Ann. \textbf{262} (1983) 167–182.

\bibitem[Gek84]{Gek84} %
E.-U. Gekeler, \textit{Modulare Einheiten f\"{u}r Funktionenkorper}, J. Reine Angew. Math., 348 (1984), 94–115.

\bibitem[Gek85]{Gek85} %
E.-U. Gekeler, \textit{A product expansion for the discriminant function of Drinfeld modules of rank two}, J. Number Theory \textbf{21}, 135--140 (1985).

\bibitem[Gek86]{Gek86} %
E.-U. Gekeler, \textit{Drinfeld Modular Curves}, Springer-Verlag Lecture Notes in Mathematics 1231, Springer (1986).

\bibitem[Gek88]{Gek88} %
E.-U. Gekeler, \textit{On the coefficients of Drinfeld modular forms}, Invent. Math. \textbf{93}, No: 3, (1988), 667--700.	

\bibitem[Gek89]{Gek89} %
E.-U. Gekeler, \textit{Quasi-periodic functions and Drinfeld modular forms}, Compos. Math. 69 (1989), no. 3, 277–293.

\bibitem[GvdP80]{GvdP80} L. Gerritzen and M. van der Put, \textit{Schottky groups and Mumford curves}, volume 817 of Lecture Notes in Mathematics, Springer, Berlin, 1980.

\bibitem[Gos80]{Gos80}D. Goss, \textit{$\pi$-adic Eisenstein series for function fields}, Compos. Math. \textbf{40} (1980),  3--38.

\bibitem[Gos96]{Gos96} %
D. Goss, \textit{Basic Structures of Function Field Arithmetic}, Springer-Verlag, Berlin, 1996.

\bibitem[Ham03]{Ham03} %
Y. Hamahata, \textit{The values of J-invariants for Drinfeld modules}, Manuscripta Math. \textbf{112}, 93–108 (2003).

\bibitem[Hay79]{Hay79} D. Hayes, \textit{Explicit class field theory in global function fields}, Studies in algebra and number theory, 1979, pp. 173–217.



\bibitem[Hid13]{Hid13} H. Hida, \textit{Elliptic Curves and Arithmetic Invariants}, Springer Monographs in Mathematics Springer, New York (2013).


\bibitem[KZ95]{KZ95} M. Kaneko and D. Zagier, \textit{A Generalized Jacobi Theta Function and Quasimodular Forms}, In The Moduli Space of Curves, edited by R. H. Dijkgraaf, C. Faber, and G. van der Geer,  Progress in Mathematics vol. 129. Birkh\"{a}user Boston, 1995.

\bibitem[Lan02]{Lan02} S. Lang, \textit{Algebra}, Graduate Texts in Mathematics, Vol. 211, Springer, New York, 2002.


\bibitem[Pap23]{Pap23}
M. Papikian, \textit{Drinfeld modules}, Grad. Texts in Mathematics, vol. 296. Springer (2023).

\bibitem[Pel21]{Pel21} F. Pellarin, \textit{The analytic theory of vectorial Drinfeld modular forms}, arXiv:1910.12743, 2019.

\bibitem[Ros02]{Ros02}
M. Rosen, \textit{Number Theory in Function Fields}, Grad. Texts in Mathematics, vol. 210. Springer, New York (2002).

\bibitem[SS91]{SS91} P. Schneider and U. Stuhler, \textit{The cohomology of $p$-adic symmetric spaces}, Invent. Math. \textbf{105} (1991), 47--122.

\bibitem[Shi71]{Shi71} G. Shimura, \textit{Introduction to the arithmetic theory of automorphic functions}, Iwanami Shoten and Princeton Univ. Press, 1971.

\bibitem[Shi75a]{Shi75a} G. Shimura, \textit{On the holomorphy of certain Dirichlet series}, Proc. Lond. Math. Soc. (3) \textbf{31} (1975), 79--98.

\bibitem[Shi75b]{Shi75b} G. Shimura, \textit{On some arithmetic properties of modular forms of one and several variables}, Ann. Math. \textbf{102} (1975), 491--515. 

\bibitem[Shi77]{Shi77} G. Shimura, \textit{On the derivatives of theta functions and modular forms}, Duke Math. J., \textbf{44} (1977), 365--387.

\bibitem[Shi87]{Shi87} G. Shimura, \textit{Nearly holomorphic functions on hermitian symmetric spaces}, Math. Ann. \textbf{278}, 1–28 (1987).

\bibitem[Shi98]{Shi98} G. Shimura, \textit{Abelian Varieties with Complex Multiplication and Modular Functions}, Princeton Mathematical Series \textbf{46}, Princeton University Press, 1998.

\bibitem[Shi00]{Shi00} G. Shimura, \textit{Arithmeticity in the theory of automorphic forms}, Math. Surv. Monog. vol. 82, Amer. Math. Soc. 2000.

\bibitem[Shi07]{Shi07} G. Shimura, \textit{Elementary Dirichlet Series and modular forms}, Springer
Monographs in Mathematics, Springer, New York, 2007.

\bibitem[Thi92]{Thi92} A. Thiery, \textit{Ind\'{e}pendance alg\'{e}brique des p\'{e}riodes et quasi-p\'{e}riodes d’un module de Drinfeld}, in: The Arithmetic of Function Fields (Proceedings of the Workshop at the Ohio State University, June 17–26, 1991), pp. 265--284, edited by D. Goss,
D. R. Hayes and M. I. Rosen, Walter de Gruyter, Berlin, 1992.

\bibitem[Tho05]{Tho05} L. Thomas, \textit{Ramification groups in Artin-Schreier-Witt extensions}, J. Th\'{e}or. Nombres Bordeaux, \textbf{17}(2005), 689--720.	

\bibitem[US98]{US98} Y. Uchino and T. Satoh, \textit{Function field modular forms and higher derivations}, Math. Ann. \textbf{311}, no. 3 (1998),  439--466.

\bibitem[Yu86]{Yu86} J. Yu, \textit{Transcendence and Drinfeld modules}, Invent. Math. \textbf{83}, 507-517 (1986).


\end{thebibliography}
\end{document}